\documentclass[10pt]{amsart}
\usepackage{titletoc}
\usepackage[utf8]{inputenc}
\usepackage[T1]{fontenc}
\usepackage{indentfirst}
\usepackage[margin=1in]{geometry} 
\usepackage{lipsum}
\usepackage{amsmath,amscd,amsthm,amssymb,mathrsfs,stmaryrd,color,tikz-cd,graphicx,tikz,mathtools,subfigure}
\tikzcdset{arrow style=tikz,diagrams={>=stealth}}
\usepackage{enumerate}
\usepackage[bookmarks, bookmarksdepth=1, colorlinks=true, linkcolor=blue, citecolor=red, urlcolor=blue,pagebackref]{hyperref}
\usepackage{subfiles}

\usetikzlibrary{matrix,arrows,decorations.pathmorphing,cd,topaths,calc,graphs}
\usepackage[all,cmtip]{xy}

\usepackage{relsize}

\newcommand{\sgn}{\operatorname{sgn}}

\renewcommand{\Re}{\operatorname{Re}}
\renewcommand{\Im}{\operatorname{Im}}

\makeatletter
\newcommand{\subalign}[1]{%
  \vcenter{%
    \Let@ \restore@math@cr \default@tag
    \baselineskip\fontdimen10 \scriptfont\tw@
    \advance\baselineskip\fontdimen12 \scriptfont\tw@
    \lineskip\thr@@\fontdimen8 \scriptfont\thr@@
    \lineskiplimit\lineskip
    \ialign{\hfil$\m@th\scriptstyle##$&$\m@th\scriptstyle{}##$\crcr
      #1\crcr
    }%
  }
}
\makeatother

\newcommand{\BA}{\mathbb{A}}

\newcommand{\BC}{\mathbb{C}}
\newcommand{\BD}{\mathbb{D}}

\newcommand{\BH}{\mathbb{H}}

\newcommand{\BQ}{\mathbb{Q}}
\newcommand{\BR}{\mathbb{R}}

\newcommand{\BZ}{\mathbb{Z}}

\newcommand{\cA}{\mathcal{A}}

\newcommand{\cB}{\mathcal{B}}

\newcommand{\sB}{\mathscr{B}}

\newcommand{\bD}{\mathbf{D}}
\newcommand{\cD}{\mathcal{D}}

\newcommand{\cF}{\mathcal{F}}

\newcommand{\sF}{\mathscr{F}}
\newcommand{\bG}{\mathbf{G}}

\newcommand{\bH}{\mathbf{H}}
\newcommand{\cH}{\mathcal{H}}

\newcommand{\bI}{\mathbf{I}}
\newcommand{\cI}{\mathcal{I}}

\newcommand{\bJ}{\mathbf{J}}

\newcommand{\cO}{\mathcal{O}}

\newcommand{\sP}{\mathscr{P}}

\newcommand{\sQ}{\mathscr{Q}}

\newcommand{\cS}{\mathcal{S}}

\newcommand{\sS}{\mathscr{S}}

\newcommand{\cT}{\mathcal{T}}

\newcommand{\sT}{\mathscr{T}}

\newcommand{\fa}{\mathfrak{a}}

\newcommand{\bc}{\mathbf{c}}

\newcommand{\fk}{\mathfrak{k}}

\newcommand{\fn}{\mathfrak{n}}

\newcommand{\bt}{\mathbf{t}}

\newcommand{\bz}{\mathbf{z}}

\newcommand{\vertiii}[1]{{\left\vert\kern-0.25ex\left\vert\kern-0.25ex\left\vert #1 
    \right\vert\kern-0.25ex\right\vert\kern-0.25ex\right\vert}}

\newcommand{\rom}[1]{\uppercase\expandafter{\romannumeral #1\relax}}
\renewcommand{\emptyset}{\varnothing}

\renewcommand{\tilde}[1]{\widetilde{#1}}

\DeclareMathOperator{\supp}{Supp}

\DeclareMathOperator{\diag}{diag}

\newcommand{\SL}{\operatorname{SL}}
\newcommand{\SO}{\operatorname{SO}}
\newcommand{\SU}{\operatorname{SU}}

\makeatletter
\newcommand{\colim@}[2]{%
  \vtop{\m@th\ialign{##\cr
    \hfil$#1\operator@font lim$\hfil\cr
    \noalign{\nointerlineskip\kern1.5\ex@}#2\cr
    \noalign{\nointerlineskip\kern-\ex@}\cr}}%
}
\newcommand{\Lim}{%
  \mathop{\mathpalette\varlim@{\leftarrowfill@\scriptscriptstyle}}\nmlimits@
}

\newcommand{\colim}{%
  \mathop{\mathpalette\varlim@{\rightarrowfill@\scriptscriptstyle}}\nmlimits@
}

\makeatother

\newtheorem{theorem}{Theorem}[section]

\newtheorem{proposition}[theorem]{Proposition}

\newtheorem{corollary}[theorem]{Corollary}

\newtheorem{lemma}[theorem]{Lemma}

\theoremstyle{definition}

\theoremstyle{remark}

\usepackage[new]{old-arrows}

\usepackage{amssymb}

\usepackage[utf8]{inputenc}

\title[Restrictions of Maass forms on $\mathrm{SL}(2,\mathbb{C})$ ]{Restrictions of Maass forms on $\mathrm{SL}(2,\mathbb{C})$ to hyperbolic surfaces and geodesic tubes}
\author{Jiaqi Hou}
\address{Department of Mathematics, Louisiana State University, Baton Rouge, LA 70803, USA}
\email{jhou7@lsu.edu}
\date{}

\usepackage{hyperref}
\numberwithin{equation}{section}

\begin{document}

\begin{abstract}
    Let $\psi$ be an $L^2$-normalized Hecke-Maass form with a large spectral parameter $\lambda>0$ on a compact arithmetic congruence hyperbolic 3-manifold $X = \Gamma \backslash \mathrm{SL}(2,\BC)/\mathrm{SU}(2)$, and let $Y$ be a totally geodesic surface in $X$ with bounded diameter. 
    
    The local $L^2$-bound for the restriction of $\psi$ to $Y$ is $\|\psi|_Y \|_{L^2(Y)}\ll \lambda^{1/4}$ by Burq, G{\'e}rard, and Tzvetkov. 
    We apply the method of arithmetic amplification developed by Iwaniec and Sarnak to obtain a power saving over the local bound.
    The new feature in the proof is that we establish two different estimates for the integrals of $\psi|_Y$ against geodesic beams over $Y$ via two amplification arguments.
    Combining these estimates, we can improve the local bound for generalized Fourier coefficients of $\psi|_Y$ against eigenfunctions on $Y$ with spectral parameters near $\lambda$.
    
    We also apply the amplification method to obtain a power saving over the trivial bound $O(1)$ for $L^2$-norms of $\psi$ restricted to $\lambda^{-1/2}$-neighborhoods of unit-length geodesic segments. Consequently, by applying a result of Blair and Sogge, we obtain power savings over the local $L^p$-bounds of $\psi$ by Sogge for $2<p<4$ from our improved bound for the Kakeya-Nikodym norm.
\end{abstract}

\maketitle
{
  \hypersetup{linkcolor=black}
  \tableofcontents
}

\section{Introduction}

Let $M$ be a compact Riemannian manifold of dimension $n\geq2$, and let $\Delta$ be the Laplace-Beltrami operator on $M$.
If $\psi$ is a Laplace eigenfunction on $M$
satisfying  $\Delta\psi + \lambda^2\psi = 0$ with $\lambda> 0$.
It is an interesting problem to study the concentration of $\psi$ as $\lambda\to\infty$.
We assume that $\|\psi \|_2 =1$.
A classical result on $\| \psi\|_p$ is due to Sogge  \cite{sogge1988concerning} (see also Avakumovi{\'c} \cite{avakumovic1956eigenfunktionen}  and Levitan \cite{levitan1952asymptotic} when $p=\infty$), 
and is
\begin{align}\label{Sogge bound}
    \| \psi\|_p \ll \lambda^{\delta(p,n)}, 
\end{align}
where $\delta(p,n)$ is given by
\begin{equation*}
            \delta(p,n) = \begin{dcases}
                \frac{n-1}{4} - \frac{n-1}{2p} \;&\text{ for }\;2\leq p\leq \frac{2(n+1)}{n-1},\\
                \frac{n-1}{2} - \frac{n}{p} \;\;\;\;&\text{ for }\;\frac{2(n+1)}{n-1}\leq p\leq\infty.
            \end{dcases}
\end{equation*}
The notation $A\ll B$ means that there is a positive constant $C$ such that $|A| \leq C B$. (See Section \ref{notation} for the notation in this paper.)
When $M$ is the round sphere, the estimates \eqref{Sogge bound} are saturated by the zonal spherical harmonics for $p\geq \frac{2(n+1)}{n-1}$ and by the highest weight spherical harmonics for $2< p\leq \frac{2(n+1)}{n-1}$.

Nevertheless, it is expected that \eqref{Sogge bound} can be improved under extra geometric assumptions on $M$.
For instance, there are log improvements if $M$ has nonpositive sectional curvature.
We assume the manifold $M$ satisfies the nonpositive curvature condition in this paragraph.
B\'erard \cite{berard1977wave} obtained a $\log\lambda$ improvement for the error term in the local Weyl law, which implies that
\begin{align*}
    \| \psi \|_\infty \ll \frac{\lambda^{\frac{n-1}{2}}}{\sqrt{\log\lambda}}.
\end{align*}
Hassell and Tacy \cite{hassell2015improvement} obtained further improvements by showing that
\begin{align*}
    \| \psi \|_p \ll \frac{\lambda^{\delta(p,n)}}{\sqrt{\log\lambda}}=\frac{\lambda^{\frac{n-1}{2}-\frac{n}{p}}}{\sqrt{\log\lambda}},\quad \text{ if }p>\frac{2(n+1)}{n-1},
\end{align*}
where the implied constant blows up as $p$ goes to $\frac{2(n+1)}{n-1} $. 
For $2<p<\frac{2(n+1)}{n-1}$, improvement involving powers of $\log\lambda$ were proved by Blair and Sogge \cite{blair2018concerning}, where the powers depend on $n$ and $p$. Blair and Sogge obtained the results by proving improved bounds for Kakeya-Nikodym norms which measure the $L^2$-concentration of eigenfunctions on $\lambda^{-1/2}$-tubes of unit length geodesics. We will discuss the Kakeya-Nikodym norms and their improved bounds in more detail in Section \ref{sec: KN norm}. In \cite{blair2019logarithmic}, Blair and Sogge established a gain of an inverse power of $\log\lambda$ at the critical exponent $p=\frac{2(n+1)}{n-1}$.

The $L^\infty$-norm problems can also be studied using arithmetic techniques when the manifold $M$ and the eigenfunction $\psi$ are equipped with additional arithmetic structures.
We let $M$ be a compact congruence arithmetic hyperbolic surface arising from an indefinite quaternion division algebra over $\BQ$. Let $\psi$ be a Hecke-Maass form which is a joint eigenfunction of the Laplace-Beltrami operator and the unramified Hecke operators on $M$. Iwaniec and Sarnak \cite{iwaniec1995norms} showed that the bound $\| \psi \|_\infty\ll\lambda^{1/2}$ given by \eqref{Sogge bound} can be strengthened by a power to $\|\psi\|_\infty\ll_\epsilon\lambda^{5/12+\epsilon}.$
Their proof uses the technique known as arithmetic amplification. This approach has since been used by many authors to bound $L^\infty$-norms of Hecke-Maass forms on other groups. For instance, see \cite{blomer2019sup,blomer2016subconvexity}. The main results in this paper also rely on the amplification method.
\medskip

We are also interested in studying the concentration of $\psi$ along submanifolds.
The standard $L^p$-restriction estimates were established by Burq, G{\'e}rard and Tzvetkov \cite{burq2007restrictions} and Hu \cite{hu2009p}. 
We assume that $\Sigma \subset M$ is a closed submanifold and $\dim \Sigma = k$.
If $k=n-1$, they showed that
\begin{align}\label{BGT trivial bound}
    \| \psi|_\Sigma \|_{L^p(\Sigma)} \ll \lambda^{\rho_{n-1}(p,n)},
\end{align}
where $\rho_{n-1}(p,n)$ is given by
\begin{equation*}
    \rho_{n-1}(p,n) = \begin{dcases}
        \frac{n-1}{4} - \frac{n-2}{2p}\quad &\text{ if }2\leq p\leq \frac{ 2n}{n-1},\\
        \frac{n-1}{2} - \frac{n-1}{p} \quad &\text{ if }\frac{2n}{n-1}\leq p \leq\infty.
    \end{dcases}
\end{equation*}
If $n\geq3$ and $k=n-2$, then we have
\begin{align}\label{eq: BGT codim 2, p=2}
    \| \psi|_\Sigma \|_{L^2(\Sigma)} \ll \lambda^{1/2}\sqrt{\log\lambda},
\end{align}
and
\begin{align}\label{eq: BGT codim 2, p>2}
    \| \psi|_\Sigma \|_{L^p(\Sigma)} \ll \lambda^{\rho_{n-2}(p,n)},\quad\text{ where }\rho_{n-2}(p,n)=\frac{n-1}{2}-\frac{n-2}{p},\quad 2<p\leq\infty.
\end{align}
If $n\geq4$ and $k\leq n-3$, then
\begin{align}\label{eq: BGT codim>2}
    \| \psi|_\Sigma \|_{L^p(\Sigma)} \ll \lambda^{\rho_{k}(p,n)},\quad\text{ where }\rho_{k}(p,n)=\frac{n-1}{2}-\frac{k}{p},\quad 2\leq p\leq\infty.
\end{align}
By Chen and Sogge \cite[Theorem 1.1]{chen2014few} and Wang and Zhang \cite[Theorem 2]{WANG2021107835}, the $\sqrt{\log\lambda}$ loss in the endpoint estimates \eqref{eq: BGT codim 2, p=2} can be removed provided that $\Sigma$ is a totally geodesic submanifold of codimension 2.
In \cite[\S 6.2]{burq2007restrictions}, Burq, G{\'e}rard and Tzvetkov gave examples to show that the estimates \eqref{BGT trivial bound}--\eqref{eq: BGT codim>2} are sharp, except for the $\sqrt{\log\lambda}$ loss in \eqref{eq: BGT codim 2, p=2}.
If $M$ is the round sphere and $\Sigma$ is any submanifold containing a pole of the sphere, the estimates \eqref{eq: BGT codim 2, p=2}--\eqref{eq: BGT codim>2}, and \eqref{BGT trivial bound} for $\frac{2n}{n-1}\leq p \leq\infty$ are saturated by the zonal spherical harmonics.
If $M$ is the round sphere and $\Sigma$ contains a geodesic, then \eqref{BGT trivial bound}, with $2\leq p\leq\frac{2n}{n-1}$, are saturated by the highest weight spherical harmonics.
The estimates \eqref{BGT trivial bound} can be improved if $\Sigma$ is a hypersurface and has curvature. Indeed \cite[Theorem 1.4]{hu2009p} showed that if $k=n-1$ and the second fundamental form of $\Sigma$ is (positive or negative) definite, then, for $2\leq p\leq\frac{2n}{n-1}$,
\begin{align*}
    \| \psi|_\Sigma \|_{L^p(\Sigma)} \ll \lambda^{\frac{n-1}{3}-\frac{2n-3}{3p}}.
\end{align*}

Like the $L^p$-norm problems, we expect that the estimates \eqref{BGT trivial bound}--\eqref{eq: BGT codim>2} may be strengthened under curvature assumptions on $M$. If $M$ has nonpositive curvature, Chen \cite{chen2015improvement} proved a $(\log\lambda)^{-1/2}$ gain for $p>\frac{2n}{n-1}$ if $k=n-1$ and for $p>2$ if $k\leq n-2$.
We let $\Pi$ denote the space of all unit-length geodesic segments on $M$.
There has been considerable work towards improving \eqref{BGT trivial bound} for geodesic restrictions on surfaces with nonpositive curvature. 
Assume that $M$ is a surface with nonpositive curvature.
Xi and Zhang \cite{xi2017improved} and Blair \cite{Blair2018} dealt with the endpoint $p=4$ and proved
\begin{align*}
    \sup_{\gamma\in\Pi} \| \psi|_\gamma \|_{L^4(\gamma)}\ll \lambda^{1/4} (\log\lambda)^{-1/4}.
\end{align*}
Blair and Sogge \cite{blair2018concerning} obtained a $(\log\lambda)^{-1/4}$ gain for $p=2$.
For geodesic restrictions on hyperbolic 3-manifolds, Zhang \cite{ZHANG20174642} improved the endpoint estimate \eqref{eq: BGT codim 2, p=2}. If we let $M$ be a hyperbolic 3-manifold, Zhang obtained a log improvement:
\begin{align*}
    \sup_{\gamma\in\Pi} \| \psi|_\gamma \|_{L^2(\gamma)}\ll \lambda^{1/2} (\log\lambda)^{-1/2}.
\end{align*}

Now we let $M$ be a compact congruence arithmetic hyperbolic surface and let $\psi$ be an $L^2$-normalized Hecke-Maass form. By extending the technique of arithmetic amplification developed by Iwaniec and Sarnak \cite{iwaniec1995norms}, Marshall \cite[Theorem 1.1]{marshall2016geodesic} obtained a power saving over the bound \eqref{BGT trivial bound} for geodesic restrictions and $p=2$:
\begin{align*}
    \sup_{\gamma\in\Pi} \| \psi|_\gamma \|_{L^2(\gamma)}\ll_\epsilon \lambda^{3/14+\epsilon}.
\end{align*}
Combining the above estimate with the main theorem of \cite{BS15APDE}, Marshall \cite[Corollary 1.2]{marshall2016geodesic} also obtained an improvement over the local bound $\|\psi\|_4\ll\lambda^{1/8}$ by Sogge \eqref{Sogge bound}:
\begin{align*}
    \| \psi \|_4\ll_\epsilon \lambda^{1/8-1/56+\epsilon}.
\end{align*}
See also Humphries and Khan \cite{humphries2023lpnormboundsautomorphicforms} and Ki \cite{ki20234} on $L^4$-norms of Hecke-Maass cusp forms on the modular surface $\SL(2,\BZ)\backslash\BH^2$ via different methods.
Using Waldspurger’s formula and $L$-functions, Ali \cite{ali20222} proved a stronger $L^2$-bound for a Hecke-Maass cusp form on the modular surface $\SL(2,\BZ)\backslash\BH^2$ restricted to closed geodesics associated to a fundamental discriminant $D>0$.

Marshall extended Sogge's bound \eqref{Sogge bound} and the estimates \eqref{BGT trivial bound}--\eqref{eq: BGT codim>2} by Burq, G{\'e}rard and Tzvetkov to higher rank cases in \cite{marshall2016p}. He proved almost sharp local bounds for the $L^p$-norms for eigenfunctions of the full ring of invariant differential operators on a compact locally symmetric space, as well as their restrictions to maximal flat subspaces.
Marshall \cite{marshall2015restrictions} also applied the amplification method to obtain a power saving for the $L^2$-norms of the maximal flat restrictions for Maass forms on $\SL(3,\BR)$.

We want to extend Marshall's results \cite{marshall2016geodesic} about geodesic restrictions on hyperbolic surfaces to restriction problems of eigenfunctions on higher dimensional manifolds. In this paper, we consider two restriction problems for eigenfunctions on hyperbolic 3-manifolds.

\subsection{Hyperbolic surface restrictions}
One generalization is to estimate the $L^2$-norms for hypersurface restrictions.
We now let  $X = \Gamma\backslash \BH^3$,
where $\BH^3 = \SL(2,\BC)/\SU(2)$ is the hyperbolic 3-space,
and $\Gamma\subset \SL(2,\BC)$ is a cocompact arithmetic congruence lattice arising from a quaternion division algebra over a number field with one complex place.
Let $\psi$ be a Hecke–Maass form on $X$, that is, an eigenfunction of the Laplace-Beltrami operator $\Delta$ and the unramified Hecke operators.
We assume that $\| \psi\|_{2} = 1$. 
In this context, it is more natural to let  $\lambda$ be the spectral parameter of $\psi$, which satisfies $\Delta\psi + (\lambda^2+1)\psi = 0$.
As we are considering large eigenvalue asymptotics, we will assume that $\lambda>0$. 
The submanifold we will consider in this paper is a totally geodesic surface $Y$ in $X$.

We put the following boundedness assumption on the submanifold $Y$ to state a uniform result. 
We suppose that $\pi:\BH^3\to X$ is the covering map,
and let $\cB_{\BH^2} \subset \BH^2$ be an open ball of radius $1$ under the hyperbolic metric.
Let $\iota: \BH^2 \hookrightarrow \BH^3$ be any embedding of the hyperbolic plane.
We assume that
\begin{align}\label{boundedness condition for Y}
    Y = \pi\circ\iota(\cB_{\BH^2} ).
\end{align}
Recall that the estimate \eqref{BGT trivial bound} for the $L^2$-norm of $\psi|_Y$ gives $ \| \psi|_Y \|_{L^2(Y)}\ll \lambda^{1/4}$. We improve it as follows.

\begin{theorem}\label{main theorem in intro}
    Let $\psi$ be an $L^2$-normalized Hecke–Maass form on $X$ with spectral parameter $\lambda>0$, and let $Y$ be a totally geodesic surface in $X$ satisfying \eqref{boundedness condition for Y}.
    For any $\epsilon>0$,
    \begin{align*}
        \| \psi|_Y \|_{L^2(Y)}\ll \lambda^{1/4 -1/1220+ \epsilon },
    \end{align*}
    where the implied constant depends only on $X$ and $\epsilon$.
\end{theorem}

Theorem \ref{main theorem in intro} is one of the main results in this paper. The proof occupies Sections \ref{section of L^2 bound}--\ref{section of near spectrum} and we will give the outline of the proof in Section \ref{sec: outline of 1.1}.

\subsection{Kakeya–Nikodym norms}\label{sec: KN norm}
As before, we let $M$ be a compact Riemannian manifold of dimension $n$, and let $\psi$ be an $L^2$-normalized Laplace eigenfunction on $M$ with frequency $\lambda>0$.
As we mentioned before, the highest weight spherical harmonics on the round spheres saturate Sogge's $L^p$-norm estimates \eqref{Sogge bound} for $2<p\leq\frac{2(n+1)}{n-1}$. They have significant $L^2$-mass on tubes of width $\lambda^{-1/2}$ about unit geodesics on the equator.

Let $\Pi$ denote the space of unit length geodesic segments in $M$.
If $\gamma\in\Pi$, we denote the $\lambda^{-1/2}$-neighborhood  of $\gamma$ in $M$ by $\sT_{\lambda^{-1/2}}(\gamma)$.
Following \cite{blair2017refined}, we define the Kakeya-Nikodym norm without averages of $\psi$ to be
\begin{align*}
    \vertiii{\psi}_{KN} = \sup_{\gamma\in\Pi}  \| \psi|_{\sT_{\lambda^{-1/2}}(\gamma)}  \|_2 = \left(\sup_{\gamma\in\Pi}  \int_{\sT_{\lambda^{-1/2}}(\gamma)}|\psi(x)|^2 dx\right)^{1/2}.
\end{align*}
These norms were introduced by Sogge \cite{Sogge11Tohoku}. Since we are assuming the eigenfunctions to be $L^2$-normalized, we always have the trivial upper bound
\begin{align*}
    \vertiii{\psi}_{KN}\leq 1.
\end{align*}
The highest weight spherical harmonics on the round spheres also saturate the trivial bound $\vertiii{\psi}_{KN}\ll 1$. We mention that when $n=2$ or 3, it may be seen that \eqref{BGT trivial bound} and \eqref{eq: BGT codim 2, p=2} (after removing the log loss) imply the trivial bound $\vertiii{\psi}_{KN}\ll 1$.
Nevertheless, in higher dimensions $n>3$, the geodesic restriction estimates given by \eqref{eq: BGT codim>2} are too singular to detect concentrations of eigenfunctions near geodesics. In fact, in these dimensions, estimates \eqref{eq: BGT codim>2} are saturated by the zonal spherical harmonics rather than the highest weight spherical harmonics on the round spheres.

Blair and Sogge \cite{blair2017refined} showed that improvements on the trivial bound for $ \vertiii{\psi}_{KN}$ will give improvements on the $L^p$-bounds \eqref{Sogge bound} for small exponents $p$.
Indeed they proved that
\begin{align}\label{Blair-Sogge}
    \| \psi \|_p \ll \lambda^{\delta(p,n)} \vertiii{\psi}_{KN}^{\frac{2(n+1)}{p(n-1)} -1 },\quad \text{ if }\;\frac{2(n+2)}{n}< p < \frac{2(n+1)}{n-1}.
\end{align}
See also \cite{Bourgain,Sogge11Tohoku,BS15APDE} for the related results in dimension two.
By using Toponogov’s comparison theorem from Riemannian geometry, Blair and Sogge \cite{blair2018concerning} obtained log improvements for the Kakeya–Nikodym norms when $M$ has nonpositive curvature. Using these and \eqref{Blair-Sogge}, they were able to obtain log improvements over Sogge's $L^p$-bounds \eqref{Sogge bound} for such manifolds when $2<p<\frac{2(n+1)}{n-1}$.

Now assume that $\psi$ is an $L^2$-normalized Hecke-Maass form, with spectral parameter $\lambda>0$, on a compact arithmetic congruence hyperbolic 3-manifold $X$ as before.
We prove the following powering saving over the trivial bound.

\begin{theorem}\label{KN main theorem}
    Let $\psi$ be an $L^2$-normalized  Hecke–Maass form on $X$ with spectral parameter $\lambda>0$.
    For any $\epsilon>0$ and any geodesic segment $\gamma$ of unit length, we have
    \begin{align*}
        \| \psi|_{\sT_{\lambda^{-1/2}}(\gamma)}  \|_2 \ll \lambda^{-1/20 + \epsilon},
    \end{align*}
    where the implied constant depends only on $X$ and $\epsilon$. Therefore, $\vertiii{\psi}_{KN}\ll_\epsilon \lambda^{-1/20 + \epsilon}$.
\end{theorem}

We will prove this result in Section \ref{section of Kakeya-Nikodym norms}. 
By combining Theorem \ref{KN main theorem} with \eqref{Blair-Sogge}, we obtain the following improvements over  \eqref{Sogge bound}.
\begin{corollary}\label{cor: improved Lp norm}
     Let $\psi$ be an $L^2$-normalized  Hecke–Maass form on $X$ with spectral parameter $\lambda>0$.    We have
    \begin{align*}
        \|  \psi \|_p \ll_\epsilon  \lambda^{\delta(p,3)- \frac{1}{20}(\frac{4}{p}-1) + \epsilon}
    \end{align*}
    if  $10/3<p<4$.
\end{corollary}
Moreover, we can interpolate the above bounds with $\| \psi\|_2 =1$ to obtain power savings over the $L^p$-bounds \eqref{Sogge bound} for $X$ with $2<p<4$. If we apply the recent work \cite{gao2025sharp} by Gao, Wu and Xi, the estimates in Corollary \ref{cor: improved Lp norm} can be extended to $22/7\leq p <4$.



\subsection{Outline of the proof for Theorem \ref{main theorem in intro}}\label{sec: outline of 1.1}
If we fix a real valued function $\chi\in C_c^\infty(\BH^2)$, it suffices to estimate the $L^2$-norm of $(\chi \psi|_Y)(z) := \chi(z) \psi(g_0z)$ in $L^2(\BH^2)$, uniformly for $g_0$ in a compact subset of $\SL(2,\BC)$.
We shall do this by estimating the operator norm for 
\begin{align*}
    \phi\mapsto\langle \chi\psi,\phi \rangle:=\int_{\BH^2}\chi(z)\psi(g_0z)\overline{\phi(z)}dz
\end{align*}
for $\phi \in L^2(\BH^2)$. The proof will use Harish-Chandra transforms and Helgason transforms on hyperbolic spaces. We give their definitions and properties in Section \ref{sec: Fourier}. Given $\phi\in L^2(\BH^2)$, we denote its Helgason transform by $\tilde{\phi}$, which is a function on $\BR\times B$. Here $B = \partial\BH^2$ is the boundary of $\BH^2$.

To prove such an estimate, we construct an $\SU(2)$-bi-invariant test function $k_\lambda$ via the inverse Harish-Chandra transform for $\BH^3$ of an even real-variable function $h_\lambda$ that is non-negative on the spectral parameters and concentrates near $\pm\lambda$. We sum the test function over the arithmetic lattice to form the kernel function
\begin{align*}
    K(x,y) = \sum_{\gamma\in\Gamma} k_\lambda(x^{-1}\gamma y)
\end{align*}
on the arithmetic hyperbolic $3$-manifold $X = \Gamma\backslash\BH^3$.
Then the integral operator on $L^2(X)$ with kernel $K(x, y)$ is an approximate spectral projector onto eigenfunctions with spectral parameter $\lambda$.
By the spectral expansion for $K$, we obtain the following pretrace formula
\begin{align}\label{pretrace formula in intro}
    \sum_{i} h_\lambda(\lambda_i) \psi_i(x)\overline{\psi_i(y)}=\sum_{\gamma\in\Gamma} k_\lambda(x^{-1}\gamma y),
\end{align}
where $\psi_i$'s are Hecke-Maass forms on $X$ with spectral parameters $\lambda_i$, which form an orthonormal basis of $L^2(X)$ and $\psi$ is one of them. 

Let us first outline how to get Burq, G{\'e}rard and Tzvetkov's local bound $\|\chi\psi|_Y\|_{L^2(\BH^2)}\ll\lambda^{1/4}$ by integrating the pretrace formula.
Let $\cS(\BH^2)$ be the Harish-Chandra Schwartz space for $\BH^2$.
Given $\phi\in\cS(\BH^2)$, if we integrate \eqref{pretrace formula in intro} against $\overline{\chi\phi}\times \chi\phi$ along  $g_0\BH^2\times g_0 \BH^2$, we obtain
\begin{align}\label{integrated pretrace formula in intro}
    \sum_{i} h_\lambda(\lambda_i) |\langle \chi\psi_i,\phi\rangle|^2=\sum_{\gamma\in\Gamma} \iint_{\BH^2}\overline{\chi\phi(x_1)}\chi\phi(x_2)  k_{\lambda}(x_1^{-1}g_0^{-1}\gamma g_0x_2) dz_1 dz_2.
\end{align}
By dropping all terms on the spectral side, i.e. the left-hand side of \eqref{integrated pretrace formula in intro}, except the term $|\langle \chi\psi,\phi\rangle|^2$, it suffices to bound the integrals
\begin{align*}
    I(\lambda,\phi,g) = \iint_{\BH^2}\overline{\chi\phi(x_1)}\chi\phi(x_2) k_\lambda (x_1^{-1}gx_2) dx_1 dx_2 \quad\text{ for }g\in\SL(2,\BC).
\end{align*}
We may assume that $\Gamma$ is torsion-free and the supports of $\chi$ and $k_\lambda$ are sufficiently small so that only $I(\lambda,\phi, e)$ contributes to the geometric side (the right-hand side of \eqref{integrated pretrace formula in intro}). Then we obtain that (see Section \ref{subsec: trivial bd}) 
\begin{align}\label{eq analysis of trivial bd}
    \langle \chi\psi,\phi\rangle \ll |I(\lambda,\phi, e)|^{1/2} 
    \ll\left( \sup_{s\in\BR}| \tilde{k_\lambda}(s)| \right)^{1/2} \|\phi\|_{L^2(\BH^2)}.
\end{align}
Here $\tilde{k_\lambda}(s)$ is the Harish-Chandra transform of $k_\lambda|_{\BH^2}$ on $\BH^2$. 
Then the local bound $\|\chi\psi|_Y\|_{L^2(\BH^2)}\ll\lambda^{1/4}$ follows from the estimate $\sup_{s\in\BR}| \tilde{k_\lambda}(s)| \ll \lambda^{1/2}$ (See \eqref{ineq all t} in Proposition \ref{supnorm of klambda}).


To improve the local bound, the first main step in our proof is to refine the pointwise bound for $\tilde{k_\lambda}(s)$.
Given any $0<\epsilon'\leq 10^{-6}$ and $\beta>0$ satisfying
$\lambda^{\epsilon'}\leq\beta\leq\lambda$, Proposition \ref{supnorm of klambda} shows that if $s\notin \pm [\lambda-\beta/2,\lambda+\beta/2]$ then
\begin{align}\label{eq: restricted bd for klambda in intro}
     \tilde{k_\lambda}(s)\ll_{\epsilon'}\lambda^{1/2}\beta^{-1/2}.
\end{align}
The proof of \eqref{eq: restricted bd for klambda in intro} occupies Section \ref{section of bound away from the spectrum}.
We define ${H}_\beta$ to be the subspace of $L^2(\BH^2)$ consisting of functions with Helgason transforms supported in $\pm[\lambda-\beta,\lambda+\beta]\times B$.
Let $H_\beta^\perp$ be the orthogonal complement of $H_\beta$ in $L^2(\BH^2)$ and let $\Pi_\beta$ be the orthogonal projection from $L^2(\BH^2)$ onto $H_\beta$.
Given $\phi\in H_\beta^\perp$, 
by a similar analysis like \eqref{eq analysis of trivial bd}, \eqref{eq: restricted bd for klambda in intro} implies the following bound
\begin{align*}
    \langle \chi\psi,\phi\rangle \ll_{\epsilon'}\lambda^{1/4}\beta^{-1/4}\|\phi\|_{L^2(\BH^2)}.
\end{align*}
Therefore,
\begin{align}\label{2 trivial inequality}
   \|(1-\Pi_\beta)(\chi\psi|_Y)\|_{L^2(\BH^2)} \ll_{\epsilon'}\lambda^{1/4}\beta^{-1/4}.
\end{align}
Here we obtain a gain of $\beta^{-1/4}$ by removing the part of $\chi\psi|_Y$ with frequencies near $\pm\lambda$.
Compared to the local bound, it may be seen that the main contribution to $\| \chi \psi\|_{L^2(\BH^2)}$ should come from those frequencies near $\pm\lambda$.
We mention that the estimate \eqref{2 trivial inequality} is purely local and does not use the arithmeticity of $\Gamma$ or $\psi$.
However, it does require asymptotics for the spherical functions on $\BH^2$ and $\BH^3$, which are taken from \cite{marshall2016p}.

The second main step is to bound the $L^2$-norm of $\Pi_\beta(\chi\psi|_Y) \in H_\beta$.
We want to get a power saving over the local bound $\|\Pi_\beta(\chi\psi|_Y)\|_{L^2(\BH^2)}\ll\lambda^{1/4}$, via the method of arithmetic amplification.
First, we will give an outline of the amplification method.
We let $\cT$ be a Hecke operator and apply $\cT\cT^*$ to $K(x, y)$. The identity corresponding to \eqref{integrated pretrace formula in intro} is now
\begin{align}\label{amplification inequality in outline}
    \begin{split}
        \sum_{\gamma\in\cT\cT^*\Gamma}C(\gamma)I(\lambda,\phi,g_0^{-1}\gamma g_0) &= \sum_{i} h_\lambda(\lambda_i) \left|\langle\cT\psi_i, \chi\phi\rangle\right|^2\geq \left|\langle\cT\psi, \chi\phi\rangle\right|^2,
    \end{split}
\end{align}
where $\cT\cT^*\Gamma$ is the set of isometries appearing in $\cT\cT^*$, and $C(\gamma)$ is the coefficient of $\gamma\in \cT\cT^*\Gamma$.
Here $\phi$ can be any Harish-Chandra Schwartz function on $\BH^2$.
We choose $\cT$ so that its eigenvalue on $\psi$ is large, and it should have small eigenvalues on the remaining $\psi_i$ by the orthogonality of systems of Hecke eigenvalues.
The term “amplification” comes from this way in which $\cT$ picks out $\psi$ from the collection $\{\psi_i\}$.

To apply the inequality \eqref{amplification inequality in outline}, besides the counting problem, the amplification proof in the previous relevant results \cite{marshall2015restrictions,marshall2016geodesic} suggests that we should prove the following two estimates.
The first expected estimate is
\begin{align}\label{std arguement oscillatory integral}
    I(\lambda,\phi,g)\ll \lambda^{1/2} \| \phi \|_{L^2(\BH^2)}^2
\end{align}
for
general $g\in \SL(2,\BC)$ in a fixed compact set, which corresponds to the local bound. We define $T_{\lambda,g}$ to be the oscillatory integral operator defined by
\begin{align*}
    [T_{\lambda,g}\phi](y) = \int_{\BH^2}\chi(x)k_\lambda(x^{-1}gy)\phi(x) dx,\quad y\in\BH^2.
\end{align*}
Then \eqref{std arguement oscillatory integral} will follow from the $L^2\to L^2$ operator norm $\| T_{\lambda,g}\|_{L^2\to L^2}\ll\lambda^{1/2}$, which we expect to hold.
If we make further assumptions on the group element $g$, for example, assuming that $g$ is away from the stabilizer of $\BH^2$ in $\SL(2,\BC)$, then we need to get a better estimate than the estimate  \eqref{std arguement oscillatory integral}, which must involve a power saving.
However, this second estimate should fail for general $\phi$ if $\BH^2\cap g\BH^2\neq\emptyset$.
In the following, we briefly explain the sharpness of \eqref{std arguement oscillatory integral} by assuming $g=\left(\begin{smallmatrix}
    e^{i\theta}&\\&e^{-i\theta}
\end{smallmatrix}\right)\neq\pm I$ for some $\theta\in\BR$. 
Let $l$ be the vertical geodesic in $\BH^3$ through the origin, which is the intersection $\BH^2\cap g\BH^2$. 
If we take $\phi\in\cS(\BH^2)$ so that $\tilde{\phi}(s,b)$ is a bump function localized near $\lambda$ of width 1 in the $s$-variable and is a bump function localized near the origin of $B$ of width $\lambda^{-1/2}$ in the $b$-variable, then $\|\phi\|_{L^2(\BH^2)}\asymp \lambda^{1/4}$.
Applying the inverse Helgason transform for $\phi$ and the inverse Harish-Chandra transform for $k_\lambda$, it may be seen that $I(\lambda,\phi,g)$ can be written as an oscillatory integral
\begin{align*}
    \asymp (\lambda^{1/2})^2\lambda^2 \int_{M\backslash\mathrm{SU}(2)}\iint_{\BH^2} \chi_1 e^{i\lambda F(z_1,z_2,k)}dz_1dz_2dk.
\end{align*}
Here $M$ is the digonal subgroup of $\mathrm{SU}(2)$, $(\lambda^{1/2})^2$ is from the integrals for $\tilde{\phi}$, $\lambda^2$ is from the Plancherel density for $\BH^3$, $\chi_1$ is a compactly supported smooth amplitude factor. The phase function $F(z_1,z_2,k)$ is equal to 
\begin{align*}
    -A(z_1)+A(z_2)-A(kz_1^{-1}gz_2) + O(\lambda^{-1})
\end{align*}
where $A$ is the Iwasawa projection (see Section \ref{notation}). 
The set of the critical points of $F$, if ignoring the error term, is 2-dimensional
\begin{align*}
    \{ (z_1,z_2,e)\in \BH^2\times\BH^2\times M\backslash \mathrm{SU}(2)\, | \,z_1,z_2\in l \}.
\end{align*}
Therefore, by assuming nondegeneracy and by stationary phase, we expect $I(\lambda,\phi,g)\asymp \lambda\asymp\lambda^{1/2} \| \phi \|_{L^2(\BH^2)}^2$ to hold for such $\phi$.
It may be seen that, in general, the contribution of the integral over a neighborhood of $\BH^2\cap g\BH^2$ can not be neglected.

The alternate approach we take in this paper is to decompose $\chi\phi$ into geodesic beams:
\begin{align}\label{geodesic beam decomposition in outline}
    \chi\phi = \sum_{m,n} \chi\varphi_{m,n},\quad\text{ for }\phi\in\cS(\BH^2)\cap H_\beta,
\end{align}
and we let $\phi_{m,n}=\chi\varphi_{m,n}\in C_c^\infty(\BH^2)$.
Each $\phi_{m,n}$ is approximately a Gaussian beam along some geodesic segment $l_{m,n}$ and does not need to be still in $H_\beta$.
We construct this decomposition \eqref{geodesic beam decomposition in outline} in detail in Section \ref{section geodesic beam}
and show the square inequality in Proposition \ref{l2L2 inequality}, that is,
\begin{align}\label{l2L2 inequality in outline}
    \sum_{m,n} \| \phi_{m,n}\|_2^2\leq  C\|\phi\|_2^2.
\end{align}
for some constant $C>0$ which is independent of $\phi$.
We divide the set $\{ \phi_{m,n} \}$ into two parts, based on their $L^2$-norms.
Let $0<\delta<1$ be a small parameter to be chosen later.
We denote by $\cS_{\geq\delta}(\phi)$ the set of $\phi_{m,n}$'s so that their $L^2$-norms are $\geq(C\delta)^{1/2}\|\phi\|_{L^2(\BH^2)}$,
and denote by $\cS_{<\delta}(\phi)$ the complement of $\cS_{\geq\delta}(\phi)$ in $\{ \phi_{m,n} \}$.
We let $\phi_\delta = \sum_{\phi_{m,n}\in\cS_{<\delta}(\phi)} \phi_{m,n}$, let $\varphi_\delta = \sum_{\phi_{m,n}\in\cS_{<\delta}(\phi)} \varphi_{m,n}$, and then write $ \langle  \psi, \chi\phi\rangle$ as
\begin{align}\label{eq: intro 2 term expression}
     \sum_{\phi_{m,n}\in\cS_{\geq\delta}(\phi)} \langle \psi , \phi_{m,n}\rangle + \langle \psi , \phi_{\delta}\rangle.
\end{align}

We estimate the first term in \eqref{eq: intro 2 term expression} by bounding $\langle  \psi, \phi_{m,n}\rangle$ for each individual $\phi_{m,n} \in\cS_{\geq\delta}(\phi)$.
In Proposition \ref{uniform bound for phi mn}, we establish an estimate
\begin{align*}
    I(\lambda,\varphi_{m,n},g)\ll_{\epsilon',\epsilon}\lambda^{1/2 + \epsilon} \beta^{3/2}\|\phi_{m,n}\|_{L^2(\BH^2)} ^2 + O_A(\lambda^{-A}\|\phi\|_{L^2(\BH^2)} ^2)
\end{align*}
for general $g$. This bound is weaker than the expected bound \eqref{std arguement oscillatory integral} but the proof is easier. 
Corollary \ref{small I result} shows that $I(\lambda,\varphi_{m,n},g_0^{-1}\gamma g_0)$ can be neglected unless the geodesic segment $\gamma g_0 l_{m,n}$ is contained in a small tubular neighborhood of $g_0 l_{m,n}$.
The proof is based on the oscillatory integral estimates proved in Section \ref{section of oscillatory integrals}.
In Section \ref{section of Hecke returns}, we use a diophantine argument to show that there are few $\gamma$ such that $\gamma g_0 l_{m,n}$ is contained in a small neighborhood of $g_0 l_{m,n}$. 
Combining these estimates with the amplification method, we conclude the following bound in Proposition \ref{amplified bound for geodesic beam}:
\begin{align*}
    \langle\psi,\phi_{m,n}\rangle \ll_{\epsilon',\epsilon} \lambda^{1/4 - 1/16 + \epsilon} \beta^{13/16} \|\phi_{m,n} \|_{L^2(\BH^2)} + O_A(\lambda^{-A}\|\phi\|_{L^2(\BH^2)} ).
\end{align*}
Inequality \eqref{l2L2 inequality in outline} implies the cardinality of the set $\cS_{\geq\delta}(\phi)$ is $\leq \delta^{-1}$. Hence, the Cauchy-Schwarz inequality and \eqref{l2L2 inequality in outline} give
\begin{align}\label{power saving L delta bound}
    \sum_{\phi_{m,n}\in\cS_{\geq\delta}(\phi)}\langle\psi,\phi_{m,n}\rangle \ll_{\epsilon',\epsilon} \delta^{-1/2}\lambda^{1/4 - 1/16 + \epsilon}\beta^{13/16}\|\phi\|_{L^2(\BH^2)}.
\end{align}

We bound $\langle  \psi, \phi_{\delta}\rangle$
by applying a second arithmetic amplification argument.
We consider the Hecke returns with respect to $g_0\BH^2$ instead of a geodesic this time.
We first deduce the trivial bound 
$I(\lambda,\varphi_\delta,g)\ll\lambda^{1/2}\|\phi\|_{L^2(\BH^2)}^2$
if $g$ stabilizes $\BH^2$ in Proposition \ref{uniform bound for phi delta}. To get a better estimate when $\gamma$ is away from stabilizing $g_0\BH^2$, we expand the integral $I(\lambda,\varphi_\delta,g_0^{-1}\gamma g_0)$ as the sum of integrals for geodesic beams
\begin{align*}
    I(\lambda,\varphi_\delta,g_0^{-1}\gamma g_0)=& \sum_{\phi_{m_1,n_1},\phi_{m_2,n_2}\in\cS_{<\delta}(\phi)}\iint_{\BH^2} \overline{\phi_{m_1,n_1}(z_1)}\phi_{m_2,n_2}(z_2)k_\lambda(z_1^{-1}g_0^{-1}\gamma g_0z_2) dz_1 dz_2 .
\end{align*}
We show in Corollary \ref{nonstationary phase bound for P} that the above integrals for the pair of $\phi_{m_i,n_i}$, $i=1,2$, can be neglected unless $\gamma g_0 l_{m_2,n_2}$ is contained in a small tubular neighborhood of $g_0 l_{m_1,n_1}$. The proof is also based on the oscillatory integral estimates proved in Section \ref{section of oscillatory integrals}.
In Section \ref{section of geodesic in H3} and Proposition \ref{most of geodesic beams satisfy uniform condition}, we show that there are few pairs of $(l_{m_1,n_1},l_{m_2,n_2})$ such that $\gamma g_0 l_{m_2,n_2}$ is contained in a small tubular neighborhood of $g_0 l_{m_1,n_1}$ if $\gamma $ is away from stabilizing $g_0\BH^2$. 
For the rest of the pairs of geodesics that are close, we use the general bound 
\begin{align*}
    &\iint_{\BH^2} \overline{\phi_{m_1,n_1}(z_1)}\phi_{m_2,n_2}(z_2)k_\lambda(z_1^{-1}g_0^{-1}\gamma g_0z_2) dz_1 dz_2 \\
    \ll_{\epsilon',\epsilon} &\lambda^{1/2+ \epsilon} \beta^{3/2}\|\phi_{m_1,n_1}\|_{L^2(\BH^2)} \|\phi_{m_2,n_2}\|_{L^2(\BH^2)} + O_A(\lambda^{-A}\|\phi\|_{L^2(\BH^2)} ^2)
\end{align*}
proved in Proposition \ref{general uniform bound} with the property $\| \phi_{m_i,n_i} \|_{L^2(\BH^2)} \ll \delta^{1/2}\|\phi\|_{L^2(\BH^2)}$.
Proposition \ref{Hecke return wrt SL(2,R)} gives another Hecke return estimate which controls the number of $\gamma$ such that $\gamma$ almost stabilizes $g_0 \BH^2$. 
Combining these, in Proposition \ref{amplified bound for small norm}, by the amplification method, we conclude the following estimate for the second term in \eqref{eq: intro 2 term expression}:
\begin{align}\label{power saving S delta bound}
    \langle\psi,\phi_{\delta}\rangle \ll_{\epsilon',\epsilon} \delta^{1/138}\lambda^{1/4+\epsilon}\beta^{1/92} \|\phi\|_{L^2(\BH^2)}.
\end{align}

By combining \eqref{2 trivial inequality}, \eqref{eq: intro 2 term expression}, \eqref{power saving L delta bound} and \eqref{power saving S delta bound}, we conclude that
\begin{align*}
    \|\chi\psi|_Y\|_{L^2(\BH^2)} &\ll  \|(1-\Pi_\beta)(\chi\psi|_Y)\|_{L^2(\BH^2)} + \|\Pi_\beta(\chi\psi|_Y)\|_{L^2(\BH^2)} \\
    &\ll_{\epsilon',\epsilon}\lambda^{1/4}\beta^{-1/4} + \delta^{-1/2}\lambda^{1/4 - 1/16 + \epsilon}\beta^{13/16} + \delta^{1/138}\lambda^{1/4+\epsilon}\beta^{1/92} .
\end{align*}
By picking suitable parameters $\delta$ and $\beta$, we can prove Theorem \ref{main theorem in intro}.

\subsection{Acknowledgements}
The author would like to thank his advisor Simon Marshall for suggesting this problem and providing much guidance and encouragement in the course of his work. 
He would like to thank Jean-Philippe Anker and Angela Pasquale for answering questions on the Helgason transform.
He also wants to thank Mingfeng Chen, Jacob Denson, Ruofan Jiang, Liding Yao, John Yin, and Cheng Zhang for helpful conversations. 
The author was supported by NSF grants DMS-1902173 and DMS-1954479.

\section{Notation}\label{notation}
Throughout the paper, the notation $A\ll B$ will mean that there is a positive constant $C$
such that $|A| \leq C B$, and $A\asymp B$ will mean that there are positive constants $C_1$ and $C_2$ such
that $C_1B \leq A \leq C_2B$.
We also use $A =O(B)$ to mean $A\ll B$.
The notation $A = o(B)$ means $A/B$ tends to 0.
Then $A\sim B$ will mean $A=B+o(B)$.

\subsection{Quaternion algebras and adelic groups}
Let $F$ be a totally real number field with the ring of integers $\cO$. We denote the norm of an ideal $\fn$ of $\cO$ by $\operatorname{N}(\fn)$. If $v$ is a finite place of $F$ and $F_v$ is the local field of $F$ at $v$ with the ring of integers $\cO_v$, $\varpi_v$ and $q_v$ denote a uniformizer and the order of the residue field. Let $E$ be a quadratic extension of $F$ with exactly one complex place $w_0$, and let $v_0$ be the place below $w_0$. Let $\cO_E$ be the ring of integers of $E$. We denote by $\operatorname{N}_E$ the norm for ideals in $\cO_E$. We let $\bar{\cdot}$ denote the conjugation of $E$ over $F$, and denote the trace and norm from $E$ to $F$ by $\operatorname{Tr}_E$ and $\operatorname{Nm}_E$ respectively. We denote the rings of adeles of $F$ and $E$ by $\BA$ and $\BA_E$, and the ring of finite adeles of $F$ by $\BA_f$. Let $|\cdot|_v$ be the absolute value on $F_v$ for any place $v$ of $F$, and let $|\cdot| = \prod_v |\cdot|_v$ be the absolute value on $\BA$. Let $|\cdot|_w$ be the absolute value on $E_w$ for any place $w$ of $E$, and let $|\cdot|_E = \prod_w |\cdot|_w$ be the absolute value on $\BA_E$.

Let $D$ be a quaternion division algebra over $F$ that is ramified at every real place except at $v_0$, and let $B = D\otimes_F E$. We assume that $B$ is also a division algebra. We denote the standard anti-involution on $D$ and $B$ by $\iota$, and also use $\bar{\cdot}$ to denote the conjugation of $B$ over $D$. Let $\operatorname{trd}_B(x) = x+\iota(x)$ and $\operatorname{nrd}_B(x)=x\iota(x)$ denote the reduced trace and reduced norm on $B$.

We denote the groups of reduced norm $1$ elements by $D^1$ and $B^1$, and denote the multiplicative group by $D^\times$ and $B^\times$. We let $\bG$ and $\bH$ be the algebraic groups over $F$ such that $\bG(F) = B^1$, $\bH(F) = D^1$.
We denote $D\otimes_F F_v$, $B\otimes_F F_v$,  $\bG(F_v)$ and $\bH(F_v)$  by $D_v$, $B_v$, $G_v$ and $H_v$.
At $v_0$, we have $G_{v_0} \simeq \SL(2,\BC)$, $H_{v_0}\simeq \SL(2,\BR)$.
We will implicitly make these identifications later.
Let $H^\prime$ be the normalizer of $H_{v_0}$ in $G_{v_0}$ which has two components with identity component $H_{v_0}$. In fact,
\begin{align*}
    H^\prime = H_{v_0} \cup \begin{pmatrix}0&i\\i&0\end{pmatrix} H_{v_0}. 
\end{align*}

Let $\cO_B\subset B$ be a maximal $\cO$ order, and let $S$ be a finite set of places of $F$ containing all infinite places and all places where $D$ ramifies and places where $E$ is ramified over $F$. We choose a compact subgroup $K = \prod_v K_v$ of $\bG(\BA)$ as follows. We choose $K_{v_0} = \SU(2)$, and $K_v = B^1_v$ for all other real places $v$.
The intersection $K_{v_0}\cap H_{v_0}$ can be identified with $\SO(2)$ and is a maximal compact subgroup of $H_{v_0}$. 
For finite places in $S$ we let $K_v$ be any subgroup of $G_v$ that $K_v\subset\cO_{B,v} = \cO_B \otimes _\cO \cO_v$, and for other finite places we let $K_v = B^1_v\cap \cO_{B,v}$.
After enlarging $S$, we may choose an isomorphism $\alpha$ from $B_v$ to the product of two $2$ by $2$ matrix algebras $M_2(F_v)\times M_2(F_v)$ for all $v\notin S$ that is split in $E$.  The isomorphism $\alpha$ satisfies:
\begin{itemize}
    \item $\alpha(D_v) = \{ (T,T)|T\in M_2(F_v) \}$, 
    \item $\alpha(K_v) = \SL(2,\cO_v)\times\SL(2,\cO_v)$,
\end{itemize}
and moreover here $\bar{\cdot}$ is identified with the map switching the two factors.
Let $\sP$ be the set of primes $v\notin S$ that split in $E$. We shall implicitly make the identification $\alpha$ at places in $\sP$. For $v\in \sP$, we write $G_v = \SL(2,F_v)\times\SL(2,F_v)$, and  $K_v = K_{v,1}\times K_{v,2} =\SL(2,\cO_v)\times\SL(2,\cO_v) $.

Let $\sQ$ be the set of primes $v\notin S$ that is inert in $E$. We still use $\alpha$ to mean the isomorphism from $B_v$ to $M_2(E_v)$, the $2$-by-$2$ matrix algebra over $E_v$, for $v\in\sQ$. We may assume $\alpha$ satisfies:
\begin{itemize}
    \item $\alpha(D_v) =  M_2(F_v) $, 
    \item $\alpha(K_v) = \SL(2,\cO_{E_v})$.
\end{itemize}
For $v\in \sQ$, we write $G_v = \SL(2,E_v)$, and  $K_v = \SL(2,\cO_{E,v})$.

For simplicity, we denote $G_0=G_{v_0}$, $H_0 = H_{v_0}$,  and $K_0=K_{v_0}$.

\subsection{Lie groups and algebras}

Let $A$ be the connected component of the diagonal subgroup of $\mathrm{SL}(2,\BR)$, i.e.,
\begin{align*}
    A = \left\{ \left.\begin{pmatrix}e^{t/2} &0 \\0 &e^{-t/2} \end{pmatrix} \right| t\in\BR\right\}.
\end{align*}
Let
\begin{align*}
    N = \left\{ \begin{pmatrix}1 &z \\0 &1 \end{pmatrix} \Big| z\in\BC \right\}\quad \text{and}\quad N_0 = \left\{ \begin{pmatrix}1 &x \\0 &1 \end{pmatrix} \Big| x\in\BR \right\}. 
\end{align*}
We denote the Lie algebras of $K_{0}$, $A$, $N$ and $N_0$ by $\fk$, $\fa$, $\fn$ and $\fn_0$. We write the Iwasawa decompositions of $G_0 = NAK_0$ as
\begin{align*}
    g = n(g)\exp\left(A(g)\right) \kappa(g) = \exp\left( N(g)\right)\exp\left(A(g)\right) \kappa(g).
\end{align*}
We define
\begin{align*}
    H = \begin{pmatrix}1/2 &0 \\0 &-1/2 \end{pmatrix} \in \fa,\quad X = \begin{pmatrix}0 &1 \\0 &0 \end{pmatrix} \in\fn.
\end{align*}
We identity $\fa\simeq \BR$ under the map $H\mapsto 1$ and consider $A(g)$ as a function $A:G_0\to\BR$ under this identification, and obtain identifications $\fn\simeq \BC$ and $\fn_0\simeq\BR$ similarly. We identify the dual space $\fa^*$ of $\fa$ as $\BR$ by sending the root $tH\mapsto t$ to $1$. Under these identifications, the pairing between $\fa$ and $\fa^*$ is the multiplication in $\BR$.
We define the homomorhpism $a:\BR \rightarrow A$ by $a(t) = \exp(tH)$, and define $n:\BC \to N$ to be $n(z) = \exp(zX)$. The restriction of $n$ to $\BR\to N_0$ will still be denoted by $n$. We denote by $\fa^+ \simeq \BR_{>0}$ the positive Weyl chamber, and by $A^+$ the image $\exp(\fa^+) = a(\BR_{>0})$ in $A$.

Let $M^\prime$ and $M$ be the normalizer and centralizer respectively of $A$ in $K_0$. Then $W=M^\prime/M\simeq \BZ/2$ is the Weyl group with respect to $A$, and  $M^\prime A$ and $MA$ are the normalizer and centralizer, respectively, of $A$ in $G_0$. 

We equip $\mathfrak{sl}(2,\BC)$ with the norm 
\begin{align*}
    \|\cdot \|:\begin{pmatrix}Z_1&Z_2\\Z_3&-Z_1\end{pmatrix}\mapsto \left(\|Z_1\|^2+\|Z_2\|^2+\|Z_3\|^2 \right)^{1/2}.
\end{align*}
This norm induces a left-invariant metric on $G_0$, which we denote by $d$. We will denote the group identity by $e$.

 We define a Haar measure $dg$ on $G_0$ through the Iwasawa decomposition $G_0=NAK_0$. Namely, if $g=n(z)a(t)k$, then $dg=e^{-2t}dzdtdk$. Here $dz,dt$ are the standard measures on $\BC$ and $\BR$ as Euclidean spaces, and $dk$ is the probability Haar measure on $K_0$. Similarly, we define a Haar measure $dh$ on $H_0$ through the Iwasawa decomposition $H_0=N_0A\SO(2)$. If $h=n(x)a(t)k$, then $dh=e^{-t}dxdtdk$.

\subsection{Hecke algebras}\label{sec of Hecke algberas}
For any continuous function $f$ on $\bG(\BA)$, we define $f^*(g) = \overline{f(g^{-1})}$. We define $\cH_f = \bigotimes_{v<\infty}^\prime \cH_v$ to be the convolution algebra of smooth functions on $\bG(\BA_f)$ that are compactly supported and bi-invariant under $K_f = \prod_{v<\infty}^\prime K_v$, and $\cH_v$ denote the space of smooth, compactly supported functions on $G_v$ that are bi-invariant under $K_v$.
We let $\cH^S = \bigotimes_{v\notin S}^\prime \cH_v$ be the unramified Hecke algebra.
We let $v\in\sP$, and $a_1,a_2\in\BZ$. Define $K_v(a_1,a_2)$ to be the double coset
\begin{align*}
    K_v(a_1,a_2) = K_{v,1}(a_1)\times K_{v,2}(a_2),
\end{align*}
where $K_{v,i}(a_i) = K_{v,i}\begin{pmatrix}\varpi_v^{a_i}&\\&\varpi_v^{-a_i}\end{pmatrix}K_{v,i}$.
We let
$T_v(a)$ be the characteristic function of $K_v(a,0)$. 
If $v\in\sQ$ and $a\in\BZ$, we define $K_v(a)$ to be the double coset
\begin{align*}
    K_v(a) = K_{v}\begin{pmatrix}\varpi_v^{a}&\\&\varpi_v^{-a}\end{pmatrix}K_{v}.
\end{align*}
We let
$T_v(a)$ be the characteristic function of $K_v(a)$. 
Given an ideal $\fn\subset\cO$ and suppose that $\fn$ is only divisible by prime ideals not in $S$.
We define the double coset in $\bG(\BA_f)$
\begin{align*}
    K(\fn) = \prod_{v\in\sP} K_v(\operatorname{ord}_{v}(\fn),0) \times \prod_{v\in\sQ} K_v(\operatorname{ord}_v(\fn))\times\prod_{v\in S} K_v.
\end{align*}
In this paper the notation $K(\fn)$ is used non-standardly. We use this notation to denote the above Hecke double coset but not a principal congruence subgroup of $K$.

The action of $\phi\in\cH_f$ on an automorphic function $f$ on $\bG(F)\backslash \bG(\BA)$ is given by the usual formula
\begin{align*}
    [\phi f](x) = \int_{\bG(\BA_f)}\phi(g)f(xg) dg.
\end{align*}
Here we use the Haar measures $dg_v$ on $G_v$, which are normalized so that $K_v$ has a unit volume.

\subsection{Arithmetic manifolds and Hecke-Maass forms}

Define $X = \bG(F)\backslash\bG(\BA)/ K$, which is a compact connected hyperbolic $3$-manifold. We let $\Omega = \prod_{v}\Omega_v\subset\bG(\BA)$ be a compact set containing a fundamental domain for $\bG(F)\backslash \bG(\BA)$.
The universal cover of $X$ is the hyperbolic $3$-space $\BH^3$ obtained by the quotient $G_0/ K_{0} = \SL(2,\BC)/\SU(2)$.
We will denote the universal covering map by $\pi: \BH^3 \to X$.
We will use the upper half 3-space model of $\BH^3$ and the upper half plane model of $\BH^2$, i.e., 
\begin{align*}
    \BH^3 = \{ (z,\bt) \,|\, z\in\BC,\, \bt\in\BR_{>0} \}\quad\text{and}\quad\BH^2 = \{ (x,\bt)\,|\,x\in\BR,\,\bt\in\BR_{>0} \}.
\end{align*}
The embedding $\BH^2\subset\BH^3$ is identified with the natural imbedding $\SL(2,\BR)/\SO(2)\subset\SL(2,\BC)/\SU(2)$.
We denote by $o = (0,1)\in\BH^3$ the origin of the symmetric space, which is the point fixed by the group $\SU(2)$. For $g =  \begin{pmatrix}a &b \\c &d \end{pmatrix} \in  \SL(2,\BC)$ and $(z,\bt)\in\BH^3$, we have the following formula:
\begin{align}\label{Poincare}
    g\cdot(z,\bt) = \left(\frac{(az+b)\overline{(cz+d)}+a\overline{c}\bt^2}{|cz+d|^2+|c|^2\bt^2},\frac{\bt}{|cz+d|^2+|c|^2\bt^2} \right).
\end{align}
And then $H^\prime$ will be the stabilzer of $\BH^2$ in $G_0$.
We will usually represent the $ \bt$-coordinate under the exponential map, that is,
$\bt = e^t$ with $t\in\BR$. Hence,
\begin{align*}
    (z,\bt) = n(z)a(t)\cdot o\in\BH^3,
\end{align*}
where $z\in\BC$ and $t\in\BR$.

We let $\psi\in L^2(X)$ be a Hecke-Maass form that is an eigenfunction of the Laplacian $\Delta$ and the unramified Hecke algebra $\cH^S$. We let $\lambda>0$ be its spectral parameter, so that
\begin{align*}
    \Delta\psi+(1+\lambda^2)\psi=0.
\end{align*}
We assume that $\|\psi\|_{L^2(X)}=1$ with respect to the hyperbolic volume on $X$. Note that because $\Delta$ and $T(\fn)$ are self-adjoint, we may assume that $\psi$ is real-valued.

\section{Fourier transforms on \texorpdfstring{$\BH^2$}{H2} and \texorpdfstring{$\BH^3$}{H3}}\label{sec: Fourier}

We use the Harish-Chandra transforms and the Helgason transforms on both $\BH^2$ and $\BH^3$, so we must fix the notation we will use. 
We use $\widehat{\cdot}$ to mean the transforms on $\BH^3$ or $\SL(2,\BC)$ and $\tilde{\cdot}$ to mean the transforms on $\BH^2$ or $\SL(2,\BR)$. 
For definitions of Helgason and Harish-Chandra transforms on general noncompact symmetric spaces and their properties, we refer to \cite{helgason1984groups,helgason1994geometric}.

Let $\varphi_s^{\BH^2}$ and $\varphi_s^{\BH^3}$ be the spherical functions on $H_0$ and $G_0$ with spectral parameter $s$ respectively. We recall the integral representations for the spherical functions. For any $x\in\BH^2$ and $y\in\BH^3$ we have
\begin{align}
    &\varphi_s^{\BH^2}(x)=\int_{\SO(2)} e^{(is+1/2)A(kx)}dk,\label{eq:HC int for H2}\\
    &\varphi_s^{\BH^3}(y)=\int_{K_0} e^{(is+1)A(ky)}dk.\label{eq:HC int for H3}
\end{align}
Here the measures on the maximal compacts are both probability Haar measures.

Let $f\in C_c^\infty(\BH^3)$ be a left $K_0$-invariant function,  that is, $f(kx) = f(x)$ for arbitrary $k\in K_0$ and $x\in\BH^3$. The Harish-Chandra transform on $\BH^3$ of $f$ is defined to be
\begin{align*}
    \widehat{f}(s) = \int_{\BH^3} f(x) \varphi_{-s}^{\BH^3}(x) dx.
\end{align*}
Here $dx$ is the standard hyperbolic measure on $\BH^3$ so that the Haar measure $dg$ on $G_0$ is the extension of $dx$ by the measure of mass $1$ on $K_{0}$.

We denote by $B$ the boundary of the symmetric space $\BH^2$, which is the union of the real line and the point at infinity in the upper half plane model. By the linear fractional action of $\SO(2)$, the boundary $B$ can be identified by $\SO(2) / \{\pm I\}$.
We use representatives in $\SO(2)$ to mean the elements in $B$, and then we can parametrize $B$ by the circle $\BR / 2\pi\BZ$, that is
\begin{align*}
    B = \left\{  \left.  b_\theta = \begin{pmatrix} \cos\theta/2 & \sin\theta/2\\ -\sin\theta/2& \cos\theta/2 \end{pmatrix}  \right| \theta\in\BR/2\pi\BZ   \right\}.
\end{align*}
Given a function $g\in C_c^\infty(\BH^2)$, the Helgason transform of $g$ on $\BH^2$ is defined by
\begin{align*}
    \tilde{g} (s,b) = \int_{\BH^2} g(x) e^{(-is+1/2) A(x,b)} dx,
\end{align*}
where $dx$ is the standard hyperbolic measure on $\BH^2$, $s\in\BC$, $b\in B$ and $A(x,b)$ is the distance from the origin in $\BH^2$ to the horocycle connecting $x$ and $b$. $A(x,b)$ can be calculated by the formula
\begin{align*}
    A(x,b) = A(b^{-1}x).
\end{align*}
Here we treat $x,b$ as group elements in $H_0$, and the $A$ on the right-hand side is the Iwasawa projection with respect to the Iwasawa decomposition $H_0=N_0 A \SO(2)$.
If, furthermore, $g$ is a left $\SO(2)$-invariant function, then its Helgason transform agrees with the Harish-Chandra transform on $\BH^2$, that is
\begin{align*}
    \tilde{g}(s,b) = \tilde{g}(s) = \int_{\BH^2} g(x) \varphi_{-s}^{\BH^2}(x) dx.
\end{align*}
Let $f,g$ be functions on $\BH^2$. Their convolution, which we denote by $f\times g$, is defined to be the convolution of the pullbacks of $f,g$ on the group $H_0$, i.e.,
\begin{align*}
    (f\times g)(x) = \int_{H_0} f(y\cdot o) g(y^{-1}\cdot x) dy.
\end{align*}
If moreover we assume $g$ is left $\SO(2)$-invariant, then
\begin{align*}
    \tilde{f\times g}(s,b) = \tilde{f}(s,b)\tilde{g}(s).
\end{align*}

We will not use the Helgason transform on $\BH^3$ until Section \ref{section of Kakeya-Nikodym norms} when dealing with the geodesic tube restrictions, so we shall postpone the definition and properties of the Helgason transform on $\BH^3$.

We denote by $d\nu(s)$ $(s\in\BR)$ the Plancherel measure for $\BH^3$, and $d\mu(s)$ $(s\in\BR)$ the Plancherel measure for $\BH^2$. Let $db$ be the Haar measure on $B$ so that the total measure of $B$ is $1$, and we denote by $d\mu(s,b) = d\mu(s)db$ where $s\in\BR$ and $b\in B$. 
We normalize the Plancherel measures so that the inverse Harish-Chandra transforms hold. More precisely, for $f\in C_c^\infty(\BH^3)$ and $g\in C_c^\infty(\BH^2)$ that are left invariant under the maximal compact subgroups, we have
\begin{align*}
    f(x) &= \frac{1}{2}\int_\BR \widehat{f}(s)\varphi_s^{\BH^3}(x) d\nu(s)=\int_0^\infty \widehat{f}(s)\varphi_s^{\BH^3}(x) d\nu(s),\\
    g(x) &= \frac{1}{2}\int_{\BR}\tilde{g}(s) \varphi_s^{\BH^2}(x) d\mu(s)=\int_0^\infty\tilde{g}(s) \varphi_s^{\BH^2}(x) d\mu(s).
\end{align*}
For any $g\in C_c^\infty(\BH^2)$, we have
\begin{align}\label{Fourier inversion}
    g(x) &= \frac{1}{2}\int_{\BR\times B}\tilde{g}(s,b) e^{(is+1/2)A(x,b)} d\mu(s,b)=\int_0^\infty \int_{ B}\tilde{g}(s,b) e^{(is+1/2)A(x,b)} d\mu(s,b).
\end{align}  
From computing Harish-Chandra's $\bc$-functions, we know that $d\nu(s)$ is some constant multiplying $s^2 ds$, and $d\mu(s)$ is some constant multiplying $s\tanh(\pi s) ds$. We will also use $\cF$ and $\cF^{-1}$ to mean Helgason transform and its inversion on $\BH^2$.
We list the results of Helgason transforms that we will use later: the Plancherel formula and the isomorphism on the classes of Schwartz functions.

\begin{theorem}\cite[Ch. III, Theorem 1.5, p. 202]{helgason1994geometric}\label{plancherel formula}
    The Helgason transform extends from $C_c^\infty(\BH^2)$ to an isometry of $L^2(\BH^2)$ onto $L^2(\BR_{>0}\times B)$ (with the measure $d\mu(s,b)$ on $\BR_{>0}\times B$).
    We shall still denote the extended transform by $\tilde{\cdot}$ or $\cF$ and we have the Plancherel formula
    \begin{align*}
        \int_{\BH^2} f_1(x)\overline{f_2(x)} dx = \int_0^\infty \int_{ B} \tilde{f_1}(s,b) \overline{\tilde{f_2}(s,b) } d\mu(s,b),
    \end{align*}
    for $f_1,f_2\in L^2(\BH^2)$.
\end{theorem}
For any $f_1,f_2\in L^2(\BH^2)$, we will denote the Hermitian form by
\begin{align*}
    \langle f_1,f_2\rangle = \int_{\BH^2}f_1(x)\overline{f_2(x)} dx.
\end{align*}
Similarly, if $\phi_1,\phi_2\in L^2(\BR_{>0}\times B)$, we will write
\begin{align*}
    \langle \phi_1,\phi_2\rangle = \int_0^\infty\int_{ B}\phi_1(s,b)\overline{\phi_2(s,b)} d\mu(s,b).
\end{align*}

The Schwartz space on a symmetric space is analogous to the Schwartz space on an Euclidean space. 
We define the distance function on $H_0$ according to the Cartan decomposition.
For $g= k_1 a(t) k_2$ with $k_1,k_2\in \SO(2)$ and $t\in\BR$, we put $|g| = |t|$.
Let $\bD(H_0)$ be the algebra of left-invariant differential operators on $H_0$, and $\overline{\bD}(H_0)$ the algebra of right invariant differential operators on $H_0$. Then the Harish-Chandra Schwartz space $\cS(H_0)$ is defined as the space of smooth functions $f$ on $H_0$ satisfying
\begin{align}\label{Seminorm for HarishChandra Schwartz}
    \sup_{g\in H_0}\left|    (1+|g|)^N \varphi^{\BH^2}_0(g)^{-1}(DEf)(g) \right| < \infty
\end{align}
for each $D\in\bD(H_0)$, $E\in\overline{\bD}(H_0)$ and each positive integer $N$. The space $\cS(H_0)$ is topologized by the seminorms (\ref{Seminorm for HarishChandra Schwartz}).
The Harish-Chandra Schwartz space of the symmetric space $\cS(\BH^2)$ is then defined as the space of $f\in\cS(H_0)$ which are right invariant under the maximal compact subgroup $\SO(2)$. It is clear that
\begin{align*}
    C_c^\infty(\BH^2)\subset\cS(\BH^2)\subset L^2(\BH^2)
\end{align*}
are dense embeddings in the $L^2$ sense. 

Let $\bD(\BR)$ be the algebra of differential operators on $\BR$ with constant coefficients, and $\bD(B)$ the algebra of differential operators on 
$B$ invariant under $\SO(2)$.
The Schwartz space $\cS(\BR\times B)$ on the frequency space $\BR\times B$ is then defined as the space of smooth functions $\phi$ on $\BR\times B$ satisfying
\begin{align}\label{Seminorm for S(R,B)}
    \sup_{s\in\BR,b\in B}\left|    (1+|s|)^N (DE\phi)(s,b) \right| < \infty
\end{align}
for each $D\in\bD(\BR)$, $E\in\bD(B)$, and each positive integer $N$.
The space $\cS(\BR\times B)$ is topologized by the seminorms \eqref{Seminorm for S(R,B)}.
We denote by $\cS(\BR\times B)_W$ the space of Schwartz functions $\phi\in\cS(\BR\times B)$ satisfying the Weyl invariance functional equation
\begin{align}\label{functional equation for functions in image of helgason transform}
    \int_B e^{(is+1/2)A(x,b)}\phi(s,b)db = \int_B e^{(-is+1/2)A(x,b)}\phi(-s,b)db 
\end{align}
for arbitrary $s\in\BC$ and $x\in\BH^2$.
It is due to Eguchi that the Helgason transform is also continuous on the Harish-Chandra Schwartz space and has the same inversion formula.

\begin{theorem}[\cite{eguchi1979asymptotic, eguchi1977fourier}]\label{Eguchi's theorem}
    The Helgason transform is a continuous isomorphism of $\cS(\BH^2)$ onto $\cS(\BR\times B)_W$, where the inverse is given by the integral (\ref{Fourier inversion}).
\end{theorem}

\subsection{Extension of Schwartz functions with Weyl invariance}

We denote by $\cS(\BR_{>0}\times B)$ the space of restrictions $\phi|_{\BR_{>0}\times B}$ for all $\phi\in\cS(\BR\times B)$.
Since every $\phi \in L^2(\BR_{>0}\times B)$ is in the image of the $L^2$ Helgason transform, it may be natural to ask whether every $\phi \in \cS(\BR_{>0}\times B)$ is also the Helgason transform of some Harish-Chandra Schwartz function.
By Theorem \ref{Eguchi's theorem}, the question is equivalent to whether we can extend any $\phi \in \cS(\BR_{>0} \times B)$ to a function in $\cS
(\BR\times B)_W$.

This direct generalization is false. 
If we take $\phi \in \cS(\BR_{>0}\times B)$ to be constant in the $B$ variable, then the Weyl invariance condition \eqref{functional equation for functions in image of helgason transform} implies that 
\begin{align*}
    \phi(s)\int_B e^{(is+1/2)A(x,b)}db = \phi(-s)\int_B e^{(-is+1/2)A(x,b)}db.
\end{align*}
By taking $x\in\BH^2$ to be the origin, we see that $\phi$ is an even Schwartz function on $\BR$, but it is clear that not all Schwartz functions on $\BR_{>0}$ can be extended to be even.
However, the simple modification requiring $0\notin\supp(\phi)$ will give the extension property.

\begin{proposition}\label{Prop Weyl invar extension}
    Suppose $\phi \in \cS(\BR_{>0} \times B)$ satisfies $\supp(\phi) \subset [c,\infty) \times B$ for some $c > 0$.
    There exists a unique $\Phi \in \cS(\BR \times B)_W$ so that $\Phi |_{\BR_{>0}\times B} = \phi$.
\end{proposition}

Following from this extension property and Eguchi's result (Theorem \ref{Eguchi's theorem}), a large class of Schwartz functions is in the image of the Helgason transform.

\begin{corollary}
    Suppose $\phi \in \cS(\BR_{>0} \times B)$ satisfies $\supp(\phi) \subset [c,\infty) \times B$ for some $c > 0$. There exists a unique $f\in \cS(\BH^2)$ so that $\tilde{f}(s,b) = \phi(s,b)$ for arbitrary $s > 0$ and $b\in B$.
\end{corollary}

To begin the proof, we recall the generalized spherical functions on $\BH^2$.
We parametrize $\SO(2)$ by the circle $\BR/2\pi\BZ$
\begin{align*}
    \SO(2) =\left\{  \left.  k_\theta = \begin{pmatrix} \cos\theta & \sin\theta\\ -\sin\theta& \cos\theta \end{pmatrix}  \right| \theta\in\BR/2\pi\BZ   \right\}.
\end{align*}
We will take the normalized Haar measure on $\SO(2)$, i.e., $dk_\theta = d\theta/2\pi$.
Let $n\in \BZ$. The character $\chi_{2n}:\SO(2)\to\BC^\times$ is defined as
\begin{align*}
    \chi_{2n}(k_\theta) = e^{2in\theta}.
\end{align*}
It may be seen that they are all characters of $\SO(2)$ that are trivial on $\{\pm I\}$.
For $s\in \BC$ and $n\in \BZ$, the function
\begin{align*}
    \varphi_{s,n} (x) &= \int_{\SO(2)} e^{(is + 1/2)A(k^{-1}x)} \chi_{2n}(k) dk \\
    &=\int_{B} e^{(is + 1/2)A(x,b)} \chi_{2n}(b) db 
\end{align*}
is called the generalized spherical function of class $\chi_{2n}$ on $\BH^2$. When $n = 0$, this is the zonal spherical function.

It is proved in \cite[Ch. III, Theorem 2.10, p. 240]{helgason1994geometric} that the generalized spherical function satisfies the functional equations
\begin{align}\label{functional eqn for generalized spherical functions}
    \varphi_{-s,n} (x)  =  \varphi_{s,n} (x)  \Gamma_n(s), \quad\text{for any }x\in\BH^2
\end{align}
with $\Gamma_n$ a meromorphic function on $\BC$.
In the case of rank one, an explicit formula for generalized spherical functions is calculated by Helgason \cite[Ch. III, Theorem 11.2, p. 328]{helgason1994geometric}.
By the equation (14) in \cite[p. 331]{helgason1994geometric}, we have
\begin{align*}
    \Gamma_n(s) = \frac{p_n(-s)q_n(-s)}{p_n(s)q_n(s)},
\end{align*}
where
\begin{align*}
    p_n(s) = \frac{\Gamma(\frac{1}{2}(is + \frac{1}{2}+|n|))}{\Gamma(\frac{1}{2}(is + \frac{1}{2}))},\quad  q_n(s) = \frac{\Gamma(\frac{1}{2}(is + \frac{1}{2}+1+|n|))}{\Gamma(\frac{1}{2}(is + \frac{1}{2}+1))}.
\end{align*}
Here $\Gamma$ is the gamma function.
Applying the Legendre duplication formula, it may be seen that
\begin{align*}
    p_n(s) q_n(s) = 2^{-|n|} \frac{\Gamma(is+\frac{1}{2} + |n|)}{\Gamma(is+\frac{1}{2})} = 2^{-|n|} \prod_{j = 0}^{|n|-1} (is + \frac{1}{2} + j).
\end{align*}
Hence,
\begin{align}\label{formula for Gamma n}
    \Gamma_n(s) = \prod_{j = 0}^{|n|-1} (is + \frac{1}{2} + j)^{-1}(-is + \frac{1}{2} + j).
\end{align}
It may be seen that $\Gamma_n(s)$ has no zeros or poles on $\BR$.
We also need estimates for all derivatives of $1/\Gamma_n$ uniformly for $s\in\BR$.
\begin{lemma}\label{bound for Gamma n in a strip}
    Suppose $z\in\BC$ lies in the horizontal strip
    \begin{align}\label{horizontal strip}
       \left \{ \sigma+ it 
        \,| \,\sigma\in\BR,  -\min(\frac{1}{4},\frac{1}{1+|n|})\leq t \leq \min(\frac{1}{4},\frac{1}{1+|n|})\right\}.
    \end{align}
    Then $|1/\Gamma_n(z)| \ll 1$, where the implied constant doesn't depend on $n$.
\end{lemma}
\begin{proof}
    If $ 0\leq t \leq  \min(\frac{1}{4},\frac{1}{1+|n|})$, then
    \begin{align*}
        \frac{1}{|\Gamma_n(\sigma+it)|} = \prod_{j = 0}^{|n|-1} \left(\frac{\sigma^2+ (\frac{1}{2} + j -t)^2}{\sigma^2+ (\frac{1}{2} + j +t)^2} \right)^{1/2}\leq \prod_{j = 0}^{|n|-1} 1= 1. 
    \end{align*}
    If $ - \min(\frac{1}{4},\frac{1}{1+|n|})\leq t \leq 0$, we let $s = -t$. We may assume $n$ is positive and $n \geq 3$ so that $1/4\geq 1/(1+|n|)$.
    Then we have
    \begin{align*}
        \log \frac{1}{|\Gamma_n(\sigma+it)|} &= \frac{1}{2}\sum_{j = 0}^{n-1} \log \frac{\sigma^2+ (\frac{1}{2} + j +s)^2}{\sigma^2+ (\frac{1}{2} + j -s)^2}\\
        &\leq \sum_{j = 0}^{n-1} \log \frac{\frac{1}{2} + j +s}{\frac{1}{2} + j -s}\\
        &\leq \sum_{j = 0}^{n-1} \log \frac{\frac{1}{2} + j +\frac{1}{1+n}}{\frac{1}{2} + j -\frac{1}{1+n}}\\
        &= \sum_{j = 0}^{n-1} \log\left( 1 + \frac{1}{(n+1)(j+1/2)}  \right)  - \log\left( 1 - \frac{1}{(n+1)(j+1/2)}  \right) \\
        &\ll \sum_{j = 0}^{n-1} \frac{1}{(n+1)(j+1/2)}  \\
        &\ll 1.
    \end{align*}
\end{proof}

\begin{proposition}
    Let $m\geq 0$ be an integer. Then we have
    \begin{align}\label{bound for derivative of Gamma n}
        \frac{d^m}{ds^m} \frac{1}{\Gamma_n(s)} \ll m! (1+|n|)^m,
    \end{align}
    for any $s\in\BR$.
\end{proposition}
\begin{proof}
    Let $\delta =  \min(\frac{1}{4},\frac{1}{1+|n|})$.
    We fix $s\in\BR$.
    The ball centered at $s$ with radius $\delta$ is contained in the strip \eqref{horizontal strip}.
    Then \eqref{bound for derivative of Gamma n} follows from Lemma \ref{bound for Gamma n in a strip} and the Cauchy integral formula.
\end{proof}

\begin{proof}[Proof of Proposition \ref{Prop Weyl invar extension}]
    We show the existence by an explicit construction.
    Recall that $\phi \in \cS(\BR_{>0} \times B)$ satisfies $\supp(\phi) \subset [c,\infty) \times B$ for some $c > 0$. 
    Using the characters $\chi_{2n}$ on $B = \SO(2)/\{\pm I\}$, $n\in\BZ$, we have the corresponding Fourier expansion
    \begin{align*}
        \phi(s,b) = \sum_{n\in\BZ} \phi_n(s) \chi_{2n}(b) 
    \end{align*}
    with
    \begin{align*}
        \phi_n(s) = \int_B \phi(s,b) \overline{\chi_{2n}(b) }db.
    \end{align*}
    Hence, $\supp(\phi_n) \subset [c,\infty)$.
    The Schwartz condition \eqref{Seminorm for S(R,B)} implies that
    \begin{align}\label{Schwartz condition for phi}
        \sup_{s\geq c, n\in\BZ} (1+|s|)^{N_1} (1+|n|)^{N_2} \phi_n^{(m)} (s) <\infty
    \end{align}
    for arbitrary non-negative integers $N_1,N_2,m$. Here $\phi_n^{(m)} $ means the $m$-th derivative of $\phi_n$.
    
    We define
    \begin{align*}
        \Phi_n(s) =\begin{cases}
            \phi_n(s)\quad&\text{ if }s\geq c,\\
             \phi_n(-s) / \Gamma_n(-s)\quad&\text{ if }s\leq  -c,\\
             0&\text{ otherwise,}
        \end{cases}
    \end{align*}
    and define
    \begin{align*}
        \Phi(s,b) = \sum_{n\in\BZ} \Phi_n(s) \chi_{2n}(b).
    \end{align*}
    It is clear that the Fourier series converges and $\Phi|_{\BR_{>0}\times B} = \phi$.
    To see that $\Phi(s,b)$ is a Schwartz function, we only need to show the rapid decay condition when $s < -c$.
    By \eqref{bound for derivative of Gamma n} and \eqref{Schwartz condition for phi}, we have, for $s < -c$,
    \begin{align*}
        \Phi_n^{(m)}(s) &= (-1)^m \sum_{j=0}^m \binom{m}{m-j} \phi_n^{(m-j)}(-s) \frac{d^j}{ds^j} \left( \frac{1}{\Gamma_n} \right )(-s)\\
        &\ll_{m, N_1, N_2} \left( (1+|s|)^{-N_1} (1+|n|)^{-N_2}\right) (1+ |n|)^m.
    \end{align*}
    This implies $\Phi \in \cS(\BR\times B)$.

    It remains to show $\Phi \in \cS(\BR\times B)_W$.
    For $s\geq 0$ and $x\in\BH^2$,
    we apply the functional equation for generalized spherical functions \eqref{functional eqn for generalized spherical functions} to obtain the following
    \begin{align*}
        \int_B e^{(-is+1/2)A(x,b)}\Phi(-s,b)db&= \sum_{n\in\BZ} \Phi_n(-s) \int_B e^{(-is+1/2)A(x,b)}\chi_{2n}(b)db \\
        &=\sum_{n\in\BZ} \frac{ \phi_n(s) }{ \Gamma_n(s)} \varphi_{-s,n}(x)\\
        &=\sum_{n\in\BZ} \phi_n(s) \varphi_{s,n}(x)\\
        &= \sum_{n\in\BZ} \phi_n(s) \int_B e^{(is+1/2)A(x,b)}\chi_{2n}(b)db \\
        &=\int_B e^{(is+1/2)A(x,b)}\Phi(s,b)db,
    \end{align*}
    which proves the Weyl invariance \eqref{functional equation for functions in image of helgason transform}.
    
    The uniqueness of the extension can be deduced from the inverse Helgason transform.
    Suppose that $\Psi \in \cS(\BR\times B)_W$ also satisfies $\Psi |_{\BR_{>0}\times B} = \phi$.
    Then
    \begin{align*}
        \cF^{-1}(\Phi) = \cF^{-1}(\Psi) = \int_0^\infty \int_{ B}\phi(s,b) e^{(is+1/2)A(x,b)} d\mu(s,b).
    \end{align*}
    By the isomorphism $\cF:\cS(\BH^2)\to\cS(\BR\times B)_W$, we conclude that $\Psi = \Phi$.
\end{proof}

\section{\texorpdfstring{$L^2$}{L2}-norms of hyperbolic surface restrictions}\label{section of L^2 bound}
\subsection{Main results}

We first set up the hyperbolic surface restriction problem for the eigenfunction more precisely.
We denote by $\cB_{\BH^2} \subset \BH^2$ the unit ball centered at the origin $o$.
Recall that $\pi:\BH^3\to X$ is the universal cover, and $Y \subset X$ satisfies \eqref{boundedness condition for Y}.
We let $g_0\in \Omega_{v_0}$ be the isometries moving in a compact set and assume that
$Y = \pi(g_0\cdot\cB_{\BH^2})$.
Let $\chi \in C_c^\infty(\BH^2)$ be a non-negative real-valued cutoff function supported in $\cB_{\BH^2}$.
We may think of $\chi \psi|_{Y}$ as a function on $\BH^2$ by
\begin{align*}
    \chi \psi|_Y(x): =\chi(x)\psi(\pi(g_0\cdot x)) ,\quad\text{ for }x\in\BH^2.
\end{align*}
If $\phi\in C^\infty(\BH^2)$, we make the following convention
\begin{align*}
    \langle \psi,\chi\phi\rangle := \langle \chi\psi|_Y,\phi\rangle = \int_{\BH^2} \chi(x)\psi(\pi(g_0\cdot x))\overline{\phi(x)} dx.
\end{align*}
If we let
\begin{align*}
        \| \chi \psi|_{Y} \|_{L^2(\BH^2)}^2 = \langle  \psi, \chi\cdot \chi\psi|_{Y} \rangle = \int_{\BH^2}\left|\chi(x)\psi(\pi(g_0\cdot x))  \right|^2 dx,
    \end{align*}
to prove Theorem \ref{main theorem in intro}, it suffices to show that $\| \chi \psi|_{Y} \|_{L^2(\BH^2)}\ll_\epsilon \lambda^{1/4- 1/1220+\epsilon}$ with the implied constant uniform in $g_0$.


Let $\epsilon'>0$ be a small constant satisfying $\epsilon'\leq10^{-6}$.
Let $\beta$ be a parameter satisfying $\lambda^{\epsilon'}\leq\beta\leq\lambda$. We define $\tilde{H}_\beta$ to be the subspace $L^2([\lambda-\beta,\lambda+\beta]\times B)$ in $L^2(\BR_{>0} \times B)$.
Define $H_\beta = \cF^{-1}(\tilde{H}_\beta)$, which is a closed subspace in $L^2(\BH^2)$.
Let $\Pi_\beta$ be the orthogonal projection from $L^2(\BH^2)$ onto $H_\beta$.
Let $H_\beta^\perp$ be the orthogonal complement of $H_\beta$ in $L^2(\BH^2)$,
and let $\tilde{H}_\beta^\perp$ be the orthogonal complement of $\tilde{H}_\beta$ in $L^2(\BR_{>0} \times B)$.
It is clear that $\tilde{H}_\beta^\perp = \cF(H_\beta^\perp)$.
To prove Theorem \ref{main theorem in intro}, we shall bound $\Pi_\beta (\chi \psi|_Y)$ and $(1-\Pi_\beta) (\chi \psi|_Y)$ separately.

\begin{proposition}\label{bound 1-Pi}
    Let $\lambda^{\epsilon'}\leq\beta\leq\lambda$. For any $\phi \in \cS(\BH^2) \cap H_\beta^\perp$, we have
    \begin{align*}
        \langle\psi,\chi\phi\rangle \ll_{\epsilon'} \lambda^{1/4}\beta^{-1/4}\|\phi\|_{L^2(\BH^2)}.
    \end{align*}
    Therefore, we have
    \begin{align*}
       \|(1-\Pi_\beta) (\chi\psi|_Y)\|_{L^2(\BH^2)} \ll_{\epsilon'}  \lambda^{1/4}\beta^{-1/4}. 
    \end{align*}
    
\end{proposition}

\begin{proposition}\label{bound Pi}
    Let $\lambda^{\epsilon'}\leq\beta\leq\lambda^{1/25}$. For any $\phi \in \cS(\BH^2) \cap H_\beta$, we have
    \begin{align*}
        \langle\psi,\chi\phi\rangle \ll_{\epsilon',\epsilon}\lambda^{1/4-1/1120+\epsilon}\beta^{25/1120}\|\phi\|_{L^2(\BH^2)}.
    \end{align*}
    Therefore, we have
    \begin{align*}
        \|\Pi_\beta(\chi\psi|_Y)\|_{L^2(\BH^2)} \ll_{\epsilon',\epsilon} \lambda^{1/4-1/1120+\epsilon}\beta^{25/1120}.
    \end{align*}
\end{proposition}

The proofs of these propositions will be finished in Sections \ref{section of bound away from the spectrum} and \ref{section of near spectrum}. Then Theorem \ref{main theorem in intro} follows from combining the above two propositions with $\beta = \lambda^{1/305}$.

\subsection{Amplification inequalities}\label{section of amplification}
We fix a real-valued function $h\in C^\infty(\BR)$ of Paley-Wiener type that is nonnegative and satisfies $h(0)=1$. Define $h_\lambda^0(s) = h(s-\lambda)+h(-s-\lambda)$, and let $k_\lambda^0$ be the $K_0$-bi-invariant function on $\BH^3$ with Harish-Chandra transform $h^0_\lambda$.
The Paley-Wiener theorem implies that $k_\lambda^0$ is of compact support that may be chosen arbitrarily small. Define $k_\lambda = k_\lambda^0*k_\lambda^0$, which has Harish-Chandra transform $h_\lambda = (h_\lambda^0)^2$. If $g\in G_{0}$ and $\phi\in \cS(\BH^2)$, we define
\begin{align*}
    I(\lambda,\phi,g) = \iint_{\BH^2}\overline{\chi(x_1)\phi(x_1)}\chi(x_2)\phi(x_2) k_\lambda (x_1^{-1}gx_2) dx_1 dx_2.
\end{align*}
We always assume that the supports of $\chi$ and $k_\lambda$ are small enough that $I(\lambda,\phi,g)=0$ unless $d(g,e)\leq 1$, and denote this compact subset by $\cD_0\subset G_{0}$. 
The main inequality for the hyperbolic surface restriction we shall use is the following.
\begin{proposition}\label{amplification inequality}
    Suppose $\cT\in\cH^S
    $ and $\phi\in \cS(\BH^2)$. We have
    \begin{align}\label{amplification inequality eqn}
        \left| \langle \cT\psi,\chi\phi\rangle\right|^2\ll \sum_{\gamma\in \bG(F)} \left| (\cT*\cT^*)(\gamma)I(\lambda,\phi,g_0^{-1}\gamma g_0) \right|.
    \end{align}
\end{proposition}
\begin{proof}
    Consider the function 
    \begin{align*}
        K(x,y)= \sum_{\gamma\in\bG(F)} k_\infty (\cT*\cT^*)(x^{-1}\gamma y)
    \end{align*}
    on $\bG(F)\backslash\bG(\BA)\times\bG(F)\backslash\bG(\BA)$, where $k_\infty$ is a compactly supported and $K_\infty$-bi-invariant function on $G_\infty$ defined by $k_\infty(x_\infty) = k_\lambda(x_{v_0})$. The spectral decomposition of $L^2(X)$ is
    \begin{align*}
        L^2(X) = \bigoplus_i \BC\psi_i,
    \end{align*}
    where $\psi_i$'s are Hecke-Maass forms on $X$ with spectral parameters $\lambda_i$, which form an orthonormal basis of $L^2(X)$ and $\psi$ is one of them. Then by \cite{selberg1956harmonic}, the integral operator acts on Hecke-Maass forms $\psi_i$ as
    \begin{align*}
        \int_{\bG(F)\backslash\bG(\BA)} K(x,y)\psi_i(y) dy &= \int_{\bG(\BA)} k_\infty\cT*\cT^*(x^{-1}y)\psi_i(y)dy = h_\lambda(\lambda_i) {\cT^*} \cT\psi_i(x).
    \end{align*}
    Hence, $K(x,y)$ has a spectral expansion
    \begin{align*}
        K(x,y) = \sum_i h_\lambda(\lambda_i)\cT\psi_i(x)\overline{\cT\psi_i(y)}.
    \end{align*}
    If we integrate it against $\overline{\chi\phi}\times \chi\phi$ on $g_0\BH^2 \times g_0\BH^2$, we obtain
    \begin{align*}
         \sum_i h_\lambda(\lambda_i)\left| \langle \cT\psi,\chi\phi\rangle\right|^2&=\iint_{\BH^2}\overline{\chi(x_1)\phi(x_1)}\chi(x_2)\phi(x_2)K(g_0x_1,g_0x_2) dx_1dx_2 \\ &=\sum_{\gamma\in \bG(F)} (\cT*\cT^*)(\gamma)\iint_{\BH^2}\overline{\chi(x_1)\phi(x_1)}\chi(x_2)\phi(x_2)k_\lambda(x_1^{-1}g_0^{-1}\gamma g_0x_2) dx_1dx_2.
    \end{align*}
    Since we have $h_\lambda(\lambda_i)\geq 0 $ for all $i$, dropping all terms but $\psi$ completes the proof.
\end{proof}

\subsection{The local bound}\label{subsec: trivial bd}

In particular, we can choose $\cT \in \cH_f$  to be the characteristic function of a
sufficiently small open subgroup of $\bG(\BA_f)$ so that only the identity element $\gamma=e$ will make a nonzero contribution to the sum in (\ref{amplification inequality eqn}). This gives
\begin{align}\label{amplification inequality trivial case}
    \left| \langle \psi,\chi\phi\rangle\right|^2\ll\left| I(\lambda,\phi,e) \right|.
\end{align}
Here
\begin{align*}
    I(\lambda,\phi,e) &= \iint_{\BH^2}\overline{\chi(x_1)\phi(x_1)}\chi(x_2)\phi(x_2) k_\lambda (x_1^{-1}x_2) dx_1 dx_2\\
    &=\iint_{\BH^2}\overline{\chi(x_1)\phi(x_1)}\chi(x_2)\phi(x_2) k_\lambda (x_2^{-1}x_1) dx_1 dx_2\\
    &= \int_{\BH^2}   \left(   \int_{\BH^2}\chi\phi(x_2)  k_\lambda  (x_2^{-1}x_1)      dx_2   \right)    \overline{\chi\phi(x_1)}   dx_1\\
    &=\langle \chi\phi \times k_\lambda|_{\BH^2}, \chi\phi \rangle.
\end{align*}
For notation simplicity, we shall still denote $k_\lambda|_{\BH^2}$ by $k_\lambda$ when there is no ambiguity.
We apply the Plancherel formula to obtain
\begin{align}\label{Ilambda e as convolution}
    I(\lambda,\phi,e) =\langle \chi\phi \times k_\lambda, \chi\phi \rangle= \langle \cF( \chi\phi \times k_\lambda) , \tilde{\chi\phi}\rangle = \langle\tilde{\chi\phi}\cdot \tilde{k_\lambda},  \tilde{\chi\phi}\rangle \leq \sup_{s\in\BR}| \tilde{k_\lambda}(s)|\cdot \|\chi\phi\|_{L^2(\BH^2)}^2.
\end{align}
Then the estimate \eqref{ineq all t} of $\tilde{k_\lambda}$ in next section will prove the local bound $\langle\psi,\chi\phi\rangle \ll\lambda^{1/4}\|\phi\|_{L^2(\BH^2)}$.


\section{Bounds away from the spectrum}\label{section of bound away from the spectrum}

\subsection{Proof of Proposition \ref{bound 1-Pi}}
We define the intervals
\begin{align*}
    I^\pm_\delta =  \pm [\lambda-\delta,\lambda+\delta],\quad\text{ and }\quad I_\delta =  I^+_\delta  \cup  I^-_\delta ,
\end{align*}
for $0<\delta\leq\lambda$.
Let $\tau_1\in C_c^\infty(\BR)$ be an even function satisfying
\begin{itemize}
    \item $0\leq \tau_1\leq 1$;
    \item $\tau_1 \equiv 1$ on $I_{\beta/4}$;
    \item $\supp(\tau_1) \subset I_{\beta/2}$.
\end{itemize}
For $\phi\in\cS(\BH^2)\cap H_\beta^\perp$, we define $\chi\phi = \phi_1 + \phi_2$ so that
\begin{align}\label{eq: defn for decompose phi to 12}
    \tilde{\phi_1}(s,b) = \tau_1(s)\tilde{\chi\phi}(s,b),\quad\text{ and }\quad \tilde{\phi_2}(s,b) =( 1-\tau_1(s))\tilde{\chi\phi}(s,b).
\end{align}
It is clear that $\tilde{\phi_1},\tilde{\phi_2} \in \cS(\BR\times B)_W$, and thus $\phi_1,\phi_2\in\cS(\BH^2)$ are well-defined.
We want to show that the main term of this decomposition is $\phi_2$.
Roughly, multiplying a cutoff $\chi$ to $\phi$ will not change the support of its Helgason transform too much. 
We may first calculate the gradients of the Iwasawa $A$-projection that will appear in the phase function of the Helgason transform. Recall that a boundary element $b_\theta$ is represented by the element $ \begin{pmatrix} \cos\theta/2 & \sin\theta/2\\ -\sin\theta/2& \cos\theta/2 \end{pmatrix}$ in $\SO(2)$. 
Let $b_\theta\in B$ and $x,t\in\BR$. 
 By a direct calculation using the fractional linear action on the upper half plane model (\ref{Poincare}), we may obtain 
    \begin{align}\label{Iwasawa A for bna}
    \begin{split}
         A(b_\theta n(x) a(t)) &= t - \log\left(\cos^2(\theta/2)+(x^2+e^{2t}) \sin^2(\theta/2)- 2x \sin(\theta/2)\cos(\theta/2)\right).
    \end{split}
    \end{align}

\begin{lemma}\label{lemma for A(bna) when b is large}
    Given $C>0$ and $0<\delta<\pi$, there is a constant $c_1 > 0$ only depending on $C$ and $\delta$ so that
    \begin{align*}
        \frac{\partial}{\partial t}A(b_\theta n(x) a(t))\leq 1-c_1,
    \end{align*}
    for $x,t\in (-C,C)$ and $\theta\notin (-\delta,\delta) \mod 2\pi$.
\end{lemma}
\begin{proof}
    By (\ref{Iwasawa A for bna}), we have
    \begin{align*}
        \frac{\partial}{\partial t}A(b_\theta n(x) a(t)) &= 1 - \frac{2e^{2t}\sin^2(\theta/2)}{\cos^2(\theta/2)+(x^2+e^{2t}) \sin^2(\theta/2)- 2x \sin(\theta/2)\cos(\theta/2)}\\
        &= 1-\frac{2e^{2t}\sin^2(\theta/2)}{e^{2t}\sin^2(\theta/2)+ (x\sin(\theta/2)-\cos(\theta/2))^2}\\
        &=1-\frac{2}{1+ (x-\cot(\theta/2))^2e^{-2t}}.
    \end{align*}
    We can take $c_1$ to be the infimum of $\frac{2}{1+ (x-\cot(\theta/2))^2e^{-2t}}$ for $x,t\in (-C,C)$ and $\theta\notin (-\delta,\delta) \mod 2\pi$.
\end{proof}

\begin{lemma}\label{lemma for A(bna) when b is small}
    Given a constant $C>0$, if $x,t\in (-C,C)$, then
    \begin{align}\label{Taylor estimate of Iwasawa A for bna}
        A(b_\theta n(x) a(t)) = t + x\theta + \frac{x^2-e^{2t}+1}{4}\theta^2 + O(\theta^3).
    \end{align}
    Moreover, we have
    \begin{align}
         \frac{\partial}{\partial x}A(b_\theta n(x) a(t)) &= \theta + \frac{x}{2}\theta^2 + O(\theta^3),\label{partial derivative wrt x for A(bna)}\\
         \frac{\partial}{\partial t}A(b_\theta n(x) a(t))& = 1 - \frac{e^{2t}}{2}\theta^2 + O(\theta^3).\label{partial derivative wrt t for A(bna)}
    \end{align}
    Here the implied constants depend only on $C$. 
\end{lemma}
\begin{proof}
 Since the Taylor approximations give
    \begin{align*}
        &\cos^2(\theta/2) = 1-{\theta^2}/{4} + O(\theta^4),\\
        &\sin^2(\theta/2) = {\theta^2}/{4} + O(\theta^4),\\
        &2\sin(\theta/2)\cos(\theta/2)  = \theta + O(\theta^3),
    \end{align*}
    we have
    \begin{align*}
        &\log\left(\cos^2(\theta/2)+(x^2+e^{2t}) \sin^2(\theta/2)- 2x \sin(\theta/2)\cos(\theta/2)\right)\\
        =&\log\left( 1-\theta^2/4 + (x^2+e^{2t})\theta^2/4  - x\theta + O(\theta^3)\right)\\
        =& \left(-x\theta + (x^2+e^{2t}-1)\theta^2/4 \right) - \left(-x\theta + (x^2+e^{2t}-1)\theta^2/4 \right)^2/2 + O(\theta^3)\\
        =&-x\theta + (x^2+e^{2t}-1)\theta^2/4 - x^2\theta^2/2 + O(\theta^3).
    \end{align*}
    Hence, we obtain \eqref{Taylor estimate of Iwasawa A for bna} from \eqref{Iwasawa A for bna}. Moreover, by the analyticity of elementary functions, we have
    \begin{align*}
         A(b_\theta n(x) a(t)) = t + x\theta + \frac{x^2-e^{2t}+1}{4}\theta^2 + \theta^3\xi(x,t,\theta),
    \end{align*}
    where $\xi$ is analytic. Then we obtain \eqref{partial derivative wrt x for A(bna)} and \eqref{partial derivative wrt t for A(bna)} by taking derivatives.
    
\end{proof}

\begin{proposition}\label{Fourier side decomposition lemma}
    Suppose that $1\leq \beta \leq \lambda$ and $\phi\in\cS(\BH^2)\cap H_\beta^\perp$. Let $\phi_1$ be defined as in \eqref{eq: defn for decompose phi to 12}. Then we have $$\| \phi_1 \|_{L^2(\BH^2)} \ll _{N} \beta^{-N}\|\phi\|_{L^2(\BH^2)}.$$
\end{proposition}
\begin{proof}
    Let $s\in \BR_{>0}$ and $b\in B$, we have
    \begin{align}
         \tilde{\phi_1}(s,b) =& \tau_1(s)\tilde{\chi\phi}(s,b)\notag \\
         =& \tau_1(s) \int_{\BH^2} \chi(x)\phi(x) e^{(-is+1/2) A(x,b)} dx\notag\\
         =&\tau_1(s) \int_{\BH^2} \chi(x)\left( \int_{\BR_{>0}\times B}\tilde{\phi}(r,d) e^{(ir+1/2)A(x,d)} d\mu(r,d)\right) e^{(-is+1/2) A(x,b)} dx\notag\\
         =& \tau_1(s)\int_{\BR_{>0}\times B} \tilde{\phi}(r,d) \left(   \int_{\BH^2}\chi(x)e^{(A(x,d)+A(x,b))/2} \exp(irA(d^{-1}x)-isA(b^{-1}x)) dx\right)d\mu(r,d).\label{inner nonstationary phase intgeral}
    \end{align}
    Because of the supports of $\tau_1$ and $\tilde{\phi}$, we can assume $s\in I_{\beta/2}$ and $r\notin I_\beta$, otherwise the outer integral in  (\ref{inner nonstationary phase intgeral})
    will vanish.
    
    We first assume that $0\leq r \leq \lambda-\beta$.
    By the change of variable $x\mapsto bx$, the integral becomes
    \begin{align*}
       \int_{\BH^2}\chi(bx)e^{(A(bx,d)+A(bx,b))/2} \exp(irA(d^{-1}bx)-isA(x)) dx.
    \end{align*}
    We assume $d^{-1}b = b_\theta$ with $\theta\in\BR / 2\pi\BZ$,
    and write the integral under the Iwasawa coordinates $(x,t)\mapsto n(x)a(t)\cdot o = (x,e^t)\in\BH^2$.
    Then the integral can be written as
    \begin{align}
        &\int_{\BR^2}\chi^\prime (x,t) \exp(irA(b_\theta n(x)a(t))-ist ) dx dt\notag\\
        =&\int_{\BR^2}\chi^\prime (x,t) \exp(-is\varphi(x,t)) dx dt\label{integral of chi prime phase varphi}
    \end{align}
    Here $\chi^\prime$ is a smooth, compactly supported function obtained by combining all of the amplitude factors,
    and
    \begin{align*}
        \varphi(x,t) &= t- \frac{r}{s}A(b_\theta n(x)a(t))
    \end{align*}
    is the phase function. We omit the parameters $\theta$ and $r/s$ when we write the variables of $\chi^\prime$ and $\varphi$ for simplicity. 
    We choose a constant $C>0$ so that $(-C,C)^2$ contains the supports of $\chi^\prime$ for arbitrary $\theta\in\BR/2\pi\BZ$. Let $\delta>0$ be chosen as in Lemma \ref{lemma for A(bna) when b is small}. When $\theta\notin (-\delta,\delta) \mod 2\pi$, Lemma \ref{lemma for A(bna) when b is large} gives a constant $c_1 > 0$ so that
    \begin{align*}
        \frac{\partial}{\partial t}A(b_\theta n(x) a(t))\leq 1-c_1,
    \end{align*}
    for $x,t\in\supp(\chi^\prime)$. Hence,
    \begin{align*}
        \frac{\partial}{\partial t} \varphi(x,t) &= 1 - \frac{r}{s}  \frac{\partial}{\partial t}A(b_\theta n(x) a(t))\\
        &\geq 1 - \frac{r}{s}(1-c_1)\\
        &\geq \min\{1,c_1 \}.
    \end{align*}
    Since $|r/s|\leq 1$, the $n$-th higher derivatives of $\varphi$ are $\ll_n 1$. By integrating by parts,
    \begin{align*}
        \int_{\BR}\chi^\prime (x,t) \exp(-is\varphi(x,t)) dt \ll_N s^{-N}.
    \end{align*}
    After integrating the $x$ variable, the integral (\ref{integral of chi prime phase varphi}) is also $\ll_N s^{-N}$.
    Now we assume $\theta\in (-\delta,\delta) \mod 2\pi$. We first write the phase function as 
     \begin{align*}
         \varphi(x,t) &= \big( t - A(b_\theta n(x)a(t))  \big) + \frac{s-r}{s} A(b_\theta n(x)a(t))\\
         &=\alpha_1 \varphi_1(x,t) + \alpha_2 \varphi_2(x,t),
     \end{align*}
     where $\alpha_1 = \theta$, $\alpha_2 = (s-r)/s$, $\varphi_1(x,t) = \big( t - A(b_\theta n(x)a(t))  \big)/\theta$ and $\varphi_2(x,t) = A(b_\theta n(x)a(t))$.
     Lemma \ref{lemma for A(bna) when b is small} implies that $\varphi_1$ is still analytic and
     \begin{align*}
         \varphi_1(x,t) = -x - \frac{x^2-e^{2t}+1}{4}\theta + O(\theta^2).
     \end{align*}
     By calculating their derivatives or applying Lemma \ref{lemma for A(bna) when b is small}, we obtain the gradients of $\varphi_1$ and $\varphi_2$:
     \begin{align*}
         &\nabla \varphi_1 = \left( \frac{\partial \varphi_1}{\partial x},\frac{\partial \varphi_1}{\partial t} \right) = \left(-1,0   \right) + O(\theta),\\
         &\nabla \varphi_2 = \left( \frac{\partial \varphi_2}{\partial x},\frac{\partial \varphi_2}{\partial t} \right) = \left(0,1   \right) + O(\theta).
     \end{align*}
    Let $\rho = (\alpha_1^2 + \alpha_2^2)^{1/2}$. Then we have
    \begin{align*}
        \nabla\left( \frac{\varphi}{\rho} \right) &= \frac{\alpha_1}{\rho} \nabla \varphi_1 + \frac{\alpha_2}{\rho} \nabla \varphi_2 = (-\alpha_1/\rho,\alpha_2/\rho) + O(\theta).
    \end{align*}
    Hence, the square of the norm is
    \begin{align*}
        \left\|  \nabla\left( \frac{\varphi}{\rho} \right) \right\|^2 = \frac{\alpha_1^2+ \alpha_2^2}{\rho^2} + O(\theta) = 1 + O(\theta).
    \end{align*}
    Without loss of generality, we can take $\delta$ to be sufficiently small so that $ \left\|  \nabla\left( {\varphi}/{\rho} \right) \right\| \geq 1/2$ when $\theta\in (-\delta,\delta)$. 
    It is clear that all the $n$-th higher derivatives of $\varphi_1$ and $\varphi_2$ are $\ll_n 1$. Since $|\alpha_1/\rho|,|\alpha_2/\rho|\leq 1$, all the $n$-th higher derivatives of $\varphi/\rho$ are also $\ll_n 1$. 
    Hence, by integrating by parts, \eqref{integral of chi prime phase varphi} is 
    \begin{align*}
        \int_{\BR^2}\chi^\prime (x,t) \exp\left(-is\rho \frac{\varphi(x,t)}{\rho}\right) dx dt \ll_N (1+s\rho)^{-N}.
    \end{align*}
    Since $s\rho\geq s\alpha_2=s-r$, the intgeral (\ref{integral of chi prime phase varphi}) is $\ll_N (s-r)^{-N} $. To sum up, we have shown that the inner integral in  (\ref{inner nonstationary phase intgeral}) is $ \ll_{N} (s-r)^{-N} $ when $0\leq r \leq \lambda-\beta$.

    If $r\geq\lambda+\beta$, likewise, by the change of variable $x\to d\cdot x$, we write the inner integral in (\ref{inner nonstationary phase intgeral}) as 
    \begin{align*}
        \int_{\BH^2}\chi(dx)e^{(A(dx,d)+A(dx,b))/2} \exp(irA(x)-isA(b^{-1}dx)) dx.
    \end{align*}
    We assume $b^{-1} d =b_\theta$ and write the integral under the Iwasawa coordinates as
    \begin{align*}
        \int_{\BR^2}\chi^\prime (x,t) \exp(ir\varphi(x,t)) dx dt
    \end{align*}
    with $\chi^\prime$ the smooth compactly supported function from all of the amplitude factors and
    \begin{align*}
        \varphi(x,t) &= t- \frac{s}{r}A(b_\theta n(x)a(t)) = \frac{s}{r}\big( t- A(b_\theta n(x)a(t)) \big) + \frac{r-s}{r}t.
    \end{align*}
    Then the same argument shows the oscillatory integral is $\ll_N (r-s)^{-N}$. Applying these bounds with Cauchy–Schwarz inequality to (\ref{inner nonstationary phase intgeral}) gives $|\tilde{\phi_1}(s,b)|\ll_{N}\beta^{-N} \| \phi \|_{L^2(\BH^2)}$. Then the proposition follows from the Plancherel identity
    \begin{align*}
        \| \phi_1 \|_{L^2(\BH^2)}^2 = \int_{I^+_{\beta/2}\times B} |\tilde{\phi_1}(s,b)|^2 d\mu(s,b)\ll_{N}\beta^{-N}  \| \phi\|_{L^2(\BH^2)}^2.
    \end{align*}
\end{proof}

\begin{proof}[Proof of Proposition \ref{bound 1-Pi}]
Suppose that $\phi \in \cS(\BH^2)\cap H_\beta^\perp$.
By (\ref{amplification inequality trivial case}) and (\ref{Ilambda e as convolution}), we have
\begin{align*}
    \left| \langle \psi,\chi\phi\rangle\right|^2&\ll\left| I(\lambda,\phi,e) \right|=\langle\tilde{\chi\phi}\cdot \tilde{k_\lambda},  \tilde{\chi\phi}\rangle.
\end{align*}
Proposition \ref{Fourier side decomposition lemma} shows that $\tilde{\chi\phi} = \tilde{\phi_1}+\tilde{\phi_2} = \tilde{\phi_2} 
 + O_{N}(\beta^{-N} \| \phi\|_{L^2(\BH^2)})$.
Hence,
\begin{align*}
      \left| \langle \psi,\chi\phi\rangle\right|^2&\ll_N  \left(\sup_{t\notin I_{\beta/4}} |\tilde{k_\lambda}(t)| +\beta^{-N}\sup_{s\in\BR} |\tilde{k_\lambda}(s)| \right) \| \phi\|_{L^2(\BH^2)}^2.
\end{align*}
Applying the following Proposition \ref{supnorm of klambda} with the assumption $\lambda^{\epsilon'}\leq\beta\leq\lambda$, we have
\begin{align*}
    \left| \langle \psi,\chi\phi\rangle\right|^2 \ll_{\epsilon',N}  \left(\lambda^{1/2}\beta^{-1/2}  +\lambda^{1/2}\beta^{-N}\right)\| \phi\|_{L^2(\BH^2)}^2\ll \lambda^{1/2}\beta^{-1/2}\| \phi\|_{L^2(\BH^2)}^2,
\end{align*}
which completes the proof.
\end{proof}

It remains to show the following point-wise estimate of $\tilde{k_\lambda}$.

\begin{proposition}\label{supnorm of klambda}
    We have
    \begin{align}\label{ineq all t}
        \sup_{s\in\BR} |\tilde{k_\lambda}(s)| &\ll \lambda^{1/2}.
    \end{align}
    Moreover, if $\lambda^{\epsilon'}\leq\beta\leq\lambda$, we have
    \begin{align}
        \sup_{0\leq |s|\leq \lambda/2}|\tilde{k_\lambda}(s)|&\ll 1, \label{ineq small t}\\
        \sup_{s\notin I_{\beta/4},|s|\geq \lambda/2}|\tilde{k_\lambda}(s)|&\ll_{\epsilon'} \lambda^{1/2}\beta^{-1/2}.\label{ineq big t}
    \end{align}
\end{proposition}
\begin{proof}
    The proofs of \eqref{ineq small t} and \eqref{ineq big t} are the main part of this section. We postpone them and prove \eqref{ineq all t} here by assuming (\ref{ineq small t}) and \eqref{ineq big t}.

    Clearly, \eqref{ineq small t} and \eqref{ineq big t} provide stronger bounds than \eqref{ineq all t} when $s\notin I_{\beta/4}$, so we only need to show the bound in \eqref{ineq all t} provided $s\in I_{\beta/4}$. Using the integral formula for $\BH^2$ under the polar coordinates, we have
    \begin{align*}
        \tilde{k_\lambda}(s) &= \int_{\BH^2} k_\lambda(x)\varphi_{-s}^{\BH^2}(x) dx\\
        &=c\int_0^\infty k_\lambda(a(t))\varphi_{-s}^{\BH^2}(a(t)) \sinh(t) dt,
    \end{align*}
    for some normalizing constant $c>0$. Because $k_\lambda$ has a compact support near the origin, there exists a smooth compactly supported function $b_1:\BR\to\BR$ so that $b_1(t)k_\lambda(a(t)) = k_\lambda(a(t))$ for all $t\in\BR$. Without loss of generality, we assume $b_1\equiv 1$ on $[-1,1]$ and vanishes outside $[-2,2]$. Hence,
    \begin{align}\label{bound using spherical function bound}
        \tilde{k_\lambda}(s) =c\int_0^2 k_\lambda(a(t))b_1(t)\varphi_{-s}^{\BH^2}(a(t)) \sinh(t) dt.
    \end{align}
    Using \cite[Lemma 2.8]
    {marshall2016p}, we have
    \begin{align*}
        k_\lambda(a(t))\ll\lambda^2(1+\lambda|t|)^{-1}.
    \end{align*}
    By \cite[Theorem 1.3]{marshall2016p}, we have
    \begin{align*}
        b_1(t)\varphi_{-s}^{\BH^2}(a(t)) \ll (1+|st|)^{-1/2}.
    \end{align*}
    Therefore, for $s\in I_{\beta/4}$, the above two bounds with the integral formula \eqref{bound using spherical function bound} imply the desired bound.
\end{proof}

\subsection{Bound of \texorpdfstring{$\tilde{k_\lambda}(s)$}{k(s)} for small \texorpdfstring{$s$}{s}}

Now we begin the proof of the inequality \eqref{ineq small t} in Proposition \ref{supnorm of klambda}.
Suppose $0\leq |s|\leq \lambda/2$. 
Since $\tilde{k_\lambda}$ is even, we assume $0\leq s\leq \lambda/2$.
Let $b_2\in C_c^\infty(\BH^2)$ be a smooth cutoff function supported in a neighborhood of origin satisfying
\begin{itemize}
    \item $b_2$ is $\SO(2)$-bi-invariant;
    \item $b_2(x)k_\lambda(x) = k_\lambda(x)$ for all $x\in\BH^2$.
\end{itemize}
We unfold $\tilde{k_\lambda}(t)$ by the inverse Harish-Chandra transform and then apply the integral formulas \eqref{eq:HC int for H2} and \eqref{eq:HC int for H3} to the spherical functions. Then we have
\begin{align}
    \tilde{k_\lambda}(s) & = \int_{\BH^2} k_\lambda(x) \varphi_{-s}^{\BH^2}(x) dx\notag\\
    &= \int_{\BH^2} b_2(x) k_\lambda(x) \varphi_{-s}^{\BH^2}(x) dx\notag\\
    &= \int_{\BH^2} b_2(x) \left(   \int_0^\infty h_\lambda(r)\varphi_r^{\BH^3}(x)    d\nu(r)        \right)\varphi_{-s}^{\BH^2}(x) dx\notag\\
    &=\int_0^\infty \left(  \int_{\BH^2} b_2(x) \varphi_{r}^{\BH^3}(x) \varphi_{-s}^{\BH^2}(x) dx  \right)   h_\lambda(r)d\nu(r) \notag\\
    &=\int_0^\infty \left(  \int_{\BH^2}\int_{\SO(2)} b_2(x) \varphi_{r}^{\BH^3}(x) e^{(-is+1/2)A(kx)} dk dx\right)   h_\lambda(r)d\nu(r)\notag\\
    &=\int_0^\infty \left(  \int_{\SO(2)}\int_{\BH^2} b_2(k^{-1}x) \varphi_{r}^{\BH^3}(k^{-1}x) e^{(-is+1/2)A(x)}dx dk \right)   h_\lambda(r)d\nu(r)\notag\\
    &=\int_0^\infty \left(  \int_{\SO(2)}\int_{\BH^2} b_2(x) \varphi_{r}^{\BH^3}(x) e^{(-is+1/2)A(x)}dx dk \right)   h_\lambda(r)d\nu(r)\notag\\
    &=\int_0^\infty   \left(  \int_{\BH^2} b_2(x) \varphi_{r}^{\BH^3}(x) e^{(-is+1/2)A(x)}dx  \right)   h_\lambda(r)d\nu(r)\notag\\
    &=\int_0^\infty    \left(  \int_{\BH^2} \int_{K_0} b_2(x) e^{(ir+1)A(kx)} e^{(-is+1/2)A(x)}dx dk \right)   h_\lambda(r)d\nu(r).\label{inner oscillatory int}
\end{align}
We denote the inner integral inside (\ref{inner oscillatory int}) by $J(r;\rho)$ with $\rho = s/r$, so
\begin{align*}
    J(r;\rho) &=  \int_{\BH^2} \int_{K_0} b_2(x) e^{A(kx)+A(x)/2} \exp(irA(kx)-isA(x))  dx dk\\
    &= \int_{\BR}\int_{\BR} \int_{K_0} b_2(n(x)a(t)\cdot o)e^{A(kn(x)a(t))+t/2} \exp(ir(A(kn(x)a(t))-\rho t)) e^{-t} dx dt dk\\
     &= \int_{\BR}\int_{\BR} \int_{M\backslash K_0} b_2(n(x)a(t)\cdot o)e^{A(kn(x)a(t))+t/2} \exp(ir(A(kn(x)a(t))-\rho t)) e^{-t} dx dt dk.
\end{align*}
 Recall that the factor $e^{-t}$ is from the hyperbolic metric, and $M$ is the centralizer of $A$ in $K_0$.
Let $b_3(x,t,k) =b_2(n(x)a(t)\cdot o)e^{A(kn(x)a(t))+t/2} e^{-t}$. It is clear that $b_3\in C_c^\infty (\BR^2\times M\backslash K_0)$. Hence,
\begin{align*}
    J(r;\rho) = \int_{\BR}\int_{\BR} \int_{M\backslash K_0} b_3(x,t,k) \exp(irF(x,t,k;\rho))  dx dt dk,
\end{align*}
where
\begin{align*}
    F(x,t,k;\rho) = A(kn(x)a(t)) - \rho t
\end{align*}
is the phase function of the oscillatory integral. 
Since $h_\lambda$ is a Paley-Wiener function concentrating near $\pm\lambda$ and $d\nu(r) = c\cdot r^2dr$ for some constant $c\neq0$, it may be seen that the bound  (\ref{ineq small t}) will follow from (\ref{inner oscillatory int}) and the following estimate of $ J(r;\rho)$.
\begin{proposition}
    Fix a constant $0<\delta<1$. We have $J(r;\rho)\ll_\delta r^{-2}$ if $\rho\in [0,\delta]$.
\end{proposition}

We will use the method of stationary phase to bound the oscillatory integral $J(r;\rho)$, so we first need to determine the critical points of $F$.
 Using the formula \eqref{Poincare}, we get the following formula about the Iwasawa projection $A$.

\begin{lemma}\label{lemma A}
    Given $g\in G_0$ with Iwasawa decomposition $g=nak$ where $a=a(t_0)\in A$ and $k =\begin{pmatrix} \alpha&\beta\\-\bar{\beta}&\bar{\alpha}\end{pmatrix}\in K_0$, we have, for $z\in\BC$ and $t\in\BR$,
    \begin{align}\label{iwasawa height}
        A(gn(z)a(t)) = t_0 + t - \log\left(|\alpha|^2+|\beta|^2(|z|^2+e^{2t}) - \alpha\bar{\beta}z-\bar{\alpha}\beta \bar{z}  \right).
    \end{align}
    By taking derivatives, for $x,t\in\BR$, we obtain
    \begin{align}\label{derivative of A about n}
        \frac{\partial}{\partial x} A(gn(x))|_{x=0} = \alpha\bar{\beta} + \bar{\alpha}\beta
    \end{align}
    and
    \begin{align}\label{derivative of A about a}
        \frac{\partial}{\partial t} A(ga(t))|_{t=0} = |\alpha|^2-|\beta|^2.
    \end{align}
\end{lemma}

For simplicity, we let $\Psi$ and $\Theta$ be the functions on $K$ given by sending $k =\begin{pmatrix} \alpha&\beta\\-\bar{\beta}&\bar{\alpha}\end{pmatrix}\in K_0$ to $\alpha\bar{\beta} + \bar{\alpha}\beta$ and $|\alpha|^2-|\beta|^2$ respectively. Therefore (\ref{derivative of A about n}) and (\ref{derivative of A about a}) can be read as $\frac{\partial}{\partial x} A(gn(x))|_{x=0} = \Psi(\kappa(g))$ and $\frac{\partial}{\partial t} A(ga(t))|_{t=0} = \Theta(\kappa(g))$.
Moreover, it may be seen that $\Psi$ and $\Theta$ are left invariant under $M$. Hence, we will also treat them as functions on $M\backslash K_0$.

For doing calculations, we choose the following basis for $\fk$:
\begin{align}\label{basis for lie algebra k}
    X_1 = \begin{pmatrix}0&i\\i&0\end{pmatrix}, \, X_2 = \begin{pmatrix}0&-1\\1&0\end{pmatrix}, \, X_3 = \begin{pmatrix}i&0\\0&-i\end{pmatrix}.
\end{align}
\begin{proposition}\label{critical for k}
    An element $g\in G_0$ lies in $AK_0$ if and only if $\frac{\partial}{\partial t} A(\exp(tX)g)|_{t=0} = 0$ for arbitrary $X\in\fk$.
\end{proposition}
\begin{proof}
    Without loss of generality, assume $g = a(t_0)\in A$. We only have to check $\frac{\partial}{\partial t} A(\exp(tX)g)|_{t=0} = 0$ holds for a basis of $\fk$.  By (\ref{iwasawa height}), we have $A(\exp(tX_1)a(t_0)) =A(\exp(tX_2)a(t_0)) = t_0 -\log(\cos^2t + e^{2t}\sin^2t)$ and $A(\exp(tX_3)a(t_0))=t_0 $, so $\frac{\partial}{\partial t} A(\exp(tX_i)g)|_{t=0} = 0$ with $i=1,2,3$.
    
    Conversely, we suppose $g\notin AK$ but $\frac{\partial}{\partial t} A(\exp(tX)g)|_{t=0} = 0$ for arbitrary $X\in\fk$. We can assume $g = n(z_0)a(t_0)$ with $z_0\neq 0$. By (\ref{iwasawa height}), we have
    \begin{align*}
        \frac{\partial}{\partial t} A(\exp(tX_2)g)|_{t=0} &=  \frac{\partial}{\partial t}\left[ t_0  - \log\left(\cos^2(t)+\sin^2(t)(|z_0|^2+e^{2t_0}) + \cos (t)\sin (t)(z_0+\bar{z}_0) \right)\right]|_{t=0}\\
        &=z_0 + \bar{z}_0 = 0,
    \end{align*}
    implying $z_0 = iy_0\neq 0$ is purely imaginary. But
    \begin{align*}
        \frac{\partial}{\partial t} A(\exp(tX_1)g)|_{t=0} &=  \frac{\partial}{\partial t}\left[ t_0  - \log\left(\cos^2(t)+\sin^2(t)(y_0^2+e^{2t_0}) + 2\cos (t)\sin (t)y_0 \right)\right]|_{t=0}\\
        &=2y_0 = 0
    \end{align*}
    gives a contradiction.
\end{proof}

\begin{proposition}\label{critical point of F}
    For any  fixed $0<\delta<1$ and $\rho\in[0,\delta]$, the only critical point of $F(\cdot,\cdot,\cdot;\rho)$ (with respect to the variables (x,t,k)) is  $(0,0,k^\prime)$ with $k^\prime = Mk_\rho$. Here
    $$k_\rho=\frac{1}{\sqrt{2}} \begin{pmatrix}\sqrt{1+\rho}& i\sqrt{1-\rho} \\ i\sqrt{1-\rho} &\sqrt{1+\rho}\end{pmatrix}.$$
\end{proposition}
\begin{proof}
    We let $l$ be the image of $A$ in $\BH^3$, which is a vertical geodesic through the origin $o$.
    Suppose that $(x^\prime,t^\prime,k^\prime)\in \BR^2\times M\backslash K_0$ is a critical point of $F$. By Proposition \ref{critical for k}, taking the differentiation with respect to the $k$ variable equal to zero is equivalent to
    \begin{align*}
        k^\prime n(x^\prime) a(t^\prime) = a^\prime u^\prime
    \end{align*}
    for some $a^\prime\in A$ and $u^\prime\in M\backslash K_0$. Note that this is well-defined because $M$ centralizes $A$.
    Now we take the derivatives with respect to the $x,t$ variables. Applying (\ref{derivative of A about n}) and (\ref{derivative of A about a}), we obtain
    \begin{align*}
        \frac{\partial}{\partial x} F(x^\prime+x,t^\prime,k^\prime)|_{x=0} &= \frac{\partial}{\partial x} \left( A(k^\prime n(x^\prime+x)a(t^\prime))-\rho t^\prime\right)|_{x=0} \\
        &= \frac{\partial}{\partial x} A(k^\prime n(x^\prime)a(t^\prime)n(e^{-t^\prime}x))|_{x=0}\\
        &=\frac{\partial}{\partial x}  A(a^\prime u^\prime n(e^{-t^\prime}x))|_{x=0}\\
        &=e^{-t^\prime}\Psi(u^\prime),
    \end{align*}
    and
    \begin{align*}
        \frac{\partial}{\partial t} F(x^\prime,t^\prime+t,k^\prime)|_{t=0} &= \frac{\partial}{\partial t} \left( A(k^\prime n(x^\prime)a(t^\prime+t))-\rho (t^\prime+t)\right)|_{t=0} \\
        &=\frac{\partial}{\partial t}  A(k^\prime n(x^\prime)a(t^\prime)a(t))|_{t=0} - \rho\\
        &=\frac{\partial}{\partial t}  A(a^\prime u^\prime a(t))|_{t=0} - \rho\\
        &=\Theta(u^\prime) - \rho.
    \end{align*}
    Therefore, $\Psi(u^\prime) = 0$ and $\Theta(u^\prime) = \rho$. Direct computation shows the only solution is $u^\prime = Mk_\rho$. Recall that we have $k^\prime n(x^\prime) a(t^\prime) = a^\prime u^\prime$. We choose a representative $k^\prime\in K_0$ so that the relation $k^\prime n(x^\prime) a(t^\prime) = a^\prime k_\rho$ holds as group elements.
    We claim that
    \begin{align*}
         k^\prime\BH^2 = k^\prime n(x^\prime) a(t^\prime) \BH^2 = a^\prime k_\rho\BH^2.
    \end{align*}
    is a hemisphere in $\BH^3$.
    Suppose otherwise, i.e., $k^\prime\BH^2$ is a vertical plane in $\BH^3$. Then $k_\rho\BH^2$ is also a vertical plane, containing $l$. 
    Therefore, $k_\rho\in M\SO(2)$ and we may find $\theta_1,\theta_2\in\BR$ so that
    \begin{align*}
        k_\rho = \begin{pmatrix}
            e^{i\theta_1}&\\& e^{-i\theta_1}
        \end{pmatrix}
        \begin{pmatrix}
            \cos(\theta_2)&\sin(\theta_2)\\-\sin(\theta_2)&\cos(\theta_2)
        \end{pmatrix}=\begin{pmatrix}
            e^{i\theta_1}\cos(\theta_2)&e^{i\theta_1}\sin(\theta_2)\\-e^{-i\theta_1}\sin(\theta_2)&e^{-i\theta_1}\cos(\theta_2)
        \end{pmatrix}.
    \end{align*}
    Therefore, we have
    \begin{align*}
        \frac{\sqrt{1+\rho}}{\sqrt{2}}=e^{i\theta_1}\cos(\theta_2)=e^{-i\theta_1}\cos(\theta_2),\quad \frac{i\sqrt{1-\rho}}{\sqrt{2}}=e^{i\theta_1}\sin(\theta_2)=-e^{-i\theta_1}\sin(\theta_2),
    \end{align*}
    which is a contradiction because no such $\theta_1,\theta_2$ satisfy the above equations.
    Hence, $k^\prime\BH^2$ is a hemisphere, and the intersection of $l$ with $k^\prime\BH^2$ is a single point, which is the origin $o$.
    Acting $k^\prime n(x^\prime) a(t^\prime) = a^\prime k_\rho$ at $o$, it may be seen that the point 
    \begin{align*}
        k^\prime(x^\prime,e^{t^\prime}) = k^\prime n(x^\prime)a(t^\prime)\cdot o= a^\prime k_\rho \cdot o = a^\prime  \cdot o 
    \end{align*}
    lies on both $k^\prime \BH^2$ and $l$. Hence, $k^\prime n(x^\prime) a(t^\prime)\cdot o=o$, and so $x^\prime = t^\prime = 0$.

    Conversely, it is easy to see $(0,0,Mk_\rho)$ is a critical point of $F$.
\end{proof}

The tangent space at $M$ in the quotient space $M\backslash K_0$ can be viewed as a two-dimensional subspace of $\fk$ that is transversal to the Lie algebra of $M$. We will choose the subspace of $\fk$ spanned by the vectors  $X_1$ and $X_2$.
By the non-stationary phase, the integral in $J(s;\rho)$ over the region away from the critical points decays rapidly, so we only need to estimate $J(s;\rho)$ by the integrals over some small neighborhoods of the critical points so that we can represent the $k$ variable in $M\backslash K_0$ via the exponential map. To sum up, we have
\begin{align*}
    J(r;\rho) = \tilde{J}(r;\rho) +  O_N(r^{-N})
\end{align*}
with
\begin{align*}
     \tilde{J}(r;\rho)= \iiiint_\BR \tilde{b_3} (x,t,s_1,s_2) \exp(ir\tilde{F}(x,t,s_1,s_2;\rho)) dxdtds_1ds_2.
\end{align*}
Here $\tilde{b_3} (x,t,s_1,s_2) $ is a smooth function supported in a small neighborhood of $(0,0,0,0)$  and 
\begin{align*}
    \tilde{F}(x,t,s_1,s_2;\rho)  = A(\exp(s_1X_1+s_2X_2)k_\rho n(x)a(t)) - \rho t.
\end{align*}
is the phase functions of the oscillatory integral $\tilde{J}(r;\rho)$, with the only critical points at $(0,0,0,0)$. It remains to calculate the Hessian of $\tilde{F}$ at $(0,0,0,0)$.

\begin{lemma}\label{Iwasawa k}
    For $g= \begin{pmatrix}a &b \\c &d \end{pmatrix}\in G_0$, its Iwasawa projection to $K_0$ is given by the formula
    \begin{align*}
        \kappa(g) = \frac{1}{\sqrt{|c|^2+|d|^2}}\begin{pmatrix}{\bar{d}} &-{\bar{c}} \\{c}&{d} \end{pmatrix},
    \end{align*}
    that is, $g\in NA\kappa(g)$.
\end{lemma}
\begin{proof}
    Directly follows from computing the action of $g$ on $\BH^3$.
\end{proof}

\begin{lemma}\label{gradient of psi theta}
    Given $k=\begin{pmatrix}\alpha&{\beta}\\-\bar{\beta}&\bar{\alpha}\end{pmatrix}\in K_0$, the gradients of $\Psi$ and $\Theta$ at $k$ are
    \begin{align*}
        &\nabla_{K_0}\Psi(k):=\left. \left(   \frac{\partial}{\partial t}\Psi(\exp(tX_1)k) ,\frac{\partial}{\partial t}\Psi(\exp(tX_2)k) ,\frac{\partial}{\partial t}\Psi(\exp(tX_3)k)    \right)\right|_{t=0}\\
        =&\left(-i(\alpha^2-\bar{\alpha}^2-\beta^2+\bar{\beta}^2), -(\alpha^2+\bar{\alpha}^2-\beta^2-\bar{\beta}^2),0 \right)
    \end{align*}
    and
    \begin{align*}
        \nabla_{K_0}\Theta(k)=\left( 2i(\alpha\beta-\bar{\alpha}\bar{\beta}), 2(\alpha\beta+\bar{\alpha}\bar{\beta}), 0\right).
    \end{align*}
    In particular,
    \begin{align*}
        &\nabla\Psi(k_\rho)
        =\left(0, -2,0 \right)
    \end{align*}
    and
    \begin{align*}
        \nabla\Theta(k_\rho)=\left( -2\sqrt{1-\rho^2}, 0, 0\right).
    \end{align*}
\end{lemma}
\begin{proof}
    Since
    \begin{align*}
        \exp(tX_1)k =  \begin{pmatrix}\alpha \cos t-i\bar{\beta}\sin t& {\beta}\cos t +i\bar{\alpha}\sin t \\  -\bar{\beta}\cos t +i\alpha\sin t& \bar{\alpha}\cos t+i\beta\sin t\end{pmatrix},
    \end{align*}
    we have
    \begin{align*}
        \Psi(\exp(tX_1)k) &= (\alpha \cos t-i\bar{\beta}\sin t)(\bar{\beta}\cos t -i\alpha\sin t)+(\bar{\alpha}\cos t+i\beta\sin t)({\beta}\cos t +i\bar{\alpha}\sin t) \\
        &=(\alpha\bar{\beta}+\bar{\alpha}\beta)(\cos^2t-\sin^2 t)-i(\alpha^2-\bar{\alpha}^2-\beta^2+\bar{\beta}^2)\sin t \cos t
    \end{align*}
    and
    \begin{align*}
        \Theta(\exp(tX_1)k) &= |\alpha \cos t-i\bar{\beta}\sin t|^2-|{\beta}\cos t +i\bar{\alpha}\sin t|^2\\
        &=(|\alpha|^2-|\beta|^2)(\cos^2 t-\sin^2 t)+2i(\alpha\beta-\bar{\alpha}\bar{\beta})\sin t\cos t.
    \end{align*}
    Therefore, $\frac{\partial}{\partial t}\Psi(\exp(tX_1)k)|_{t=0} = -i(\alpha^2-\bar{\alpha}^2-\beta^2+\bar{\beta}^2)$ and $\frac{\partial}{\partial t}\Theta(\exp(tX_1)k)|_{t=0} = 2i(\alpha\beta-\bar{\alpha}\bar{\beta})$

    The calculations for $X_2$ and $X_3$ are similar.
\end{proof}

\begin{proposition}
    The Hessian of $\tilde{F}$ at $(0,0,0,0)$ is
    \begin{align*}
        D_\rho =\begin{pmatrix}-(1-\rho)&0&0&-2\\0&-(1-\rho^2)&-2\sqrt{1-\rho^2}&0\\0&-2\sqrt{1-\rho^2}&0&0\\-2&0&0&0\end{pmatrix}, 
    \end{align*}
    so its determinant is
    \begin{align*}
        \det D_\rho=16(1-\rho^2).
    \end{align*}
\end{proposition}

\begin{proof}
   To calculate $\partial^2  \tilde{F}/\partial x^2$ and   $\partial^2  \tilde{F}/\partial x\partial t$, we define the map $v:\BR \to K_0$ by the condition that $k_\rho n(x) \in NAv(x)$, i.e. $v(x) = \kappa(k_\rho n(x))$. Lemma \ref{lemma A} gives
    \begin{align*}
        \frac{\partial \tilde{F}}{\partial x} (x,0,0,0) = \frac{\partial }{\partial x}   A(k_\rho n(x))= \Psi(v(x)),
    \end{align*}
    and
    \begin{align*}
        \frac{\partial \tilde{F}}{\partial t} (x,0,0,0) = \left. \frac{\partial}{\partial t}\left(A(k_\rho n(x)a(t)) - \rho t\right) \right|_{t=0}= \Theta(v(x))-\rho.
    \end{align*}
    By Lemma \ref{Iwasawa k},
    \begin{align*}
        v(x)&=\kappa\left( k_\rho n(x)  \right) \\
        &= \kappa\left(\frac{1}{\sqrt{2}} \begin{pmatrix}\sqrt{1+\rho}& i\sqrt{1-\rho} \\ i\sqrt{1-\rho} &\sqrt{1+\rho}\end{pmatrix}\begin{pmatrix}1&x\\0&1\end{pmatrix}\right)\\
        &=\kappa\left(\frac{1}{\sqrt{2}}\begin{pmatrix}\sqrt{1+\rho} &     x\sqrt{1+\rho}+i\sqrt{1-\rho}\\ i\sqrt{1-\rho}&\sqrt{1+\rho}+ix\sqrt{1-\rho}\end{pmatrix}\right)\\
        &=\frac{1}{\sqrt{(1-\rho)x^2+2} }\begin{pmatrix}                    \sqrt{1+\rho}-ix\sqrt{1-\rho}& i\sqrt{1-\rho}\\i\sqrt{1-\rho}& \sqrt{1+\rho}+ix\sqrt{1-\rho}\end{pmatrix},
    \end{align*}
    so
    $$\Psi(v(x)) = \frac{-2(1-\rho)x}{(1-\rho)x^2+2},$$
    and
    $$\Theta(v(x)) = \frac{(1-\rho)x^2+2\rho}{(1-\rho)x^2+2} .$$
    Hence, we have
     \begin{align*}
        \frac{\partial^2\tilde{F}}{\partial x^2} (0,0,0,0) = \left.\frac{\partial}{\partial x}\Psi(v(x))\right|_{x=0} =\left.\frac{\partial}{\partial x}\left(\frac{-2(1-\rho)x}{(1-\rho)x^2+2} \right)\right|_{x=0}=-(1-\rho),
    \end{align*}
    and
    \begin{align*}
        \frac{\partial^2\tilde{F}}{\partial x\partial t} (0,0,0,0)=\left.\frac{\partial}{\partial x}\left(\Theta(v(x))-\rho\right)\right|_{x=0}=\left.\frac{\partial}{\partial x}\left( \frac{(1-\rho)x^2+2\rho}{(1-\rho)x^2+2}-\rho\right)\right|_{x=0}  =   0.
    \end{align*}

    To calculate  $\partial^2 \tilde{F}/\partial t^2$,  we define a map $w:\BR \to K_0$ by the condition that $k_\rho a(t) \in NAw(t)$, i.e., $w(t) = \kappa(k_\rho a(t) )$.  Lemma \ref{lemma A} gives
    \begin{align*}
        \frac{\partial \tilde{F}}{\partial t} (0,t,0,0) = \frac{\partial}{\partial t}\left( A(k_\rho a(t)) - \rho t \right)= \Theta(w(t))-\rho.
    \end{align*}
   By Lemma \ref{Iwasawa k},
    \begin{align*}
        w(t)&= \kappa\left(k_\rho a(t) \right)\\
        &=\kappa\left(\frac{1}{\sqrt{2}} \begin{pmatrix}\sqrt{1+\rho}& i\sqrt{1-\rho} \\ i\sqrt{1-\rho} &\sqrt{1+\rho}\end{pmatrix}\begin{pmatrix}e^{t/2}&0\\0&e^{-t/2}\end{pmatrix}\right)\\
        &= \kappa\left(\frac{1}{\sqrt{2}} \begin{pmatrix}e^{t/2}\sqrt{1+\rho}& e^{-t/2}i\sqrt{1-\rho} \\ e^{t/2}i\sqrt{1-\rho} &e^{-t/2}\sqrt{1+\rho}\end{pmatrix}\right)  \\
        &=\frac{1}{\sqrt{(1-\rho)e^t+(1+\rho)e^{-t}}}\begin{pmatrix} e^{-t/2}\sqrt{1+\rho}&e^{t/2}i\sqrt{1-\rho} \\e^{t/2}i\sqrt{1-\rho} &e^{-t/2}\sqrt{1+\rho} \end{pmatrix},
    \end{align*}
    and so
    \begin{align*}
        \Theta(w(t)) = \frac{(1+\rho)e^{-t}-(1-\rho)e^t}{(1+\rho)e^{-t}+(1-\rho)e^t}.
    \end{align*}
    Hence, we have
     \begin{align*}
        \frac{\partial^2 \tilde{F}}{\partial t^2} (0,0,0,0)= \left.\frac{\partial}{\partial t}\Theta(w(t))\right|_{t=0}= \left.\frac{\partial}{\partial t}\left(\frac{(1+\rho)e^{-t}-(1-\rho)e^t}{(1+\rho)e^{-t}+(1-\rho)e^t}\right)\right|_{t=0}=-(1-\rho^2).
    \end{align*}

    To calculate $\partial^2 \tilde{F}/\partial s_i\partial x$ and $\partial^2 \tilde{F}/\partial s_i\partial t$, by Lemma \ref{lemma A}, we have
    \begin{align*}
        \frac{\partial \tilde{F}}{\partial x} (0,0,s_1,s_2) = \left.\frac{\partial}{\partial x}A(\exp(s_1X_1+s_2X_2)k_\rho n(x)) \right|_{x=0} = \Psi(\exp(s_1X_1+s_2X_2)k_\rho),
    \end{align*}
    and
    \begin{align*}
        \frac{\partial \tilde{F}}{\partial t} (0,0,s_1,s_2) = \left.\frac{\partial}{\partial t}\left(A(\exp(s_1X_1+s_2X_2)k_\rho a(t)) - \rho t\right) \right|_{t=0} = \Theta(\exp(s_1X_1+s_2X_2)k_\rho)-\rho.
    \end{align*}
    Lemma \ref{gradient of psi theta} implies that
    \begin{align*}
        \left(\frac{\partial^2}{\partial s_1\partial x},\frac{\partial^2}{\partial s_2\partial x} \right)\tilde{F}(0,0,0,0) =\left(0, -2  \right)
    \end{align*}
    and
    \begin{align*}
        \left(\frac{\partial^2}{\partial s_1\partial t},\frac{\partial^2}{\partial s_2\partial t} \right)\tilde{F}(0,0,0,0) =\left( -2\sqrt{1-\rho^2},0\right).
    \end{align*}
    
   We have $\partial^2 {\tilde{F}}/\partial s_i\partial s_j (0,0,0,0)=0$ because $\tilde{F}(0,0,s_1,s_2) = A(\exp(s_1X_1+s_2X_2)k_\rho)\equiv 0$.
\end{proof}

Therefore, the critical point is non-degenerate uniformly in the parameter $\rho$, provided $\rho\in [0,\delta]$ for a fixed $0<\delta<1$.
The method of stationary phase implies
\begin{align*}
    J(r;\rho) = \tilde{J}(r;\rho) +  O_N(r^{-N})\ll r^{-2}.
\end{align*}

\subsection{Bound of \texorpdfstring{$\tilde{k_\lambda}(s)$}{k(s)} for big \texorpdfstring{$s$}{s}}
It remains to prove the inequality \eqref{ineq big t}.
We assume that $\lambda^{\epsilon'}\leq\beta\leq\lambda$ and $s\notin I_{\beta/4}$  satisfying $s\geq \lambda/2$. Recall the expansion (\ref{bound using spherical function bound}) of $\tilde{k_\lambda}$ under the polar coordinates. We unfold that integral by the inverse Harish-Chandra transform to obtain
\begin{align*}
        \tilde{k_\lambda}(s) &=c\int_0^\infty k_\lambda(a(t))b_1(t)\varphi_{-s}^{\BH^2}(a(t)) \sinh(t) dt\\
        &=c\int_0^\infty \left(\int_0^\infty   h_\lambda(r) \varphi_{r}^{\BH^3}(a(t))d\nu(r)\right)b_1(t)\varphi_{-s}^{\BH^2}(a(t)) \sinh(t) dt \\
        &=c\int_0^\infty\left(\int_0^\infty b_1(t)\varphi_{r}^{\BH^3}(a(t))\varphi_{-s}^{\BH^2}(a(t)) \sinh(t) dt\right) h_\lambda(r) d\nu(r).
\end{align*}
Recall that $b_1\in C_c^\infty(\BR)$   satisfies $b_1(t)k_\lambda(a(t)) = k_\lambda(a(t))$ for all $t\in\BR$, $b_1\equiv 1$ on $[-1,1]$ and vanishes outside $[-2,2]$.
Since $h_\lambda(r)$ concentrates around $\pm \lambda$ and rapidly decays outside, it suffices to show the bound
\begin{align}\label{the sup norm bound for large t after inversion}
    \int_0^\infty b_1(t)\varphi_{r}^{\BH^3}(a(t))\varphi_{-s}^{\BH^2}(a(t)) \sinh(t) dt \ll \lambda^{-3/2}\beta^{-1/2}
\end{align}
uniformly for $s\notin I_{\beta/4}$  satisfying $s \geq \lambda/2$, and $|r-\lambda|\leq \beta/8$. 
First, we decompose the integral as
\begin{align*}
    \int_0^\infty b_1(t)\varphi_{r}^{\BH^3}(a(t))\varphi_{-s}^{\BH^2}(a(t)) \sinh(t) dt = I_0 + \sum_{n=1}^\infty I_n
\end{align*}
where
\begin{align*}
    I_0 = \int_0^\infty b_1(t)b_1(\beta t)\varphi_{r}^{\BH^3}(a(t))\varphi_{-s}^{\BH^2}(a(t)) \sinh(t) dt,
\end{align*}
and 
\begin{align*}
    I_n= \int_0^\infty b_1(t)\left(b_1(2^{-n}\beta t)-b_1(2^{-n+1}\beta t)\right)\varphi_{r}^{\BH^3}(a(t))\varphi_{-s}^{\BH^2}(a(t)) \sinh(t) dt.
\end{align*}
 The bounds in \cite[Theorem 1.3]{marshall2016p} implies that
\begin{align}\label{spherical bound H3}
    |b_1(t)\varphi_{r}^{\BH^3}(a(t))|\ll \left(1 + |rt|\right)^{-1},
\end{align}
and
\begin{align}\label{spherical bound H2}
    |b_1(t)\varphi_{-s}^{\BH^2}(a(t))|\ll \left(1 + |st|\right)^{-1/2}.
\end{align}
Recall that we assume $b_1$ is supported in $[-2,2]$.
The bounds (\ref{spherical bound H3}), (\ref{spherical bound H2}) and the fact $\sinh(x)= x + O(x^3)$ imply
\begin{align*}
    I_0 &= \int_0^{2/\beta}b_1(t)b_1(\beta t)\varphi_{r}^{\BH^3}(a(t))\varphi_{-s}^{\BH^2}(a(t)) \sinh(t) dt\\
    &\ll \int_0^{2/\beta}  b_1(\beta t ) \left(1 + rt\right)^{-1}\left(1 + st\right)^{-1/2} tdt \\
    &\ll r^{-1}s^{-1/2}\int_0^{2/\beta}  b_1(\beta t)  t^{-1/2}dt \\
    &\ll \lambda^{-3/2}\int_0^{2} b_1(t) \left(\frac{t}{\beta}\right)^{-1/2} \frac{dt}{\beta}\\
    &\ll \lambda^{-3/2}\beta^{-1/2}.
\end{align*}
To bound $I_n$, we use the exponential asymptotic formulas for  $\varphi_r^{\BH^3}(a(t)) $ and $\varphi_{s}^{\BH^2}(a(t))$ given in \cite[Theorem 1.5]{marshall2016p}. There are functions $f^{\BH^3}_\pm,\, f^{\BH^2}_\pm \in C^\infty\left((0,3)\times\BR_{\geq 0} \right)$ such that
\begin{align}
     \frac{\partial^n}{\partial t^n} f^{\BH^3}_{\pm}(t,r)&\ll_n t^{-n}(rt)^{-1} \label{asyprobertyH3}\\
    \frac{\partial^n}{\partial t^n} f^{\BH^2}_{\pm}(t,s)&\ll_n t^{-n}(st)^{-1/2} \label{asyprobertyH2}
\end{align}
and
\begin{align}
    \varphi^{\BH^3}_r (a(t)) &= f^{\BH^3}_+(t,r) e^{irt} + f^{\BH^3}_-(t,r) e^{-irt} + O_N\left( (rt)^{-N}\right)\label{asyH3}\\ 
    \varphi^{\BH^2}_s (a(t)) &= f^{\BH^2}_+(t,s) e^{ist} + f^{\BH^2}_-(t,s) e^{-ist} + O_N\left( (st)^{-N}\right)\label{asyH2}
\end{align}
for $t\in (0,3)$ and $r,s\geq 0$.
We substitute $\varphi^{\BH^3}_s$ in $I_n$ via (\ref{asyH3}). The error term in (\ref{asyH3}) makes a contribution of
\begin{align*}
    &\ll_N \int_0^\infty \left(b_1(2^{-n}\beta t)-b_1(2^{-n+1}\beta t)\right) (rt)^{-N}(1+st)^{-1/2}tdt\\
    &\ll_N (\lambda \beta^{-1}2^n)^{-N} (\lambda\beta^{-1}2^n)^{-1/2}(\beta^{-1}2^n)^2 \\
    &= \lambda^{-N-1/2}\beta^{N-3/2}2^{-nN+3n/2}\\
    &\ll \lambda^{-5/2}\beta^{1/2}2^{-n/2}\\
    &\leq \lambda^{-3/2} \beta^{-1/2} 2^{-n/2}
\end{align*}
to $ I_n$. Thus, summing over $n$,
\begin{align*}
    \sum_{n=1}^\infty I_n =&\sum_{n=1}^\infty  \int_0^\infty b_1(t)\left(b_1(2^{-n}\beta t)-b_1(2^{-n+1}\beta t)\right)\left(f_+^{\BH^3}(t,r) e^{irt} + f_-^{\BH^3}(t,r) e^{-irt}  \right)\varphi_{-s}^{\BH^2}(a(t)) \sinh(t)dt\\ &+ O(\lambda^{-3/2}\beta^{-1/2}).
\end{align*}
We now substitute $\varphi^{\BH^2}_{-s} = \varphi^{\BH^2}_{s}$ via (\ref{asyH2}) in $I_n$. Because of the bound (\ref{asyprobertyH3}), the error term in (\ref{asyH2}) makes a contribution of
\begin{align*}
    &\ll_N \int_0^\infty \left(b_1(2^{-n}\beta t)-b_1(2^{-n+1}\beta t)\right) (st)^{-N}(rt)^{-1}tdt\\
    &\ll_N (\lambda \beta^{-1}2^n)^{-N} (\beta^{-1}2^n)^2\\
    &= \lambda^{-N}\beta^{N-2}2^{-nN+2n}\\
    &\ll \lambda^{-3/2}\beta^{-1/2}2^{-n/2}
\end{align*}
to $ I_n$. After summing over $n$, we have
\begin{align*}
    \sum_{n=1}^\infty I_n =O(\lambda^{-3/2}\beta^{-1/2})+\sum_{n=1}^\infty  \int_0^\infty b_1(t)\left(b_1(2^{-n}\beta t)-b_1(2^{-n+1}\beta t)\right)&\left(f_+^{\BH^3}(t,r) e^{irt} + f^{\BH^3}_-(t,r) e^{-irt}  \right)
   \\
    &  \left( f^{\BH^2}_+(t,s) e^{ist} + f^{\BH^2}_-(t,s) e^{-ist} \right)\sinh(t)dt.
\end{align*}
We shall only consider the term involving $f^{\BH^3}_+(t,r)f_-^{\BH^2}(t,s)$, as the other terms are similar. We must then estimate
\begin{align*}
    \int_0^\infty b_1(t)\left(b_1(2^{-n}\beta t)-b_1(2^{-n+1}\beta t)\right)f^{\BH^3}_+(t,r) f^{\BH^2}_-(t,s) e^{i(r-s)t} t dt.
\end{align*}
If we replace $t$ with $2^n\beta^{-1}t$, this becomes
\begin{align*}
    \int_0^\infty 2^{2n}\beta^{-2}b_1(2^n\beta^{-1}t)\left(b_1(t)-b_1(2t)\right)f^{\BH^3}_+(2^n\beta^{-1}t,r) f^{\BH^2}_-(2^n\beta^{-1}t,s) e^{i(r-s)2^n\beta^{-1}t} tdt.
\end{align*}
The conditions $s\notin I_{\beta/4}$ and $|r-\lambda|\leq \beta/8$ imply that $|(r-s)2^n\beta^{-1}|\geq 2^{n-3}$. Also \eqref{asyprobertyH3} and \eqref{asyprobertyH2} imply that
\begin{align}\label{bound for fH3fH2}
    f^{\BH^3}_+(2^n\beta^{-1}t,r) f^{\BH^2}_-(2^n\beta^{-1}t,s)\ll \lambda^{-3/2}\beta^{3/2} 2^{-3n/2} t^{-3/2} ,
\end{align}
and
\begin{align}\label{bound for derivative of fH3fH2}
\begin{split}
    &\frac{\partial}{\partial t} \left( f^{\BH^3}_+(2^n\beta^{-1}t,r) f^{\BH^2}_-(2^n\beta^{-1}t,s)\right) \\
    =& \left( \frac{\partial}{\partial t}  f^{\BH^3}_+(2^n\beta^{-1}t,r)\right) f^{\BH^2}_-(2^n\beta^{-1}t,s)\ +  f^{\BH^3}_+(2^n\beta^{-1}t,r) \left(\frac{\partial}{\partial t}f^{\BH^2}_-(2^n\beta^{-1}t,s)\right)\\
    \ll& \lambda^{-3/2}\beta^{3/2} 2^{-3n/2} t^{-5/2}.
\end{split}
\end{align}
Integrating by parts once, and then applying the bounds \eqref{bound for fH3fH2} and \eqref{bound for derivative of fH3fH2} gives
\begin{align*}
    &\int_0^\infty 2^{2n}\beta^{-2}b_1(2^n\beta^{-1}t)\left(b_1(t)-b_1(2t)\right)f^{\BH^3}_+(2^n\beta^{-1}t,r) f^{\BH^2}_-(2^n\beta^{-1}t,s) e^{i(r-s)2^n\beta^{-1}t} tdt \\
    \ll&\frac{1}{|(r-s)2^n\beta^{-1}|}\left(2^{2n}\beta^{-2}\right)\left( \lambda^{-3/2}\beta^{3/2} 2^{-3n/2}\right)\\
    \ll&  \lambda^{-3/2}\beta^{-1/2}2^{-n/2},
\end{align*}
and summing over $n$ completes the proof of \eqref{ineq big t}.

\section{Estimates of oscillatory integrals}\label{section of oscillatory integrals}
\subsection{Lemmas on the function \texorpdfstring{$A$}{A}}
We begin with some results we shall use repeatedly in this section.
For $g\in G_0$, let $\Phi_g: K_0\to K_0$ be the map sending $k$ to $\kappa(kg)$, i.e $kg\in NA\Phi_g(k)$. 
\begin{lemma}\label{change of variable}
    $\Phi$ is a smooth group action of $G_0$ on $K_0$ from right.

    Moreover, $\Phi_g$ induces a diffeomorphism from $M\backslash K_0$ to itself.
\end{lemma}
\begin{proof}
    The smoothness of the Iwasawa decomposition implies that $\Phi_g$ is smooth and depends smoothly on $g$. We identify $K_0$ with the quotient $NA\backslash G_0$ via the Iwasawa decomposition. Then $\Phi_g$ is obtained by composing the diffeomorhpism $K_0\to NA\backslash G_0$, the right multiplication by $\cdot g:NA\backslash G_0\to NA\backslash G_0$ and the diffeomorphism $ NA\backslash G_0\to K_0$.
 It remains to show $\Phi_{gh} = \Phi_h\circ\Phi_g$ for $g,h\in G_0$. By definition
     \begin{align*}
         kgh\in NA\Phi_{gh}(k)\\
         kg\in NA\Phi_g(k),
     \end{align*}
     which implies
     \begin{align*}
         NA\Phi_g(k) h = NA\Phi_{gh}(k),\\
         \Phi_g(k) h \in NA\Phi_{gh}(k),
     \end{align*}
     that is $\Phi_h\left(  \Phi_g(k) \right) = \Phi_{gh}(k)$.
     
    The second statement follows from the fact that $M$ normalizes $NA$.
\end{proof}

\begin{lemma}\label{spliting A}
    Let $y,z\in G_0$ and let $k\in K_0$. Then we have
    \begin{align*}
        A(ky^{-1}z) = A\left(\Phi_{y^{-1}}(k)z\right)- A\left(\Phi_{y^{-1}}(k)y\right).
    \end{align*}
\end{lemma}
\begin{proof}
    Let $ky^{-1} = na\Phi_{y^{-1}}(k)$. We have
    \begin{align*}
        A(ky^{-1}z) &= A\left( na\Phi_{y^{-1}}(k)z\right)\\
        &=A(na) + A\left( \Phi_{y^{-1}}(k)z\right)\\
        &=A\left( \Phi_{y^{-1}}(k)z\right) - A\left((na)^{-1}k\right)\\
        &=A\left( \Phi_{y^{-1}}(k)z\right) - A\left(\Phi_{y^{-1}}(k)y\right).
    \end{align*}
\end{proof}

 Recall that the tangent space of $M\backslash K_0$ at the origin $M$ can be viewed as a two-dimensional subspace of the Lie algebra $\fk$ spanned by the vectors  $X_1  =\begin{pmatrix}0&i\\i&0\end{pmatrix}$ and $X_2 = \begin{pmatrix}0&-1\\1&0\end{pmatrix}$.
Let $U_\fk$ be a circle centered at 0 in the tangent space defined by
\begin{align}\label{eq:defn of Uk}
    U_\fk = \left\{r_1X_1+r_2X_2: r_1^2+r_2^2 = 1, r_1,r_2\in\BR  \right\}.
\end{align}
 For a given $C>0$, we let
    \begin{align*}
        \BD_C = \{   z\in\BC:|z|<C        \}
    \end{align*}
    be the open disc of radius $C$ in $\BC$.
We shall need the following uniformization lemma for the function $A$.

\begin{lemma}\label{uniformazation lemma}
    Let $C> 0$.
    There exists $\delta,\sigma>0$ depending on $C$, and a real analytic function
    \begin{align*}
        \xi:(-\delta,\delta)\times U_\fk \times \BD_C \times (-C,C)\to\BR
    \end{align*}
    such that
    \begin{align*}
        \frac{\partial}{\partial t} A(\exp(rX)n(z)a(t)) = 1-r^2\xi(r,X,z,t),
    \end{align*}
    and
    \begin{align}\label{34}
    \begin{split}
        |\xi(r,X,z,t)| 
        &\geq \sigma,\\
        \left|\frac{\partial^n\xi}{\partial t^n} (r,X,z,t)\right|
        &\ll_{C,n}1,
        \end{split}
    \end{align}
    for $(r,X,z,t)\in(-\delta,\delta)\times U_\fk \times \BD_C \times (-C,C)$. 
\end{lemma}

\begin{proof}
    Let $\delta_0>0$ be a small number so that the exponential map from the set
    \begin{align*}
        U_{\fk}(\delta_0) := \{ rX\,|\, r\in (-\delta_0,\delta_0), X\in U_\fk     \}
    \end{align*}
    onto its image in $K_0$ is a diffeomorphism.
    Moreover, $\exp(U_\fk(\delta_0))$ is transversal to $M$ at the origin.
    Hence, we may think of it as a neighborhood near the origin in $M\backslash K_0$.
    We define the functions 
    \begin{align*}
        \alpha,\,\beta:(-\delta_0,\delta_0) \times U_\fk \times \BD_{C}\times (-C,C) \to \BC
    \end{align*}
    by requiring that
    \begin{align*}
        \exp(rX) n(z) a(t) \in NAk(\alpha(r,X,z,t),\beta(r,X,z,t)),
    \end{align*}
   where
    \begin{align*}
        k(\alpha,\beta) =\begin{pmatrix} \alpha&\beta\\-\bar{\beta}&\bar{\alpha}\end{pmatrix}\in  K_0.
    \end{align*}

    The analyticity of the Iwasawa decomposition implies that $\alpha$ and $\beta$ are analytic, and in particular $|\alpha|^2$ and $|\beta|^2$ are analytic, as functions of $(rX,z,t) \in U_\fk(\delta_0)\times \BD_C\times (-C,C)$.
    \eqref{derivative of A about a} in Lemma \ref{lemma A}  implies that
    \begin{align*}
        1-\frac{\partial}{\partial t} A(\exp(rX) n(z) a(t)) = 1-\Theta(k(\alpha,\beta)) = (1- |\alpha|^2)+|\beta|^2 = 2|\beta|^2,
    \end{align*}
    which is equal to $0$ if and only if $\beta = 0$. This is equivalent to $k(\alpha,\beta)\in M$. 
    By Lemma \ref{change of variable}, this is also equivalent to $\exp(rX)\in M$.
    By the transversality,
    we can choose $0<\delta<\delta_0$ such that $\beta(r,X,z,t)$ vanishes on $(-\delta,\delta)\times U_\fk \times \BD_{C} \times (-C,C)$ if and only if $r = 0$.

    Lemma \ref{change of variable} implies that $\partial\beta / \partial r$ is analytic and never vanishes on $\{ 0\}\times U_\fk\times\BD_{C}\times (-C,C)$. 
    It may be seen that there is a real analytic function $\xi_0$ on $(-\delta,\delta)\times U_\fk \times \BD_{C}  \times (-C,C)$ such that $\beta (r,X,z,t)= r\xi_0(r,X,z,t)$. Defining $\xi = 2|\xi_0|^2$ gives the result. 
\end{proof}

\begin{lemma}\label{uniform boundedness}
    Let $C, \delta > 0$. There exists $\sigma>0$ depending on $C$ and $\delta$ such that
    \begin{align*}
        1-\frac{\partial}{\partial t} A(k n(z) a(t)) \geq \sigma,
    \end{align*}
    for arbitrary $k\in K_0$ with $d(k,M)\geq \delta$, $z\in \BD_C$ and $t\in (-C,C)$.
\end{lemma}
\begin{proof}
    We fix the element $k\in K_0$ with $d(k,M)\geq \delta$.
    Define the functions $\alpha,\beta: \BD_{C} \times (-C,C)\to \BC$ by requiring that
    \begin{align*}
        kn(z) a(t) \in NAk(\alpha(z,t),\beta(z,t)),
    \end{align*}
    i.e.
    \begin{align*}
        k(\alpha(z,t),\beta(z,t)) = \Phi_{n(z)a(t)}(k).
    \end{align*}
    Here
    \begin{align*}
         k(\alpha,\beta) =\begin{pmatrix} \alpha&\beta\\-\bar{\beta}&\bar{\alpha}\end{pmatrix}.
    \end{align*}
    Since $d(k,M)\geq \delta$ and $\Phi_{n(z)a(t)}$ acts trivially on $M$ for any $z\in\BD_C$ and $t\in (-C,C)$,
    by Lemma \ref{change of variable}, there is some $C_1>0$ such that $d(\Phi_{n(z)a(t)}(k),M)>C_1\delta$. Hence, $|\beta(z,t)| > C_2\delta$ for some other constant $C_2>0$, for any $z\in\BD_C$ and $t\in (-C,C)$.
    Lemma \ref{lemma A} (\ref{derivative of A about a}) implies that
    \begin{align*}
        1-\frac{\partial}{\partial t} A(k n(z) a(t)) = 1-\frac{\partial}{\partial t} A(k(\alpha,
        \beta) a(t))|_{t=0}  = (1- |\alpha|^2)+|\beta|^2 = 2|\beta|^2 > 2C_2^2 \delta^2.
    \end{align*}
    Then the result follows by taking $\sigma = 2C_2^2\delta^2$.
\end{proof}

The above two lemmas, in particular, hold for $n:\BC\to N$ restricted to $\BR\to N_0$ and for $B$ instead of $M\backslash K_0$. Notice that we have the embedding $B = \{ \pm I\}\backslash \SO(2) \hookrightarrow M\backslash K_0$ with $\SO(2)\subset K_0=\SU(2)$ and $\SO(2)\cap M = \{ \pm I\}$. 
Therefore, we obtain the following corollary, which we have proved by a direct formula for the $A$ function on $\BH^2$ in Lemma \ref{lemma for A(bna) when b is large} and Lemma \ref{lemma for A(bna) when b is small}.
\begin{corollary}\label{uniformization lemma for B}
    Let $C > 0$. There exists $\delta,\sigma>0$ depending on $C$, and a real analytic function
    \begin{align*}
        \xi:(-\delta,\delta)\times (-C,C)^2\to\BR
    \end{align*}
    such that
    \begin{align*}
        \frac{\partial}{\partial t} A(b_\theta n(x)a(t)) = 1-\theta^2\xi(\theta,x,t),
    \end{align*}
    and
    \begin{align*}
        |\xi(\theta,x,t)| 
        &\geq \sigma,\\
        \left|\frac{\partial^n\xi}{\partial t^n} (\theta,x,t)\right|
        &\ll_{C,n}1,
    \end{align*}
    for $(\theta,x,t)\in(-\delta,\delta) \times (-C,C)^2 $. Moreover, if $\theta\notin (-\delta,\delta) \mod 2\pi$ and $x,t\in (-C,C)$, then
    \begin{align*}
        1-\frac{\partial}{\partial t} A(b_\theta n(x) a(t)) \geq \sigma.
    \end{align*}
\end{corollary}

\subsection{Oscillatory integrals for geodesic tubes}

We first estimate two one-dimensional integrals, which can be thought of as the integrals along a geodesic.

\begin{proposition}\label{first one dimensional intgeral estimate}
    Let $B,C, D>0$ and $\epsilon_0>0$ be constants.
    Let $\chi_0\in C_c^\infty (\BR)$ be a smooth function supported in $[-1,1]$, and let $\rho$ be a smooth function defined on $(-1,1)$. 
    If $z\in \BD_C\subset \BC$, $t\in (-C,C)$, $k\in K_0$, $s^\prime>0$ satisfying
    \begin{align}\label{35}
        d(k,M)\geq Bs^{-1/2 + \epsilon_0}\beta^{1/2}\quad\text{ and }\quad       |s^\prime-s|\leq \beta
    \end{align}
    for some $s,\beta\geq 1$ with $\beta\leq 3s^{1/4}$, and
    \begin{align}\label{condition for rho}
        \left|s^\prime \frac{\partial\rho}{\partial t^{\prime}} (t^\prime)\right|\leq Ds^{\epsilon_0} \beta,\quad\text{ and }\quad \left|s^\prime \frac{\partial^n\rho}{\partial t^{\prime n}} (t^\prime)\right|\ll_n s^{\epsilon_0} \beta,\quad \text{ for }t^\prime\in (-1,1),
    \end{align}
    then
    \begin{align}\label{36}
        \int_{-\infty}^\infty \chi_0(t^\prime) \exp\left( is^\prime (t^\prime + \rho(t^\prime)) - is A(kn(z)a(t+t^\prime))  \right) dt^\prime\ll s^{-A}.
    \end{align}
    The implied constant depends on $A,B,C,D,\epsilon_0$, the first $n$ implied constants in \eqref{condition for rho},  and the size of the first $n$ derivatives of $\chi_0$, where $n$ depends on $\epsilon_0$ and $A$.
\end{proposition}
\begin{proof}
    By applying Lemma \ref{uniformazation lemma}, we see that there is a $\delta>0$ and a nonvanishing real analytic function  $\xi$ on $(-\delta,\delta)\times U_\fk \times \BD_{C}\times (-C-2,C+2)$ such that
    \begin{align*}
        \frac{\partial}{\partial t} A(\exp(rX)n(z)a(t)) = 1-r^2\xi(r,X,z,t),
    \end{align*}
    when $r\in (-\delta,\delta)$, $X\in U_\fk$, $z\in \BD_C$ and $t\in (-C-2,C+2)$. If $Z(r,X,z,t)$ is an antiderivative of $\xi$ with respect to $t$ that is smooth as a function of $(r,X,z,t)$, we may integrate this to obtain
    \begin{align*}
        A(\exp(rX)n(z)a(t)) = t - r^2 Z(r,X,z,t) + c(r,X,z)
    \end{align*}
    for some function $c(r,X,z)$.
    
    If $k=\exp(rX)$ for some $r\in (-\delta,\delta)$ and $X\in U_\fk$, we may use this to rewrite the integral (\ref{36}) as
    \begin{align}\label{37}
         &\int_{-\infty}^\infty \chi_0(t^\prime) \exp\left( is^\prime (t^\prime + \rho(t^\prime)) - is A(\exp(rX)n(z)a(t+t^\prime))  \right) dt^\prime\notag\\
         =&\int_{-\infty}^\infty \chi_0(t^\prime) \exp\left( is^\prime  (t^\prime + \rho(t^\prime)) - is (t+t^\prime) + isr^2 Z(r,X,z,t+t^\prime) - is c(r,X,z)  \right) dt^\prime \notag\\
         =& e^{-is(t+c(r,X,z))} \int_{-\infty}^\infty \chi_0(t^\prime) \exp\left( i(s^\prime-s)t^\prime  +is^\prime\rho(t^\prime) + isr^2 Z(r,X,z,t+t^\prime)\right) dt^\prime  \notag \\
         =& e^{-is(t+c(r,X,z))} \int_{-\infty}^\infty \chi_0(t^\prime) \exp\left( isr^2 \Psi(t^\prime)\right) dt^\prime,
    \end{align}
    where we define
    \begin{align*}
        \Psi(t^\prime) = Z(r,X,z,t+t^\prime) + s^{-1}r^{-2}(s^\prime-s)t^\prime + s^{-1}r^{-2}s^\prime\rho(t^\prime).
    \end{align*}

    Our assumption (\ref{35}) implies that
    $r \geq B_1 s^{-1/2 + \epsilon_0}\beta^{1/2}$ for some constant $B_1 > 0$, and
    \begin{align*}
        |s^{-1}r^{-2}(s^\prime-s)| \leq s^{-1} \left( B_1 s^{-1/2 + \epsilon_0}\beta^{1/2} \right)^{-2} \beta = B_1^{-2} s^{-2\epsilon_0}.
    \end{align*}
    The assumption (\ref{condition for rho}) implies that
    \begin{align*}
        \left|s^{-1}r^{-2}s^\prime\frac{\partial\rho}{\partial t^\prime}(t^\prime)\right|\leq s^{-1} \left( B_1 s^{-1/2 + \epsilon_0}\beta^{1/2} \right)^{-2} D s^{\epsilon_0} \beta= B_1^{-2}D s^{-\epsilon_0}.
    \end{align*}
    Hence,
    \begin{align}\label{38}
        \begin{split}
            \Psi &= Z(r,X,z,t+t^\prime)  +  s^{-1}r^{-2}s^\prime\rho(t^\prime) + O\left(s^{-2\epsilon_0}\right)t^\prime, \quad\text{ and }\\
            \frac{\partial\Psi}{\partial t^\prime}&= \xi(r,X,z,t+t^\prime)  + O\left(s^{-\epsilon_0}\right).
        \end{split}
    \end{align}
    It follows from (\ref{34}) and (\ref{38}) that there exists a $\sigma>0$ so that, for $s$ sufficiently large, $|\partial\Psi/\partial t^\prime| > \sigma/2$ for all $r\in(-\delta,\delta)$, $X\in U_\fk$, $z\in\BD_C$, $t\in(-C,C)$ and $t^\prime\in(-1,-1)$.
    In addition, all higher derivatives of $\Psi$ are bounded from above. As \eqref{35} implies that
    \begin{align*}
        sr^2 \geq  B_1^2 s^{2\epsilon_0}\beta \geq B_1^2 s^{2\epsilon_0},
    \end{align*}
    the bound (\ref{36}) follows by integration by parts in (\ref{37}).

    In the case where $k$ can not be written as $\exp(rX)$ for $r\in (-\delta,\delta)$, we can assume $d(k,M)$ to be bounded below by some constant multiple of $\delta$, so Lemma \ref{uniform boundedness} implies that 
    \begin{align*}
        1-\frac{\partial}{\partial t^\prime} A(k n(z) a(t+t^\prime)) \geq C_1
    \end{align*}
    for some $C_1 > 0$. It gives
    \begin{align*}
        &\left|\frac{\partial}{\partial t^\prime} \left(s^\prime s^{-1} (t^\prime +\rho(t^\prime))- A(kn(z)a(t+t^\prime)) \right)\right|\\
        \geq &\left(  1-\frac{\partial}{\partial t^\prime} A(k n(z) a(t+t^\prime)) \right)- \left|s^\prime s^{-1} - 1 \right| - \left| s^\prime s^{-1} \frac{\partial\rho}{\partial t^\prime} (t^\prime)\right|\\
        \geq & \,C_1 - \beta s^{-1} -  s^{-1} Ds^{\epsilon_0} \beta \\
        \gg&1.
    \end{align*}
    The result also follows by integration by parts.
\end{proof}

\begin{lemma}\label{lemma 6.7}
     Let $C > 0$ and $1/4 > \delta >0$. There exists $C_1>0$ depending on $C$,  such that
     \begin{align*}
         &\left|\left( \frac{\partial}{\partial r_1},  \frac{\partial}{\partial r_2}\right)   A(\exp(r_1X_1+r_2X_2)n(z) a(t))   \right|\\
         :=&  \left(  \sum_{i=1}^2 \left(\frac{\partial}{\partial r_i}   A(\exp(r_1X_1+r_2X_2)n(z) a(t))   \right)^2\right)^{1/2}\\
         \asymp&  |z|,
     \end{align*}
     for $z\in \BD_C$ with $|z| \geq \delta$, $t\in (-C,C)$, and $r_1,r_2\in (-C_1 \delta, C_1 \delta)$, where the implied constants are absolute. Here $X_1, X_2\in\fk$ are defined as (\ref{basis for lie algebra k}).
\end{lemma}
\begin{proof}
    By computing the matrix exponential directly, we obtain
    \begin{align}\label{formula of exp(r1X1+r2X2)}
        \exp(r_1X_1+r_2 X_2) = \begin{pmatrix}
            \cos (r) & (ir_1 - r_2)\sin (r)/r\\
            (ir_1 + r_2)\sin (r)/r & \cos (r)
        \end{pmatrix},
    \end{align}
    where $r = \sqrt{r_1^2 + r_2^2}$. It may be seen from (\ref{iwasawa height}) that
    \begin{align*}
        &A\left(\exp(r_1X_1+r_2 X_2)n(z)a(t)\right)\\
        =&t - \log\left(\cos^2(r)+\sin^2(r)(|z|^2+e^{2t}) -\frac{\sin(r)\cos(r)}{r}\left( (-ir_1-r_2)z + (ir_1-r_2)\bar{z} \right)\right).
    \end{align*}
    If we fix $\delta_0 > 0$ sufficiently small, for $r_1,r_2 \in (-\delta_0,\delta_0)$,  $z\in\BD_C$ and $t\in (-C,C)$,
    we have
    \begin{align*}
        &\frac{\partial}{\partial r_1}\left( \cos^2(r)+\sin^2(r)(|z|^2+e^{2t}) -\frac{\sin(r)\cos(r)}{r}\left( (-ir_1-r_2)z + (ir_1-r_2)\bar{z} \right)  \right)\\
        =&  -2\sin(r)\cos(r)\frac{r_1}{r} + 2\sin(r)\cos(r)\frac{r_1}{r} (|z|^2+e^{2t})-\frac{\sin(r)\cos(r)}{r}\left( -iz + i\bar{z} \right) \\
        &- \frac{r(\cos^2(r)-\sin^2(r)) -\sin(r)\cos(r) }{r^2} \frac{r_1}{r}\left( (-ir_1-r_2)z + (ir_1-r_2)\bar{z} \right).
    \end{align*}
    Since $\sin(r)= r + O(r^3)$, $\cos(r)=1+O(r^2)$ and $r_1$, $r_2$ are $\ll r$, we have
    \begin{align*}
        &-2\sin(r)\cos(r)\frac{r_1}{r} + 2\sin(r)\cos(r)\frac{r_1}{r} (|z|^2+e^{2t})\ll r,\\
        &-\frac{\sin(r)\cos(r)}{r}\left( -iz + i\bar{z} \right)=i(z-\overline{z})+O(r^2).
    \end{align*}
    Because $r(\cos^2(r)-\sin^2(r)) -\sin(r)\cos(r) \ll r^3$,
    \begin{align*}
        - \frac{r(\cos^2(r)-\sin^2(r)) -\sin(r)\cos(r) }{r^2} \frac{r_1}{r}\left( (-ir_1-r_2)z + (ir_1-r_2)\bar{z} \right)\ll r^2.
    \end{align*}
     Hence,
    \begin{align*}
        \frac{\partial}{\partial r_1} A\left(\exp(r_1X_1+r_2 X_2)n(z)a(t)\right)=  -i(z-\bar{z}) + O(r) = 2\Im(z) + O(r).
    \end{align*}
    Similarly, we have
    \begin{align*}
        &\frac{\partial}{\partial r_2}\left( \cos^2(r)+\sin^2(r)(|z|^2+e^{2t}) -\frac{\sin(r)\cos(r)}{r}\left( (-ir_1-r_2)z + (ir_1-r_2)\bar{z} \right)  \right)\\
        =&  -2\sin(r)\cos(r)\frac{r_2}{r} + 2\sin(r)\cos(r)\frac{r_2}{r} (|z|^2+e^{2t})-\frac{\sin(r)\cos(r)}{r}\left( -z -\bar{z} \right) \\
        &- \frac{r(\cos^2(r)-\sin^2(r)) -\sin(r)\cos(r) }{r^2} \frac{r_2}{r}\left( (-ir_1-r_2)z + (ir_1-r_2)\bar{z} \right).
    \end{align*}
    and we have the following estimates
    \begin{align*}
        &-2\sin(r)\cos(r)\frac{r_2}{r} + 2\sin(r)\cos(r)\frac{r_2}{r} (|z|^2+e^{2t})\ll r,\\
        &-\frac{\sin(r)\cos(r)}{r}\left( -z -\bar{z} \right) = (z+\overline{z})+O(r^2),\\
        &- \frac{r(\cos^2(r)-\sin^2(r)) -\sin(r)\cos(r) }{r^2} \frac{r_2}{r}\left( (-ir_1-r_2)z + (ir_1-r_2)\bar{z} \right)\ll r^2
    \end{align*}
    Hence, 
    \begin{align*}
        \frac{\partial}{\partial r_2} A\left(\exp(r_1X_1+r_2 X_2)n(z)a(t)\right)=-(z+\bar{z}) + O(r)=- 2\Re(z) + O(r).
    \end{align*}
    We pick $C_1$ small enough so that
    \begin{align*}
        \left|\left( \frac{\partial}{\partial r_1},  \frac{\partial}{\partial r_2}\right)  A\left(\exp(r_1X_1+r_2 X_2)n(z)a(t)\right)\right| = 2|z| + O(r) \asymp |z|
    \end{align*}
    for $z\in \BD_C$ with $|z| \geq \delta$, $t\in (-C,C)$ and $r_1,r_2\in (-C_1 \delta, C_1 \delta)$.
\end{proof}

The second integral we want to estimate includes the spherical functions on $\BH^3$.

\begin{proposition}\label{second one dimensional integral estimate}
    Let $B,C,D>0$ and $1/8>\epsilon_0>0$ be constants. 
    Let $\chi_0\in C_c^\infty(\BR)$ be a smooth function supported in $[-1,1]$,
    and let $\rho$ be a smooth function defined on $(-1,1)$. 
    If $z\in \BD_C$, $t\in(-C,C)$ and $s^\prime > 0$ satisfying
    \begin{align}\label{39}
        |z|\geq B s^{-1/2 + \epsilon_0} \beta^{1/2}\;\text{ and }\;|s^\prime-s|\leq \beta
    \end{align}
    for some $s,\beta\geq 1$ with $\beta\leq 3s^{1/4}$, and
    \begin{align}\label{condition for rho second}
        \left|s^\prime \frac{\partial\rho}{\partial t^{\prime}} (t^\prime)\right|\leq Ds^{\epsilon_0}\beta,\quad\text{ and }\quad \left|s^\prime \frac{\partial^n\rho}{\partial t^{\prime n}} (t^\prime)\right|\ll_n s^{\epsilon_0}\beta,\quad \text{ for }t^\prime\in (-1,1),
    \end{align}
    then
    \begin{align}\label{40}
        \int_{-\infty}^\infty \chi_0(t^\prime) \exp\left(is^\prime (t^\prime + \rho(t^\prime)) \right)\varphi_s^{\BH^3}\left(n(z)a(t+t^\prime) \right) dt^\prime       \ll s^{-A}.
    \end{align}
    The implied constant depends on $A,B,C,D,\epsilon_0$, the first $n$ implied constants in \eqref{condition for rho second},  and the size of the first $n$ derivatives of $\chi_0$, where $n$ depends on $\epsilon_0$ and $A$.
\end{proposition}
\begin{proof}
    If we apply the functional equation $\varphi^{\BH^3}_s = \varphi^{\BH^3}_{-s}$, and substitute $\varphi^{\BH^3}_{-s}$ by the integral formula \eqref{eq:HC int for H3} into the left-hand side of \eqref{40}, it becomes
    \begin{align*}
        &\int_{-\infty}^\infty \int_{K_0} \chi_0(t^\prime) \exp\left( is^\prime (t^\prime + \rho(t^\prime))   +(1-is) A(kn(z)a(t+t^\prime))\right) dk dt^\prime \\
        =&\int_{-\infty}^\infty \int_{M\backslash K_0} \chi_0(t^\prime) \exp\left( is^\prime (t^\prime + \rho(t^\prime))   +(1-is) A(kn(z)a(t+t^\prime))\right) dk dt^\prime .
    \end{align*}
   Let $C_1>0$ be a constant to be chosen later. Let $\chi_1,\chi_2\in C_c^\infty(M\backslash K_0)$ be such that $0\leq \chi_1,\chi_2\leq 1$, $\chi_1 + \chi_2 \equiv 1$, and
    \begin{align*}
        &\supp(\chi_1)\subset \left\{ Mk \in M\backslash K_0 \, | \,  k = \exp(rX), r\in [-2C_1 s^{-1/2+\epsilon_0}\beta^{1/2},2C_1 s^{-1/2+\epsilon_0}\beta^{1/2}], X\in U_\fk \right\},           \\
        &\supp(\chi_2)\subset M\backslash K_0 - \left\{ Mk \in M\backslash K_0 \, | \,  k = \exp(rX), r\in [-C_1 s^{-1/2+\epsilon_0}\beta^{1/2},C_1 s^{-1/2+\epsilon_0}\beta^{1/2}], X\in U_\fk \right\}.
    \end{align*}
    Proposition \ref{first one dimensional intgeral estimate} implies that
    \begin{align*}
        \int_{-\infty}^\infty \int_{M\backslash K_0} \chi_2(k) \chi_0(t^\prime) \exp\left( is^\prime (t^\prime + \rho(t^\prime))   +(1-is) A(kn(z)a(t+t^\prime))\right) dk dt^\prime \ll_A s^{-A}.
    \end{align*}
    Hence, it suffices to estimate
    \begin{align*}
         &\int_{-\infty}^\infty \int_{M\backslash K_0} \chi_1(k) \chi_0(t^\prime) \exp\left( is^\prime (t^\prime + \rho(t^\prime))   +(1-is) A(kn(z)a(t+t^\prime))\right) dk dt^\prime\\
         =&\int_{-\infty}^\infty\left( \int_{M\backslash K_0} \chi_1(k)  \exp\left(   (1-is) A(kn(z)a(t+t^\prime))\right) dk \right)\chi_0(t^\prime)\exp\left(is^\prime (t^\prime + \rho(t^\prime)) \right) dt^\prime.
    \end{align*}
    We shall do this by estimating the inner integral above. That is,
    \begin{align*}
        &\int_{M\backslash K_0} \chi_1(k) \exp\left(  (1 -is) A(kn(z)a(t))\right) dk\notag\\
        =& \iint_{-\infty}^\infty  \tilde{\chi_1}(r_1,r_2)\exp\left(   -is A(\exp(r_1X_1 + r_2X_2)n(z)a(t))\right) dr_1 dr_2.
    \end{align*}
    Here $z\in \BD_C$, $|z| \geq Bs^{-1/2+\epsilon_0}\beta^{1/2}$ and $t\in (-C-1,C+1)$, and $\tilde{\chi_1}(r_1,r_2) = \tilde{\chi_1}(r_1,r_2;z,t)$ is the smooth compactly supported function obtained by combining all of the amplitude factors. Moreover,
    \begin{align*}
        \supp(\tilde{\chi_1}) \subset \left\{ (r_1,r_2)\in\BR^2 \,|\, r =(r_1^2 +  r_2^2)^{1/2} \in [-2C_1 s^{-1/2+\epsilon_0}\beta^{1/2},{2}C_1 s^{-1/2+\epsilon_0}\beta^{1/2}]  \right\}.
    \end{align*}
    The proposition now follows by applying the following Lemma to $\delta=Bs^{-1/2+\epsilon_0}\beta^{1/2}$ and choosing $C_1$ to be sufficiently small.
\end{proof}

\begin{lemma}\label{nonstationary phase on K}
    Let $C>0$ be a constant and let $1>\delta>0$. Then there exists a constant $C_1>0$ such that the following holds. Suppose that $\chi_0\in C_c^\infty(\BR^2)$ is a smooth function supported in the unit ball centered at the origin $(0,0)$. 
    If $t\in(-C,C)$, $z\in \BD_C$ satisfying $|z|\geq\delta,$
    then
    \begin{align*}
         \iint_{-\infty}^\infty  \chi_0\left((C_1\delta)^{-1}(r_1,r_2)\right)\exp\left(  is A(\exp(r_1X_1 + r_2X_2)n(z)a(t))\right) dr_1 dr_2\ll \delta^2(s\delta^2)^{-A}.
    \end{align*}
    The implied constant depends only on $A,C$, and the size of the first $n$ derivatives of $\chi_0$, where $n$ depends on $A$.
\end{lemma}
\begin{proof}  
Lemma \ref{lemma 6.7}  implies that if $C_1$ is chosen to be sufficiently small then
    \begin{align*}
         \left|\left( \frac{\partial}{\partial r_1},  \frac{\partial}{\partial r_2}\right) A(\exp(r_1X_1+r_2X_2)n(z) a(t))   \right| \geq  |z| \ge \delta,
     \end{align*}
     when $|z|,|t|<C$, $|z|\ge\delta$ and $(r_1,r_2) \in \supp(\chi_0((C_1\delta)^{-1}\cdot))$. It is clear that all the derivatives of $A(\exp(r_1X_1+r_2X_2)n(z) a(t)) $ with respect to $r_1,r_2$ are bounded by the compactness. The lemma now follows from the following Lemma \ref{rescale nonstationary phase}.
\end{proof}

\begin{lemma}\label{rescale nonstationary phase}
    If $\chi \in C_c^\infty(\BR^2)$ is a cutoff function at scale $1\geq \delta_1 >0$, and $\phi \in C^\infty(\BR^2)$ is real-valued and satisfies
    \begin{align*}
        \nabla\phi(x) \gg \delta_2>0, \;\; \phi^{(n)}(x)\ll_n 1
    \end{align*}
    for $x\in \supp(\chi)$ and $n> 1$. If $\delta_1\ll\delta_2$, then
    \begin{align*}
        \int  \chi(x) e^{is\phi(x)} dx \ll_A \delta_1^2(s\delta_1\delta_2)^{-A}.
    \end{align*}
    In particular, if $\delta_1\asymp\delta_2\asymp\delta$, then
    \begin{align*}
        \int  \chi(x) e^{is\phi(x)} dx \ll_A \delta^2 (s\delta^2)^{-A}.
    \end{align*}
\end{lemma}
\begin{proof}
    We have
    \begin{align*}
        \int \chi(x) e^{is\phi(x)} dx = \delta_1^2 \int \chi(\delta_1 x) e^{i(s\delta_1\delta_2)(\delta_1\delta_2)^{-1}\phi(\delta_1 x)} dx.
    \end{align*}
    The function $\chi(\delta_1 x)$ is now a cutoff function at scale $1$, and $(\delta_1\delta_2)^{-1}\phi(\delta_1 x)$ satisfies
    \begin{align*}
        \nabla\left( (\delta_1\delta_2)^{-1}\phi(\delta_1 x) \right)\gg 1.
    \end{align*}
    Since $\delta_2\gg\delta_1$, the $n$-th derivatives are bounded
    when $x\in \supp(\chi(\delta_1\cdot))$ for all $n > 1$. The lemma now follows by integration by parts.
    Note that here we are not putting a bound on the first derivative of the phase function $(\delta_1\delta_2)^{-1}\phi(\delta_1 x)$. This is because the first derivative will only appear on the denominator after integrating by parts.
\end{proof}

Let $\chi_0 \in C_c^\infty(\BR)$ be a smooth function supported in $[-1,1]$, $\rho_1,\rho_2$ two real-valued smooth functions defined on $(-1,1)$, $s_1,s_2 \in \BR$ and $g\in G_0$. We define
\begin{align}\label{integral J(s t1 t2 g)}
\begin{split}
    &J(s,s_1,s_2,g;\chi_0,\rho_1,\rho_2)\\
    =& \iint_{-\infty}^\infty \chi_0(t_1) \chi_0(t_2) \exp\left({-is_1(t_1 +\rho_1(t_1))+ is_2(t_2+\rho_2(t_2))} \right)\varphi_s^{\BH^3} \left( a(-t_1)ga(t_2) \right) dt_1 dt_2.
\end{split}
\end{align}
The main result of this section is the following estimate of the oscillatory integral $J(s,s_1,s_2,g;\chi_0,\rho_1,\rho_2)$. We now combine the one-dimensional results Proposition \ref{first one dimensional intgeral estimate} and \ref{second one dimensional integral estimate} to prove it.

\begin{proposition}\label{nonstationary oscillatory integral estimate}
    Let $B_0,C_0,D >0$ and $1/8>\epsilon_0>0$ be given. If $g\in G_0$ satisfies
    \begin{align*}
        d(g,e)\leq C_0\quad\text{ and }\quad d(g,MA)\geq B_0 s^{-1/2 + \epsilon_0}\beta^{1/2}
    \end{align*}
    for some $s,\beta\geq 1$ satisfying $\beta \leq 3 s^{1/4}$, and
    \begin{align}\label{condition for rho two dimensional integral}
        \left|s_j \frac{\partial\rho_j}{\partial t_j} (t_j)\right|\leq Ds^{\epsilon_0}\beta,\quad\text{ and }\quad \left|s_j \frac{\partial^n\rho_j}{\partial t_j^n} (t_j)\right|\ll_n s^{\epsilon_0}\beta,\quad \text{ for }t_j\in (-1,1),\; j=1,2,
    \end{align}
    and $s_1,s_2\in [s-\beta,s+\beta]$, then
    \begin{align}\label{42}
        J(s,s_1,s_2,g;\chi_0,\rho_1,\rho_2)\ll s^{-A}.
    \end{align}
     The implied constant depends on $A,B_0,C_0,D, \epsilon_0$, the first $n$ implied constants in \eqref{condition for rho two dimensional integral},   and the size of the first $n$ derivatives of $\chi_0$, where $n$ depends on $\epsilon_0$ and $A$.
\end{proposition}
\begin{proof}
    We define the map $u: K_0\times \BR\to K_0$ by $ka(-t_1)\in NA u(k,t_1)$, i.e., $u(k,t_1) = \kappa(ka(-t_1))$. By Lemma \ref{change of variable}, for each fixed $t_1\in\BR$, $u(\cdot,t_1)$ is a diffeomorphism from $K_0$ onto itself and induces a diffeomorphism from $M\backslash K_0$ onto itself. If we apply the functional equation $\varphi_s^{\BH^3} = \varphi_{-s}^{\BH^3}$, write $\varphi_{-s}^{\BH^3}$ as an integral over $K_0$ by the integral formula \eqref{eq:HC int for H3}, and apply Lemma \ref{spliting A}, we obtain
    \begin{align*}
        &\varphi_{s}^{\BH^3}\left(  a(-t_1)ga(t_2)   \right)\\
        =& \int_{M\backslash K_0} \exp\big( (1-is)   (A\left( ka(-t_1)ga(t_2) \right) \big) dk  \\
        =& \int_{M\backslash K_0} \exp\big( (1-is)   \left(A\left( u(k,t_1)ga(t_2) \right) - A\left( u(k,t_1)a(t_1)\right) \right)\big) dk\\
        =&\int_{M\backslash K_0} \exp\big( (1-is)   \left(A\left( uga(t_2) \right) - A\left( ua(t_1)\right) \right)\big) \left|\det \frac{\partial k}{\partial u}\right| du.
    \end{align*}
    Here we use $\partial u/\partial k$ to denote the Jacobian matrix of the diffeomorphism $u(\cdot,t_1):M\backslash K_0\to M\backslash K_0$, and $\partial k/\partial u$ is the inverse matrix.
    Substituting this into the definition (\ref{integral J(s t1 t2 g)}) of $J(s,s_1,s_2,g;\chi_0,\rho_1,\rho_2)$ gives
    \begin{align}\label{44}
    \begin{split}
         \iint_{-\infty}^\infty \int_{M\backslash K_0}\chi_0(t_1) \chi_0(t_2) &\exp\left({-is_1(t_1 +\rho_1(t_1))+is_2(t_2+\rho_2(t_2))} \right)\\
         &\exp\big( (1-is)   \left(A\left( uga(t_2) \right) - A\left( ua(t_1)\right) \right)\big) \left|\det \frac{\partial k}{\partial u}\right| du dt_1 dt_2 . 
         \end{split}
    \end{align}
    Let $g = k^\prime n(z^\prime) a (t^\prime)$ with $k^\prime \in K_0$, $z^\prime\in\BC$ and $t^\prime\in\BR$. 
    The condition $d(g,e)\leq C_0$ implies that $z^\prime$ and $t^\prime$ are bounded. 
    Choose a constant $C>0$. 
    If $d(u,M) \geq Cs^{-1/2 + \epsilon_0}\beta^{1/2}$, then integrating (\ref{44}) in $t_1$ and applying Proposition \ref{first one dimensional intgeral estimate} shows that the integral of (\ref{44}) over $t_1$ and $t_2$ with this $u$ is $\ll_{C,A,\epsilon_0} s^{-A}$.
    And if $d(u,M) < Cs^{-1/2 + \epsilon_0}\beta^{1/2}$ but $d(k^\prime,M) \geq C_1 C s^{-1/2 + \epsilon_0}\beta^{1/2}$ for some constant $C_1>0$ independent of the choice of $C$ such that $d(uk^\prime,M) \geq Cs^{-1/2 + \epsilon_0}\beta^{1/2}$, then we obtain the same conclusion by integrating in $t_2$.
    Combining these, we see that (\ref{44}) will be $\ll_{C,A,\epsilon_0} s^{-A}$ unless $d(k^\prime,M) < C_1 C s^{-1/2 + \epsilon_0}\beta^{1/2}$, and we assume that this is the case.

    If $C$ is chosen sufficiently small, the condition $d(k^\prime,M) < C_1 C s^{-1/2 + \epsilon_0}\beta^{1/2}$ and our assumption that $d(g,MA)\geq B_0 s^{-1/2 + \epsilon_0}\beta^{1/2}$ imply that $|z^\prime| \geq C_2 s^{-1/2 + \epsilon_0}\beta^{1/2} $ for some $C_2>0$ depending only on $B_0$. 
    For $-1\leq t_1 \leq 1$, we define $k^\prime(t_1) \in K_0$, $z^\prime(t_1)\in\BC$ and $t^\prime(t_1) \in \BR$ by  
    \begin{align*}
        a(-t_1)g = a(-t_1)k^\prime n(z^\prime) a (t^\prime) =k^\prime(t_1)n(z^\prime(t_1)) a(t^\prime(t_1)),
    \end{align*}
    so that
    \begin{align*}
        a(-t_1)k^\prime = k^\prime(t_1)n(z^\prime(t_1)-e^{t^\prime(t_1)-t'}z') a(t^\prime(t_1)-t').
    \end{align*}
    Since $d( a(-t_1)k^\prime,MA) < C_1 C s^{-1/2 + \epsilon_0}\beta^{1/2}$, it may be seen that 
    \begin{align*}
        \left| z^\prime(t_1)-e^{t^\prime(t_1)-t'}z' \right| < C_3 C s^{-1/2 + \epsilon_0}\beta^{1/2}
    \end{align*}
    for some  $C_3>0$.
    It follows from $|z^\prime| \geq C_2 s^{-1/2 + \epsilon_0}\beta^{1/2} $ that if $C$ is sufficiently small, 
    we have  $|z^\prime(t_1)| \geq C_4 s^{-1/2 + \epsilon_0}\beta^{1/2} $ for some  $C_4>0$ and for all $-1\leq t_1 \leq 1$. 
    The result now follows by applying Proposition \ref{second one dimensional integral estimate} to the integral (\ref{integral J(s t1 t2 g)}) for each fixed $t_1$.
\end{proof}

\section{Geodesics in \texorpdfstring{$\BH^3$}{H3}}\label{section of geodesic in H3}

Let $l$ be the image of $A$ in $\BH^3$, which is the vertical geodesic through the origin.
We choose an orientation for $l$ so that the upward orientation of $l$ is positive.
Given $T,\delta >0$, we define the $(\delta,T)$ tubular neighborhood of $l$ centered at $o$ as
\begin{align*}
    \sT^T_\delta(l,o) =\{ (z,e^t)\in \BH^3\, :\, |z|\leq\delta,|t|\leq T \}= \bigcup_{z\in\BC,|z|\leq\delta}n(z)\cdot \{ a(t)\cdot o\,:\, |t| \leq T \},
\end{align*}
which is the union of vertical geodesics that are close to $l$ near the point $o$. Here $\delta$ is the "radius" of the discs of the tube, and $T$ is the length of the geodesic segment.
We will denote by $l_{[-T,T]}$ the oriented geodesic segment
\begin{align*}
    l_{[-T,T]} = \{a(t)\cdot o\,|\, -T\leq t\leq T \}
\end{align*}
with the orientation as $t$ increases.
In the general case, consider an oriented geodesic $\gamma$ in $\BH^3$ with a point $p$ on $\gamma$. 
There exists a  $g\in G_0$ so that $\gamma = g\cdot l$ preserving the orientations and $p=g\cdot o$.
We define the $(\delta,T)$ tubular neighborhood of $\gamma$ centered at $p$ to be
\begin{align*}
    \sT^T_\delta(\gamma,p) = g \cdot \sT^T_\delta(l,o)=\bigcup_{z\in\BC,|z|\leq\delta}gn(z)\cdot \{ a(t)\cdot o\,:\, |t|\leq T \}.
\end{align*}
We define the projection along the horocycles from the tube onto the segment to be
\begin{align}\label{eq: horocycle projection}
    \pi_\gamma: \sT^T_\delta(\gamma,p) \to \gamma,\qquad gn(z)a(t)\cdot o\mapsto ga(t)\cdot o.
\end{align}


Note that $MA$ consists of all orientation-preserving isometries of $l$ in $G_0$. 
Recall that $M^\prime$ is the normalizer of $A$ in $K_0$.
All isometries of $l$ in $G_0$, which do not need to preserve orientation, form the group
\begin{align*}
   M^\prime A =  \{ e, w_0 \}MA,
\end{align*}
where $e$ is the identity and $w_0$ is a representative of the nontrivial element in the Weyl group. We will take
\begin{align*}
    w_0 = \begin{pmatrix}
        &i\\i&
    \end{pmatrix}.
\end{align*}
Using geodesic tubes, we give a geometric description of the condition for a group element $g\in G_0$ to be close to $MA$.

\begin{lemma}\label{tube inclusion lemma}
    Let $T>0$ be given. There exist constants $C,D>0$, depending only on $T$, so that
    \begin{align*}
          k\cdot\sT_\delta^{T/2}( l,o)=\sT_\delta^{T/2}(k\cdot l,o) \subset \sT_{C\delta}^T(l,o)
    \end{align*}
    for all $0<\delta<D$ and $k\in K_0$ satisfying $d(k,M^\prime)\leq \delta$.
\end{lemma}
\begin{proof}
    The group $M^\prime = M\cup w_0 M$ has two connected components.
    For an element $k\in K_0$ satisfying $d(k,M^\prime)\leq \delta$,
    there are two cases, either $d(k,M)\leq \delta$, or $d(k,w_0M)\leq \delta$.
    We first assume $d(k,M)\leq \delta$.
    We may choose a constant $D$ sufficiently small so that the exponential map is a diffeomorphism when restricted to any $\delta$-neighborhood of the origin in $K_0/M$. We then assume $k \in \exp(r_1X_1+r_2X_2) M$ with $r = \sqrt{r_1^2+r_2^2} \leq C_1 \delta$ for some constant $C_1>0$. Since $M$ stablizes $\sT_\delta^{T/2}(l,o)$, we assume $k = \exp(r_1X_1+r_2X_2)$. By (\ref{formula of exp(r1X1+r2X2)}), we get
    \begin{align*}
       k = \begin{pmatrix}
            \cos (r) & (ir_1 - r_2)\sin (r)/r\\
            (ir_1 + r_2)\sin (r)/r & \cos (r)
            \end{pmatrix}.
    \end{align*}
    Let $(z,e^t)$ be an arbitrary point in $\sT_\delta^{T/2}( l,o)$, i.e. $|z|\leq \delta$ and $|t|\leq T/2$. Via the formula (\ref{Poincare}), the Iwasawa $N$-projection and $A$-projection of $kn(z)a(t)$ are
    \begin{align*}
        N(kn(z)a(t))=\frac{z+O_T(r)}{1+O_T(r)} = z+O_T(r)
    \end{align*}
    and
    \begin{align*}
        A(kn(z)a(t))=\log \frac{e^t}{1+O_T(r)} = t + O_T(r).
    \end{align*}
    Then the result follows.
    If $d(k,w_0M)\leq \delta$, we can assume
    \begin{align*}
       k = w_0\begin{pmatrix}
            \cos (r) & (ir_1 - r_2)\sin (r)/r\\
            (ir_1 + r_2)\sin (r)/r & \cos (r)
            \end{pmatrix}.
    \end{align*}
    The same argument finishes the proof.
\end{proof}

\begin{proposition}\label{geodesic tube theorem for close to MA}
    Let $C_0 > 0$ be given. There exist constants $C,D,T>0$, depending only on $C_0$, such that the following holds. If  $0<\delta<D$ and $g\in G_0$ satisfying $d(g,e)\leq C_0$ and $d(g,M^\prime A)\leq \delta$, then the geodesic segment
    \begin{align*}
        g\cdot l_{[-1,1]} = \{ g a(t)\cdot o\,|\,-1\leq t\leq 1  \}
    \end{align*}
    lies inside the tube $\sT_{C\delta}^T(l,o)$.
\end{proposition}
\begin{proof}
    Let $g= kn(z)a(t)$ with $k\in K_0$, $z\in\BC$ and $t\in\BR$. The condition $d(g,e)\leq C_0$ implies that $z$ and $t$ are bounded. Choose $T>0$ sufficiently large so that $|t+1|<T/2$ and $|t-1| < T/2$ hold for all possible choices of $t$ for $d(g,e)\leq C_0$. 
    The condition $d(g,M^\prime A)\leq \delta$ implies that $d(k,M^\prime)\leq C_1 \delta$ and $|z|\leq C_1 \delta$ for some constant $C_1 > 0$.
    Then we have
    \begin{align*}
       & n(z)a(t)\cdot l_{[-1,1]}\\
       =& \{ n(z) a(t+t_0)\cdot o\,|\,-1\leq t_0\leq 1  \}\\
        \subset& \{ (z,e^t)\,|\, |z|\leq C_1\delta,|t|\leq T/2  \} \\
        =& \sT_{C_1\delta}^{T/2}(l,o) 
    \end{align*}
    Now the result follows by Lemma \ref{tube inclusion lemma}.
\end{proof}

Next, we show the converse of Proposition \ref{geodesic tube theorem for close to MA}.
\begin{proposition}\label{tube inclusion implies close in moduli space}
    Let $C_0, T>0$ be given. There exist constants $C,D>0$, depending on $C_0$ and $T$, such that the following holds. If  $0<\delta<D$ and $g\in G_0$ satisfying $d(g,e)\leq C_0$ and 
    \begin{align*}
        g\cdot l_{[0,1]} = \{ g a(t)\cdot o\mid 0\leq t\leq 1 \} \subset \sT_{\delta}^T(l,o),
    \end{align*}
    then
    \begin{align*}
        d(g,M^\prime  A) \leq C\delta.
    \end{align*}
\end{proposition}
\begin{proof}
    We write $g = k^\prime n(z^\prime) a(t^\prime)$ with $k^\prime = \begin{pmatrix}
        \alpha&\beta\\-\bar{\beta}&\bar{\alpha}
    \end{pmatrix} \in K_0$, $z^\prime\in\BC$ and $t^\prime\in\BR$ by the Iwasawa decomposition.
    It suffices to prove $d(k^\prime,M^\prime)\ll \delta$ and $|z^\prime|\ll \delta$.
    Since $d(g,e)\leq C_0$, both $z^\prime$ and $t^\prime$ can only vary in a bounded set, say $|z^\prime|\leq C_1$ and $|t^\prime|\leq C_1$.
    We may assume $T>10(C_1 + 1)$. 
    Then by \eqref{Poincare},
    \begin{align*}
        g\cdot l_{[0,1]} &= k^\prime\cdot \{ (z^\prime, e^t)\, |\, t^\prime\leq t \leq t^\prime +1 \}\\
        &=\left.\left\{\left(\frac{(\alpha z^\prime+\beta)\overline{(-\bar{\beta}z^\prime+\bar{\alpha})}-\bar{\alpha}\beta e^{2t}}{|-\bar{\beta}z^\prime+\bar{\alpha}|^2+|\beta|^2e^{2t}},\frac{e^t}{|-\bar{\beta}z^\prime+\bar{\alpha}|^2+|\beta|^2e^{2t}} \right) \right|   t^\prime\leq t \leq t^\prime +1\right\}\\
        &\subset \sT_{\delta}^T(l,o)
    \end{align*}
    implies that, for $t^\prime\leq t \leq t^\prime +1$,
    \begin{align}\label{bound for z}
        \left|\frac{(\alpha z^\prime+\beta)\overline{(-\bar{\beta}z^\prime+\bar{\alpha})}-\bar{\alpha}\beta e^{2t}}{|-\bar{\beta}z^\prime+\bar{\alpha}|^2+|\beta|^2e^{2t}}\right| \leq \delta,
    \end{align}
    and
    \begin{align}\label{bound for t}
        e^{-T}\leq \frac{e^t}{|-\bar{\beta}z^\prime+\bar{\alpha}|^2+|\beta|^2e^{2t}}  \leq e^T.
    \end{align}
    The assumptions on $T$ and $t^\prime$ with (\ref{bound for t}) imply that
    \begin{align*}
        e^{-2T}\leq |-\bar{\beta}z^\prime+\bar{\alpha}|^2+|\beta|^2e^{2t} \leq e^{2T},
    \end{align*}
    so (\ref{bound for z}) gives
    \begin{align}\label{tube inequality for z}
       \left| (\alpha z^\prime+\beta)\overline{(-\bar{\beta}z^\prime+\bar{\alpha})}-\bar{\alpha}\beta e^{2t}\right| \leq e^{2T}\delta,
    \end{align}
    for $t^\prime\leq t \leq t^\prime +1$.
    We can plug in $t = t^\prime +1$ and $t = t^\prime$ for \eqref{tube inequality for z}. 
    This gives
    \begin{align*}
        \left| \bar{\alpha}\beta (e^{2(t^\prime+1)}-e^{2t^\prime}) \right| \leq & \left| (\alpha z^\prime+\beta)\overline{(-\bar{\beta}z^\prime+\bar{\alpha})}-\bar{\alpha}\beta e^{2(t^\prime+1)}\right| + \left| (\alpha z^\prime+\beta)\overline{(-\bar{\beta}z^\prime+\bar{\alpha})}-\bar{\alpha}\beta e^{2t^\prime}\right|\\
        \leq & 2 e^{2T }\delta.
    \end{align*}
    Hence, $|\alpha\beta|\ll \delta$.
    Since $|\alpha|^2 + |\beta|^2 = 1$ and $|\alpha\beta|\ll \delta$,  either $|\alpha|\ll \delta$ or $|\beta|\ll \delta$. Hence, $d(k^\prime,M^\prime) \ll \delta$.
    Combing this with (\ref{tube inequality for z}) again gives  $|z^\prime|\ll \delta$.
\end{proof}

Next we consider the following family of geodesics, which we will care about in Section \ref{section geodesic beam}.
Let $\epsilon_1>0$ be a small constant.
We will let $N_1$ and $N_2$ be two large integers so that $N_1 = \lambda^{1/2-\epsilon_1}\beta^{-1/2}+ O(1)$ and $N_2 = \lambda^{1/2}\beta^{-1/2}+O(1)$.
We define $N_1$ many intervals covering $B$
\begin{align}\label{eq: defn of Im}
    \bI_m =\left\{b_\theta\in B\, :\,\theta \in \left[\frac{2\pi }{N_1}(m-2/3) ,\frac{2\pi }{N_1} (m+2/3) \right] \mod 2\pi  \right\}, \quad m=0,\dots,N_1-1,
\end{align}
and define $N_2+1$ many intervals in $\BR$ that cover $[-2,2]$
\begin{align}\label{eq: defn of Jn}
    \bJ_n = \left[\frac{4 }{N_2}(n-N_2/2-2/3) ,\frac{4 }{N_2} (n-N_2/2+2/3) \right], \;\;\; n=0,\dots,N_2.
\end{align}
We let
\begin{align}\label{eq: defn of bm and xn}
\begin{split}
    b_m\in B \text{ be the center point of }\bI_m, \\
    \text{ and }x_n\in\BR \text{ be the center point of }\bJ_n.  
\end{split}
\end{align}

Each pair $(b,x)\in B\times \BR$ uniquely determines an oriented geodesic $b n(x) \cdot l$ in $\BH^2$. 
Moreover, the next lemma shows that if $b_1,b_2\in B$, $x_1,x_2\in\BR$, and the element $(b_1n(x_1))^{-1}b_2n(x_2)$ almost fix the geodesic $l$ and its orientation, then the geodesics determined by $(b_1,x_1)$ and $(b_2,x_2)$ are almost the same.

\begin{lemma}\label{close pairs (b,x)}
    Let $C>0$ be given.
    There exists a small constant $D>0$, depending only on $C$, such that the following holds.
    Suppose $b_1, b_2\in B = \SO(2)/\{\pm I\}$ and $x_1,x_2\in [-C,C]$. If $0<\delta<D$  and
    \begin{align*}
        d \left((b_1n(x_1))^{-1}b_2n(x_2),M^\prime A\right)\leq \delta,
    \end{align*}
    then 
    \begin{enumerate}
        \item If $ d \left((b_1n(x_1))^{-1}b_2n(x_2),M A\right)\leq \delta$, we have $d(b_1, b_2) \ll \delta$ and $|x_1-x_2|\ll \delta$;
        \item If $ d \left((b_1n(x_1))^{-1}b_2n(x_2),w_0M A\right)\leq \delta$, we have $d\left(b_1^{-1}b_2,\pm  \frac{1}{\sqrt{1+x_1^2}}\begin{pmatrix}
       x_1&1\\-1&x_1
    \end{pmatrix}\right) \ll \delta$ and $|x_1+x_2|\ll\delta$.
    \end{enumerate}
    Here all the implied constants depend only on $C$.
\end{lemma}
\begin{proof}
    We suppose 
    $$b_1^{-1} b_2 = b_\theta =\begin{pmatrix} \cos\theta/2 & \sin\theta/2\\ -\sin\theta/2& \cos\theta/2 \end{pmatrix},\;\;\;\text{ with }\theta\in\BR/2\pi\BZ. $$
    Then
    \begin{align*}
        (b_1n(x_1))^{-1}b_2n(x_2) &= n(-x_1) b_\theta n(x_2) \\
        &=\begin{pmatrix} 1 & -x_1\\ & 1 \end{pmatrix}\begin{pmatrix} \cos\theta/2 & \sin\theta/2\\ -\sin\theta/2& \cos\theta/2 \end{pmatrix}\begin{pmatrix} 1 & x_2\\ & 1 \end{pmatrix} \\
        &=\begin{pmatrix}
            \cos(\theta/2) + x_2\sin(\theta/2)& (x_1-x_2)\cos(\theta/2) + (x_1x_2+1)\sin(\theta/2)\\
            -\sin(\theta/2) & \cos(\theta/2) - x_1\sin(\theta/2)
        \end{pmatrix}.
    \end{align*}
    
    We first assume that $d( (b_1n(x_1))^{-1}b_2n(x_2) ,MA)\leq \delta$.
    Then 
    \begin{align*}
        \sin(\theta/2) \ll \delta\quad\text{ and }\quad (x_1-x_2)\cos(\theta/2) + (x_1x_2+1)\sin(\theta/2)\ll \delta,
    \end{align*}
    so the results $d(b_1,b_2)\ll \delta$ and $|x_1-x_2|\ll\delta$ follow. 
    
    Next, we assume $d( (b_1n(x_1))^{-1}b_2n(x_2) ,w_0MA)\leq \delta$.
    We have
    \begin{align}\label{eqn for geodesic reverse}
        \cos(\theta/2) + x_2\sin(\theta/2)\ll\delta,\quad\text{ and }\quad \cos(\theta/2) - x_1\sin(\theta/2)\ll \delta.
    \end{align}
    We claim that $|\sin(\theta/2)|\geq c >0$ for some small constant $c$. If we suppose that $|\sin(\theta/2)|< c$, then \eqref{eqn for geodesic reverse} implies that $|\cos(\theta/2)|\leq 2c + O(\delta)$, which is a contradiction if $D$ and $c$ are chosen to be sufficiently small.
    Moreover, we have
    \begin{align*}
        |x_1+x_2|\ll|(x_1+x_2)\sin(\theta/2)| \leq  | \cos(\theta/2) + x_2\sin(\theta/2)| + |\cos(\theta/2) - x_1\sin(\theta/2)| \ll \delta.
    \end{align*}
    The identity $\cos^2(\theta/2)+\sin^2(\theta/2) = 1$ implies that
    \begin{align*}
        \left(x_1\sin(\theta/2)+O(\delta)\right)^2 + \sin^2(\theta/2) = 1
    \end{align*}
    so that
    \begin{align*}
        \sin^2(\theta/2) = \frac{1}{1+x_1^2} + O(\delta) = \frac{1}{1+x_1^2} (1+O(\delta)) \Rightarrow  \sin(\theta/2) = \pm \frac{1}{\sqrt{1+x_1^2}} + O(\delta).
    \end{align*}
    By \eqref{eqn for geodesic reverse}, we have $\cos(\theta/2) = x_1\sin(\theta/2)+O(\delta) = \pm \frac{x_1}{\sqrt{1+x_1^2}} + O(\delta)$.
\end{proof}

We next derive a trivial bound, in the following Proposition \ref{a tube not contain many geodesic beam}, for the number of geodesic segments corresponding to $(b_m,x_n)$ contained in a tube.
Suppose $\gamma$ is an oriented geodesic in $\BH^3$ with a point $p\in\gamma$, and $\delta > 0$.
Moreover, let $C_0 > 0$ be a fixed constant,
and we assume that $(\gamma,p) = g\cdot (l,o)$ with $d(g,e)\leq C_0$.
We let  $N_\delta^T(\gamma,p)$ be the set of pairs $(b_m,x_n)$, $0\leq m\leq N_1 -1$, $0\leq n\leq N_2$, satisfying
    \begin{align*}
        b_m n(x_n) \cdot l_{[-1,1]} =  \{ b_m n(x_n) a(t)\cdot o\,|\,-1\leq t\leq 1  \} \subset \sT_{\delta}^T(\gamma,p).
    \end{align*}

\begin{lemma}\label{geodesic in a tube then a larger tube contains the tube}
    Let $C_0, T>0$ be given. 
    There exist constants $C,D>0$, depending on $C_0, T$,  such that the following holds.
    For any $0<\delta<D$,
    if $l_{[-1,1]}$ is contained in some tube $\sT_{\delta}^T(\gamma,p) = g\cdot \sT_\delta^T(l,o)$ with $g\in G_0$ and $d(g,e)\leq C_0$,
    then $\sT_{\delta}^T(\gamma,p) \subset  \sT_{C\delta}^{CT}(l,o)$.
\end{lemma}
\begin{proof}
   Since $g^{-1} l_{[-1,1]} \subset \sT^T_\delta(l,o)$, by Proposition \ref{tube inclusion implies close in moduli space}, $d(g,M^\prime A) \ll\delta$.
   Hence, we can write $g = k^\prime n(z^\prime) a(t^\prime)$
   with $k^\prime\in K_0$, $z^\prime\in\BC$ and $t^\prime\in\BR$ satisfying
   $d(k^\prime,M^\prime) \ll \delta$, $|z^\prime| \ll \delta$ and $|t^\prime| \ll 1$.
   Hence, there exists $C_1 > 0$ so that
   \begin{align*}
       g\cdot \sT_\delta^T(l,o) = k^\prime n(z^\prime) a(t^\prime)\cdot \sT_\delta^T(l,o) \subset k^\prime \cdot \sT_{C_1\delta}^{C_1T}(l,o).
   \end{align*}
   Then the lemma follows from Lemma \ref{tube inclusion lemma}.
\end{proof}

\begin{proposition}\label{a tube not contain many geodesic beam}
    Let $C_0,T>0$ be given. There exists a small constant $D>0$, depending on $C_0$ and $T$,  such that for any $0 < \delta <D$, and any $(\gamma,p)=g\cdot (l,o)$ with $d(g,e)\leq C_0$,
    \begin{align*}
        \# N_\delta^T(\gamma,p)\ll  (1+\delta N_1)(1+\delta N_2).
    \end{align*}
    Here the implied constant also depends on $C_0$ and $T$.
\end{proposition}
\begin{proof}
    If $N_\delta^T(\gamma,p)$ is empty, then there is nothing to prove, so we assume that there exists some $m,n$ so that  $(b_m,x_n) \in N_\delta^T(\gamma,p)$, i.e. $b_m n(x_n) \cdot l_{[-1,1]} \subset \sT_{\delta}^T(\gamma,p)$.
    By Lemma \ref{geodesic in a tube then a larger tube contains the tube}, there exists a constant  $C>0$ so that $\sT_{\delta}^T(\gamma,p) \subset b_m n(x_n) \cdot \sT_{C\delta}^{CT}(l,o)$. Therefore, the problem can be reduced to prove the case where $(\gamma,p) = b_m n(x_n) \cdot(l,o) $.

    We assume $(\gamma,p) = b_m n(x_n) \cdot(l,o) $. Suppose $(b_{m^\prime},x_{n^\prime}) \in  N_\delta^T(\gamma,p)$. Then we have
    \begin{align*}
        &b_{m^\prime} n(x_{n^\prime}) \cdot l_{[-1,1]} \subset  b_m n(x_n) \cdot\sT_{\delta}^T(l,o)\\
        \Rightarrow &\left( b_m n(x_n)  \right)^{-1} b_{m^\prime} n(x_{n^\prime}) \cdot l_{[-1,1]} \subset  \sT_{\delta}^T(l,o)
    \end{align*}
    By Proposition \ref{tube inclusion implies close in moduli space}, there exists $C>0$ so that
    \begin{align*}
        d\left(\left( b_m n(x_n)  \right)^{-1} b_{m^\prime} n(x_{n^\prime}),    M^\prime A\right)\leq C\delta.
    \end{align*}
    Lemma \ref{close pairs (b,x)}
    implies  that the number of possible $b_{m^\prime}$ is $\ll 1+\delta N_1$, and 
    the number of possible $x_{n^\prime}$ is $\ll 1+\delta N_2$.
\end{proof}

\begin{proposition}\label{not so many tube exists in intersection}
     Let $C_0>0$ and $T>2$ be given. There exist constants $C, D, D'>0$, depending on $C_0, T$,  such that the following holds. Let $\delta,\omega>0$, and $g\in G_0$, which satisfies $d(g,e)\leq C_0$ and $d(g,H^\prime) \geq C\omega$.  If $\delta$ and $\omega$ satisfy
     \begin{align*}
        \delta,\omega <D\quad\text{ and }\quad \delta < D' \omega,
     \end{align*} 
     the number of pairs $(b_m,x_n)$ (which are defined in \eqref{eq: defn of bm and xn}) so that
     \begin{align}\label{condition for the pair of bm xn}
          \text{there is a geodesic segment of length }2\text{ contained in }b_mn(x_n)\cdot \sT_{\delta}^T(l,o)\cap g\BH^2
     \end{align}
     is $\ll  (1+\omega^{-1}\delta N_1) (1+\omega^{-1}\delta N_2)$.
     Here the implied constant depends on $C_0$ and $T$.
\end{proposition}

\begin{proof}

     We take $R>0$ so that the ball  $B(o,R)\subset \BH^2$ centered at $o$ of radius $R$ contains $b_mn(x_n)\cdot l_{[-1,1]}$ for any $m,n$. If $B(o,2R)\cap g\BH^2=\emptyset$, then
     \begin{align*}
         d(B(o,R),g\BH^2)\gg_{R,C_0} d(g,H')\geq C\omega >\frac{C}{D'} \delta.
     \end{align*}
     If we suppose that $C/D'$ is sufficiently large so that $b_mn(x_n)\cdot \sT_{\delta}^T(l,o) \cap g\BH^2 = \emptyset$, then, in this case, the number of pairs satisfying \eqref{condition for the pair of bm xn} is 0.

     Now we suppose that $B(o,2R)\cap g\BH^2\neq\emptyset$. In particular, $\BH^2\cap g\BH^2=\gamma$ is a geodesic.  
     Since a segment of $\gamma$ lies in the ball $B(o,2R)$, there exists $h_0\in H_0$, which is bounded depending on $R$, so that $h_0\gamma = l$ is the vertical geodesic through the origin $o$. Hence, $l = h_0\gamma = h_0 (\BH^2 \cap g\BH^2 ) = \BH^2 \cap h_0g\BH^2 $. Hence, $h_0g\BH^2 $ is a vertical plane in $\BH^3$ through $l$. We may find an element $k(\theta) = \begin{pmatrix} e^{i\theta}&\\&e^{-i\theta} \end{pmatrix}\in M$ with $\theta\in\BR/2\pi\BZ,$ so that $h_0 g\BH^2 = k(\theta)\BH^2$. Since $k(\theta)^{-1}h_0 g$ stabilizes $\BH^2$, we have $k(\theta)^{-1}h_0 g \in H^\prime$. Using the Iwasawa decomposition for $H_0$, we can write
     $$g = \big(k_0 n(x_0)a(t_0)\big)  k(\theta) h' = k_0 n(x_0) k(\theta)a(t_0) h'$$
     where $k_0 \in \mathrm{SO}(2)$, $t_0,x_0\in\BR$ are bounded, and $h'\in H'$.
     In particular, $h'$ is also bounded because of the boundedness of $g$.
     Since both $k_0 n(x_0)$ and $a(t_0)h'$ are in $H'$,
     \begin{align*}
         C\omega \leq d(g,H') = d(k_0 n(x_0)k(\theta)a(t_0)h',H') \asymp d(k(\theta),H').
     \end{align*}
     If $C$ is sufficiently large, it may be seen that $\theta \notin  [- \omega,\omega] \mod \pi/2$. Moreover, as 
     $$g\BH^2 = k_0n(x_0)k(\theta)\big(a(t_0)h'\BH^2\big)=k_0n(x_0)k(\theta)\BH^2,$$
     without loss of generality, we reduce to the case where $g=k_0n(x_0)k(\theta)$ with $k_0\in\mathrm{SO}(2)$, $x_0\in\BR$ being bounded and $\theta \notin  [- \omega,\omega] \mod \pi/2$.
     Assume that $(b_m,x_n)$ is a pair satisfying \eqref{condition for the pair of bm xn}, i.e. there is a geodesic segment of length 2 contained in $b_mn(x_n)\cdot \sT_{\delta}^T(l,o)\cap g\BH^2$, and we denote the endpoints of such segment by $p_1$ and $p_2$. Because $p_i \in b_mn(x_n)\cdot \sT_\delta^T(l,o)$, we have $d(p_i,q_i)\ll \delta$, where $q_i=\pi_{b_mn(x_n)\cdot l}(p_i)$ and $\pi_{b_mn(x_n)\cdot l}$ is the projection map to $b_mn(x_n)\cdot l$ defined as in \eqref{eq: horocycle projection}. Since $p_i\in g\BH^2$, we have $d(q_i,g\BH^2)\leq d(q_i,p_i)\ll \delta$. Therefore, there exists another constant $C_2>0$ so that
     \begin{align*}
         q_i \in g\cdot B_{C_2\delta}(\BH^2),\quad\text{ where }B_{C_2\delta}(\BH^2): = \{n(x+iy)a(t)\cdot o\mid x,y,t\in\BR,\, |y|\leq C_2\delta  \}.
     \end{align*}
     Therefore, $q_i$ lies in 
     \begin{align*}
         \big(g\cdot B_{C_2\delta}(\BH^2)\big)\cap\BH^2 &= \big(k_0n(x_0)k(\theta)\cdot B_{C_2\delta}(\BH^2)\big)\cap\BH^2=k_0n(x_0)\Big(\big(k(\theta)\cdot B_{C_2\delta}(\BH^2)\big)\cap\BH^2\Big).
     \end{align*}
     By a direct trigonometric calculation with the group action formula \eqref{Poincare}, it may be seen that
     \begin{align*}
         T_{\theta,C_2\delta}&:=\big(k(\theta)\cdot B_{C_2\delta}(\BH^2)\big)\cap\BH^2=\{n(x)a(t)\cdot o\mid x,t\in\BR,\, |x|\leq C_2\delta\csc(2\theta)\}.
     \end{align*}
     The same argument shows that for every point $p$ lying in the geodesic segment between $p_1,p_2$, its projection $\pi_{b_mn(x_n)\cdot l}(p)$ is in $k_0n(x_0)T_{\theta,C_2\delta}$ as well, i.e., the geodesic segment joining $q_1,q_2$ (which is a segment of the geodesic $b_mn(x_n)\cdot l$) is contained in $k_0n(x_0)T_{\theta,C_2\delta}$. Moreover, the triangle inequality implies that 
     \begin{align*}
         d(q_1,q_2)\geq d(p_1,p_2) - d(p_1,q_1)-d(p_2,q_2) = 2-O(\delta)\geq1
     \end{align*}
     if $D$ is sufficiently small. It follows that there is some $t_1\in\BR$, which is bounded depending on $C_0$ and $T$, such that $b_mn(x_n)a(t_1)\cdot l_{[0,1]} \subset k_0n(x_0)T_{\theta,C_2\delta}$, which implies that $n(x_0)^{-1}k_0^{-1}b_mn(x_n)a(t_1)\cdot l_{[0,1]} \subset T_{\theta,C_2\delta}$
     By the boundedness of $n(x_0)^{-1}k_0^{-1}b_mn(x_n)a(t_1)\cdot l_{[0,1]}$, there exists $T_1>0$, depending on $C_0$ and $T$, so that
     \begin{align*}
         n(x_0)^{-1}k_0^{-1}b_mn(x_n)a(t_1)\cdot l_{[0,1]} \subset \{n(x)a(t)\cdot o\mid  |x|\leq C_2\delta\csc(2\theta), |t|\leq T_1\} = \sT^{T_1}_{C_2\delta\csc(2\theta)}(l,o).
     \end{align*}
     By Proposition \ref{tube inclusion implies close in moduli space}, we have
    \begin{align*}
        d\left(n(x_0)^{-1}k_0^{-1}b_mn(x_n)a(t_1),    M^\prime A\right)\ll\delta\csc(2\theta)\ll \omega^{-1}\delta        \Rightarrow d\left(\left(k_0n(x_0)\right)^{-1}b_mn(x_n),    M^\prime A\right)\ll \omega^{-1}\delta.
    \end{align*}
    Lemma \ref{close pairs (b,x)}
    implies  that the number of possible $b_{m}$ is $\ll 1+\omega^{-1}\delta N_1$, and 
    the number of possible $x_{n}$ is $\ll 1+\omega^{-1}\delta N_2$.
\end{proof}

The main result in this section is that most of $b_{m_1},x_{n_1},b_{m_2},x_{n_2}$ will satisfy
 the condition
\eqref{uniform condition}, provided certain condition that $g\in G_0$ is away from $H^\prime$. 
Recall that $b_m, x_n$ are defined as in \eqref{eq: defn of bm and xn}.

\begin{proposition}\label{most of geodesic beams satisfy uniform condition}
    Let $C_0>0$ be given.  
    There exist constants $C, D, D'>0$, depending on $C_0$ such that the following holds. 
    Suppose that $\delta,\omega>0$ satisfy
     \begin{align*}
        \delta,\omega <D\quad\text{ and }\quad \delta < D' \omega.
     \end{align*}
    Let $g\in G_0$ satisfy
    \begin{align*}
         d(g,e)\leq C_0\quad \text{ and }\quad d(g,H^\prime) \geq C\omega.
    \end{align*}
    Then the number of quadruples $(b_{m_1},x_{n_1},b_{m_2},x_{n_2})$ satisfying
    \begin{align}\label{quadruple condition}
        d(n(-x_{n_1})b_{m_1}^{-1}gb_{m_2}n(x_{n_2}), MA) \leq \delta
    \end{align}
    is $\ll (1+\delta N_1)(1+\delta N_2)(1+\omega^{-1}\delta N_1) (1+\omega^{-1}\delta N_2)$. 
\end{proposition}

\begin{proof}
    Suppose that $(b_{m_1},x_{n_1},b_{m_2},x_{n_2})$ is a quadruple satisfying (\ref{quadruple condition}). 
    By Proposition \ref{geodesic tube theorem for close to MA}, 
    there exists $C_1,D_1,T>0$ so that for any $0<\delta <D_1$
    \begin{align*}
        n(-x_{n_1})b_{m_1}^{-1}gb_{m_2}n(x_{n_2})\cdot l_{[-1,1]} \subset \sT_{C_1\delta}^T(l,o).
    \end{align*}
     Hence,
     \begin{align}\label{5 2 draft}
         gb_{m_2}n(x_{n_2})\cdot l_{[-1,1]} \subset b_{m_1}n(x_{n_1})\cdot\sT_{C_1\delta}^T(l,o).
     \end{align}
     Since $ gb_{m_2}n(x_{n_2})\cdot l_{[-1,1]} $ is a geodesic segment of length $2$ in $g\BH^2$,
     by Proposition \ref{not so many tube exists in intersection}, there exist $C>0$, $0<D<D_1$ and $D'>0$ so that
    the number of possible pairs of  $(b_{m_1},x_{n_1})$ satisfying \eqref{5 2 draft} for some $(b_{m_2},x_{n_2})$ with $0<\delta,\omega <D$, $\delta<D'\omega$ and $d(g,H')\geq C\omega$ is $\ll (1+\omega^{-1}\delta N_1) (1+\omega^{-1}\delta N_2)$.
    Writing (\ref{5 2 draft}) as
     \begin{align}\label{5 3 draft}
          b_{m_2}n(x_{n_2})\cdot l_{[-1,1]} \subset g^{-1}b_{m_1}n(x_{n_1})\cdot\sT_{C_1\delta}^T(l,o),
     \end{align}
     by Proposition \ref{a tube not contain many geodesic beam},
     for each fixed $(b_{m_1},x_{n_1})$, the number of $(b_{m_2},x_{n_2})$ satisfying (\ref{5 3 draft}) is $\ll (1+\delta N_1)(1+\delta N_2)$.
     Taking the product of the numbers gives the result.
\end{proof}

\section{Estimates of Hecke returns}\label{section of Hecke returns}

In this section, we give two Hecke return estimates, which will be used in the amplification arguments.

\subsection{Estimates with respect to maximal tori}

We let $g\in\Omega_{v_0}$,  $\fn\subset\cO$ and $\delta>0$, and suppose that $\fn$ is only divisible by prime ideals in $\sQ$. 
We recall that $\sQ$ is the set of finite places of $F$ that is not in $S$ and is inert in $E$.
We define the counting function 
\begin{align*}
    &N(g,\delta,\fn) = \left\{ \gamma\in\bG(F)\cap K(\fn) \,|\, d(g^{-1}\gamma g,e)\leq 2, d(g^{-1}\gamma g,MA)\leq \delta\right\},
\end{align*}
which describes how many times the Hecke operators map $gl$ close to itself that are orientation preserving.

\begin{lemma}\label{fixing a vector for max tori}
    There exists a constant $C>0$ with the following property. Let $g\in\Omega_{v_0}$ and $\fn\subset\cO$ be an ideal divisible only by primes in $\sQ$. If $X=\mathrm{N}(\fn)$ and $\delta>0$ satisfies 
    \begin{align}\label{eq: ineq for eps X C}
        \delta X\leq C,
    \end{align}
    then there exists a quadratic field extension $L$ of $E$ in $B$ so that $N(g,\delta,\fn)$ is contained in $L$.
\end{lemma}
\begin{proof}
    We first prove that the elements in $N(g,\delta,\fn)$ are commutative with each other. For the commutativity, it suffices to check that $xy-yx=0$ for any $x,y\in N(g,\delta,\fn)$. We fix a basis $\{\alpha_1,\alpha_2,\alpha_3,\alpha_4\}$ for $B$ over $E$. By possibly enlarging $S$, we can assume that $\{\alpha_1,\alpha_2,\alpha_3,\alpha_4\}$ forms an $\cO_{E,v}$ basis for $\cO_{B,v}$ whenever $v\notin S$. Every element $x\in B$ can be written as $x=\sum_i x_i\alpha_i$ and we shall write $x$ as a vector $x=(x_i)_{i=1}^4$. If we suppose that $N(g,\delta,\fn)$ is not commutative, then we can let $x=(x_{i}),y=(y_{i})$ be one noncommutative pair and let $z=(z_{i})=xy-yx$. It is clear that $|x_{i}|_{w},|y_{i}|_{w}\ll 1$ whenever $w$ is archimedean, and so $|z_i|_w\ll 1$. At the place $w_0$, there are $a,b\in MA$ such that $d(g^{-1}x g,a)\leq\delta$ and $d(g^{-1}yg,b)\leq\delta$. Since the absolute value at the complex archimedean place $w_0$ is the square of the usual distance, $|z_{i}|_{w_0}\ll\delta^2$. Since we have assumed $z\neq 0$, there is at least one $z_{i}\neq 0$. Since $x,y\in K(\fn)$, it may be seen that $\prod_{w<\infty}|z_{i}|_w\ll X^2$. By the product formula, $1=|z_{i}|_E\ll \delta^2 X^2\leq C^2$. Hence, we get a contradiction if a condition of the form \eqref{eq: ineq for eps X C} holds for $C$ sufficiently small.
    Since $N(g,\delta,\fn)$ is commutative,  there exists a quadratic field over $E$ in $B$ containing $N(g,\delta,\fn)$, which completes the proof.
\end{proof}

\begin{proposition}\label{Hecke return wrt MA}
    There exists a constant $C>0$ with the following property. Let $g\in\Omega_{v_0}$ and $\fn\subset\cO$ be an ideal divisible only by primes in $\sQ$. If $X=\mathrm{N}(\fn)$ and $\delta>0$ satisfies $\delta\leq CX^{-1}$, then $\#N(g,\delta,\fn) \ll_\epsilon X^\epsilon $.
\end{proposition}

\begin{proof}
    We pick $C>0$ as in Lemma \ref{fixing a vector for max tori}, and let $L\subset B$ be the quadratic extension of $E$ produced by the lemma. 
    Since $\sQ$ consists of finite places for $F$ that are inert in $E$, we will denote by $w$ the unique place of $E$ above $v\in\sQ$.
    Let $S^\prime$ be a set of finite places of $E$ generating the ideal class group of $E$ and containing all places lying over $S$ and we assume that $S'$ is disjoint from the places over $\sQ$.
    Let $n\in E$ satisfy $|n|_w = \operatorname{N}_E(\fn_w^{-1})=\operatorname{N}(\fn_v^{-2})$ for all $v\in\sQ$, and $|n|_w\asymp1$  if $w\in S^\prime$ and $w\neq w_0$. Moreover, for all other $E$-places $w\neq w_0$, we require $|n|_w=1$. It implies that $|n|_{w_0}\asymp \operatorname{N}_E(\fn)=X^2$. The reduced norm $ \operatorname{nr}_B $ restricted to $L$ agrees with the norm map from $L$ to $E$. We see that $n\cdot N(g,\delta,\fn) = \{n\eta\,|\,\eta \in N(g,\delta,\fn) \} $ is contained in the set of $\gamma$ in $L$ so that its norm in $E$ is $n^2$; $\gamma$ is integral away from $S^\prime$; All the norms of $\gamma$ in the archimedean places over $w_0$ must be $\ll |n|_{w_0}\asymp X^2$, and all the other norms must be $\ll 1$. 
    The number of these algebraic numbers must be $\ll_\epsilon X^\epsilon$. 
\end{proof}

\subsection{Estimates with respect to \texorpdfstring{$H^\prime$}{H'}}

We define another counting function that will be used in another arithmetic amplification.
We let $g\in\Omega_{v_0}$,  $\fn\subset\cO$ and $\delta>0$, and suppose that $\fn$ is only divisible by prime ideals in $\sP$. 
We recall that $\sP$ is the set of finite places of $F$ that is not in $S$ and is split in $E$.
\begin{align*}
    M(g,\delta,\fn) = \left\{ \gamma\in\bG(F)\cap K(\fn) \,|\, d(g^{-1}\gamma g,e)\leq 1, d(g^{-1}\gamma g,H^\prime)\leq \delta\right\},
\end{align*}
which describes how many times the Hecke operators map $g\BH^2$ close to itself. We have the following estimate for $M(g,\delta,\fn)$, which is proved in \cite[\S4]{hou2023bounds}.

\begin{proposition}\label{Hecke return wrt SL(2,R)}
    There exists a constant $C>0$ with the following property. Let $g\in\Omega_{v_0}$ and $X>0$. There exist a set $\sP_0\subset\sP$, depending on $g$ and with $\#(\sP\backslash\sP_0)\ll\log X$, and with the following property. If $\delta<CX^{-8}$, and the ideal $\fn$ is divisible only by primes in $\sP_0$, and satisfies $\operatorname{N}(\fn)<X$, $\fn\neq\cO$, then $M(g,\delta,\fn) = \emptyset$.
\end{proposition}

\section{Bounds near the spectrum}\label{section of near spectrum}

In this section, we prove Proposition \ref{bound Pi}.

\subsection{Decomposition into geodesic beams}\label{section geodesic beam}

Let $\eta :\BR\to\BR$ be an even smooth function with compact support satisfying:
\begin{itemize}
    \item The support of $\eta$ is $[-2/3,2/3]$,
    \item $\eta(x) \equiv 1$ for $x\in [-1/3,1/3]$,
    \item $0<\eta(x)<1$ if $1/3<|x|<2/3$,
    \item $\sum_{n\in\BZ} \eta (x+n) \equiv 1$.
\end{itemize}
Using $\eta$, we may construct a partition of unity of the boundary $B$. Recall that we identify $\BR/2\pi\BZ$ with $B$ via the map $\theta\mapsto b_\theta$.
Let $N_1, N_2>0$ be two large integers.
For $\epsilon_1>0$ that can be arbitrarily small, 
we suppose that $N_1 = \lambda^{1/2-\epsilon_1}\beta^{-1/2}+ O(1)$ and $N_2 = \lambda^{1/2}\beta^{-1/2}+ O(1)$. 
We define $N_1$ many intervals $\bI_m$ covering $B$ to be as in \eqref{eq: defn of Im}.
We denote by $b_m = b_{\theta_m}$ the center point of $\bI_m$, and let $\tau_0(x) = \eta(N_1x/2\pi)$. We may think of it as a smooth bump function defined on $\BR/2\pi\BZ$ by choosing the fundamental domain $[-\pi,\pi]$ of the quotient space.
We will still denote by $\tau_0$ the function on $B$ by $\tau_0(b_\theta) = \tau_0(\theta)$.
For any other $0\leq m\leq N_1-1$, we define
\begin{align*}
    \tau_m(b) = \tau_0(b_m^{-1}b),\quad b\in B.
\end{align*}
By our construction, $\tau_m$ is a bump function with support $\bI_m$ and $\sum_{m=0}^{N_1-1}\tau_m$ is the constant function $1$ on the boundary $B$.

Next, we construct a partition using bump functions for $\BH^2$. Recall that  $\chi\in C_c^\infty(\BH^2)$ is the given cutoff function in Proposition \ref{bound 1-Pi}. Without loss of generality, we assume the support of $\chi$ is strictly contained in the square
\begin{align*}
    U = \{\bz = (x,e^t)\in\BH^2 : -1\leq x, t\leq 1  \}.
\end{align*}
We use $\bz$ to denote a point in $\BH^2$, and use $(x,e^t)$ with $x,t\in\BR$ for the coordinates of the point $\bz$.
We define $N_2+1$ many intervals $\bJ_n$ in $\BR$ that cover $[-2,2]$ to be as in \eqref{eq: defn of Jn}.
Let $x_n = 4(n-N_2/2)/N_2$ be the  center point of $\bJ_n$. We define
\begin{align*}
    \eta_n(x,e^t) = \eta_n(x)= \eta(N_2(x-x_n)/2),
\end{align*}
that is a bump function in the $x$ direction with support $\bJ_n$. By this construction, we have
\begin{align*}
    \chi(\bz) = \chi(\bz) \sum_{n=0}^{N_2} \eta_n(b^{-1}\bz)
\end{align*}
for any boundary element $b\in B$.

We recall that ${H}_\beta$ is the subspace of $L^2(\BH^2)$ so that its image under the Helgason transform is $L^2([\lambda-\beta,\lambda+\beta]\times B)$. Let $\phi \in \cS(\BH^2)\cap H_\beta$.
We define, for any $0\leq m \leq N_1-1$,
\begin{align*}
    \tilde{\phi}\tau_m(s,b):= \tilde{\phi}(s,b)\tau_m(b)
\end{align*}
if $s>0$ and $b\in B$.
Hence, $\tilde{\phi} \tau_m \in \cS(\BR_{>0}\times B)$, and by Proposition \ref{Prop Weyl invar extension} there exists an extension, which we still denote by $\tilde{\phi} \tau_m $, in $\cS(\BR\times B)_W$.
Hence, $\cF^{-1}(\tilde{\phi} \tau_m ) \in \cS(\BH^2)$.
We have
\begin{align*}
    \chi\phi = \chi\cF^{-1}\left(\tilde{\phi} \sum_{m=0}^{N_1-1}\tau_m \right) = \sum_{m=0}^{N_1-1}\chi\cF^{-1}(\tilde{\phi}\tau_m) =  \sum_{m=0}^{N_1-1}\sum_{n=0}^{N_2}  \chi\eta_n(b_m^{-1}\cdot)\cF^{-1}(\tilde{\phi}\tau_m).
\end{align*}
We let
\begin{align}\label{eq: defn of geodesic beam}
    \phi_{m,n}(\bz) = \chi(\bz) \eta_n(b_m^{-1}\cdot\bz)\cF^{-1}(\tilde{\phi}\tau_m)(\bz) \in \cS(\BH^2),
\end{align}
and so we have the decomposition $\chi\phi = \sum_{m,n}\phi_{m,n}$. 
We call each $\phi_{m,n}$ a geodesic beam along the geodesic $b_m n(x_n)\cdot l$. 

Fix a single beam $ \phi_{m,n}(\bz) = \chi(\bz) \eta_n(b_m^{-1}\bz)\cF^{-1}(\tilde{\phi}\tau_m)(\bz)$.
Let $2\bI_m ,4\bI_m \subset B$ be the intervals centered at $b_m$ with length  $2|\bI_m|$ and $4|\bI_m|$. 
Let $\sB \in C^\infty(B)$  satisfy
\begin{itemize}
    \item $0\leq\sB\leq 1$;
    \item $\sB\equiv 1$ on $2\bI_m$;
    \item $\supp(\sB) \subset 4\bI_m$ .
\end{itemize}
For $s>0$ and $b\in B$, we define
\begin{align}\label{fg decomposition for phi mn}
\begin{split}
    &{f_{m,n}}(s,b)=\tilde{\phi_{m,n}}(s,b)\sB(b),\\
   &{ g_{m,n}}(s,b)=\tilde{\phi_{m,n}}(s,b)(1-\sB(b)).
\end{split}
\end{align}
It is clear that $f_{m,n},g_{m,n}\in L^2(\BR_{>0}\times B)$ and $\tilde{\phi_{m,n}} = f_{m,n} + g_{m,n}$ as functions in $L^2(\BR_{>0}\times B)$.
We shall prove that $g_{m,n}$ is an error term.

\begin{lemma}\label{lemma for gradient of A(bna)}
    Let $0\neq \theta\in\BR/2\pi\BZ$, $x,t\in\BR$. 
    We have
    \begin{align*}
        \frac{\partial}{\partial t}A(b_\theta n(x) a(t)) =-\frac{e^{2t} -  (x-\cot(\theta/2))^2}{e^{2t}+ (x-\cot(\theta/2))^2},
    \end{align*}
    and
    \begin{align*}
        \frac{\partial}{\partial x}A(b_\theta n(x) a(t)) =-\frac{2 (x-\cot(\theta/2))}{e^{2t}+ (x-\cot(\theta/2))^2}.
    \end{align*}
\end{lemma}
\begin{proof}
    The results follow directly from \eqref{Iwasawa A for bna}.
\end{proof}

\begin{proposition}\label{Fourier support lemma for a single geodesic beam}
    Suppose that  $\phi\in\cS(\BH^2)\cap H_\beta$ and let $\phi_{m,n}$ and $g_{m,n}$ be as in \eqref{eq: defn of geodesic beam} and \eqref{fg decomposition for phi mn}.
     Then we have 
     \begin{align*}
         \| g_{m,n} \|_{L^2(\BR_{>0}\times B)} \ll_{A} \lambda^{-A} \|  \phi\|_{L^2(\BH^2)}.
     \end{align*}
\end{proposition}

\begin{proof}
     Let $s\in\BR_{>0}$ and $b\in B$. Then $\tilde{\phi_{m,n}}(s,b)$ is equal to
    \begin{align*}
          &  \int_{\BH^2} \chi(\bz)\eta_n(b_m^{-1}\bz) \cF^{-1}(\tilde{\phi}\tau_m) (\bz) e^{(-is+1/2) A(\bz,b)} d\bz\notag\\
         =&\int_{\BH^2} \chi(b_m \bz)\eta_n(\bz) \cF^{-1}(\tilde{\phi}\tau_m) (b_m\bz) e^{(-is+1/2) A(b_m\bz,b)} d\bz\\
         =&\int_{\BH^2}\chi(b_m \bz)\eta_n(\bz) \left( \int_{\BR_{>0}\times B}\tilde{\phi}(r,d) \tau_m(d) e^{(ir+1/2)A(b_m\bz,d)} d\mu(r,d)\right) e^{(-is+1/2) A(b_m\bz,b)} d\bz\notag\\
         =& \int_{0}^{\infty}\int_{B}  \left(   \int_{\BH^2}\chi(b_m\bz)\eta_n(\bz)e^{(A(b_m\bz,d)+A(b_m\bz,b))/2} \exp(irA(d^{-1}b_m\bz)-isA(b^{-1}b_m\bz)) d\bz\right) \tilde{\phi}(r,d)\tau_m(d)d\mu(r,d).
    \end{align*}
    We consider the integral inside the parenthesis, 
    which can be written as an oscillatory integral.
    We denote it by $\cI(s,r,b,d)$:
    \begin{align}
        \cI(s,r,b,d) = &\iint_\BR \eta_n(x) \chi_1(x,t) \exp\left( irA(d^{-1}b_mn(x)a(t))-isA(b^{-1}b_mn(x)a(t)) \right) dxdt\notag\\
        =&\int_\BR \left(\int_\BR \chi_1  \exp\left( irA(d^{-1}b_mn(x)a(t))-isA(b^{-1}b_mn(x)a(t)) \right) dt\right)\eta_n(x) dx. \label{I(s,r,b,d)wrt t}
    \end{align}
    Here $r\in [\lambda-\beta,\lambda+\beta]$, $b\notin 2\bI_m$, $d\in \bI_m$ and
    \begin{align*}
        \chi_1(x,t) = \chi(b_m\bz)\exp\left({(A(b_m\bz,d)+A(b_m\bz,b))/2-t} \right),\quad\text{ where }\bz = n(x)a(t)\cdot o,
    \end{align*}
    is a smooth compactly supported function.  We have $d^{-1}b_m \in \bI_0$ and $b^{-1}b_m\notin 2\bI_0$.
    We can also write $\cI(s,r,b,d)$ as an oscillatory integral in the $x$-variable:
    \begin{align}
        \cI(s,r,b,d) &=\int_\BR \left(\int_\BR \eta(N_2(x-x_n)/2) \chi_1  \exp\left( irA(d^{-1}b_mn(x)a(t))-isA(b^{-1}b_mn(x)a(t)) \right) dx \right)dt\notag\\
        &=\frac{2}{N_2}\int_\BR \left(\int_\BR \chi_2  \exp\left( irA(d^{-1}b_mn(x_n + 2x/N_2)a(t))-isA(b^{-1}b_mn(x_n + 2x/N_2)a(t)) \right)  dx \right)dt.  \label{I(s,r,b,d)wrt x}
    \end{align}
    Here we let $\chi_2(x,t) = \eta(x) \chi_1(x_n + 2x/N_2,t)$.
    We notice that $\chi_2$ is a cutoff function with support at scale $\asymp 1\times 1$.
    We define $\theta_1 \in \bI_0,\theta_2 \notin 2\bI_0$ so that they satisfy
    \begin{align*}
        d^{-1}b_m = b_{\theta_1}\quad\text{ and }\quad b^{-1}b_m = b_{\theta_2}.
    \end{align*}
    By Corollary \ref{uniformization lemma for B}, there exists $\delta,\sigma>0$ so that for any $x,t\in [-1,1]$,
    \begin{align*}
        1-\frac{\partial}{\partial t} A(b_\theta n(x)a(t)) &= \theta^2\xi(\theta,x,t),\\
        \xi(\theta,x,t)&\geq \sigma,\\
        \left|\frac{\partial^n\xi}{\partial t^n} (\theta,x,t)\right|
        &\ll_{n}1,
    \end{align*}
    if  $\theta\in (-\delta,\delta) \mod 2\pi$, and \begin{align*}
        1-\frac{\partial}{\partial t} A(b_\theta n(x) a(t)) \geq \sigma,
    \end{align*}
    if $\theta\notin (-\delta,\delta) \mod 2\pi$.
    Here $\theta$ can be either $\theta_1$ or $\theta_2$, and we can assume $0<\sigma<1$.
    
    If $0 \leq s/r \leq (2-\sigma)/{2(1-\sigma)}$ and $\theta_2\notin (-\delta,\delta)$, 
    we consider the oscillatory integral in (\ref{I(s,r,b,d)wrt t})
    \begin{align}\label{oscillatory integral for small s/r}
        \int_\BR \chi_1 \exp(ir\varphi(x,t;s/r,\theta_1,\theta_1)) dt,
    \end{align}
    with
    \begin{align*}
        \varphi(x,t;s/r,\theta_1,\theta_1) &= A(b_{\theta_1}n(x)a(t)) - \frac{s}{r} A(b_{\theta_2}n(x)a(t))\\
        &=-\left( t - A(b_{\theta_1}n(x)a(t))\right) + \frac{s}{r}\left( t-A(b_{\theta_2}n(x)a(t)) \right) + \frac{r-s}{r}t.
    \end{align*}
    Hence,
    \begin{align*}
        \frac{\partial \varphi}{\partial t}&= -\theta_1^2 \xi(\theta_1,x,t) +\frac{s}{r} \left(1-\frac{\partial}{\partial t} A(b_{\theta_2} n(x) a(t))  \right) + \frac{r-s}{r}\\
        &\geq -\theta_1^2 \xi + \frac{s}{r} \sigma + \frac{r-s}{r}\\
        &\geq  -\theta_1^2 \xi + \sigma/2 \\
        &\gg 1.
    \end{align*}
    Also, it is clear that all higher derivatives $\partial^n \varphi / \partial t^n \ll_n 1$.
    Hence, in this case, the bound $\int_\BR \chi_1 e^{ir\varphi} dt \ll_A r^{-A}\ll \lambda^{-A}$ follows from integrating by parts. 

     If $0 \leq s/r \leq 2/3$ and  $\theta_2\in (-\delta,\delta)$,
     we still consider the oscillatory integral in \eqref{I(s,r,b,d)wrt t}.
     Again, we consider the phase function $\varphi$ defined in \eqref{oscillatory integral for small s/r}.
     We have
     \begin{align*}
         \varphi(x,t;s/r,\theta_1,\theta_1)         &=-\left( t - A(b_{\theta_1}n(x)a(t))\right) + \frac{s}{r}\left( t-A(b_{\theta_2}n(x)a(t)) \right) + \frac{r-s}{r}t.
     \end{align*}
     Since $(s/r) \theta_2^2 \xi(\theta_2,x,t)$ is non-negative and $\frac{r-s}{r} \geq \frac{1}{3}$,
     \begin{align*}
        \frac{\partial \varphi}{\partial t}&= -\theta_1^2 \xi(\theta_1,x,t) +\frac{s}{r} \theta_2^2 \xi(\theta_2,x,t)+ \frac{r-s}{r}\geq  -\theta_1^2 \xi(\theta_1,x,t) + \frac{1}{3} \gg 1+ O(\theta_1^2)\gg1.
     \end{align*}
     Since $\partial^n \varphi / \partial t^n \ll_n 1$,
     the bound $\int_\BR \chi_1 e^{ir\varphi} dt \ll_A \lambda^{-A}$ follows by integrating by parts. 
    
    If $2/3 \leq s/r \leq (2-\sigma)/{2(1-\sigma)}$ and  $\theta_2\in (-\delta,\delta)$,
    we consider the oscillatory integral in (\ref{I(s,r,b,d)wrt x})
    \begin{align*}
        \int_\BR \chi_2  \exp\left(- i(\beta N_2\theta_2) \varphi(x,t;s/r,\theta_1,\theta_2) \right)  dx 
    \end{align*}
    with
    \begin{align*}
        \varphi(x,t;s/r,\theta_1,\theta_2) = (r\beta^{-1}N_2^{-1}\theta_2^{-1}) \left( -A(b_{\theta_1}n(x_n + 2x/N_2)a(t))+\frac{s}{r}A(b_{\theta_2}n(x_n + 2x/N_2)a(t)) \right).
    \end{align*}
    We have $\beta N_2\theta_2 \gg \beta (\lambda^{1/2} \beta^{-1/2}) (\lambda^{-1/2 + \epsilon_1}\beta^{1/2}) \asymp \lambda^{ \epsilon_1}\beta$. By  \eqref{partial derivative wrt x for A(bna)} in Lemma \ref{lemma for A(bna) when b is small},
    \begin{align*}
        \frac{\partial \varphi}{\partial x}\asymp (r\beta^{-1}N_2^{-1}\theta_2^{-1} )N_2^{-1}(\frac{s}{r}\theta_2-\theta_1)\asymp (r\beta^{-1}N_2^{-1}\theta_2^{-1} )N_2^{-1} \theta_2 \asymp r\beta^{-1}N_2^{-2}\asymp 1.
    \end{align*}
    Here $(s/r)\theta_2-\theta_1\asymp\theta_2$ holds because $\theta_1 \in \bI_0,\theta_2 \notin 2\bI_0$ and we have assumed $s/r\geq2/3$.
    For any $n\geq 2$, we have
    \begin{align*}
        \frac{\partial^n\varphi}{ \partial x^n} \ll (r\beta^{-1}N_2^{-1}\theta_2^{-1} ) N_2^{-n} \ll \lambda \beta^{-1} N_2^{-3} |\bI_0|^{-1} \ll \lambda \beta^{-1} (\lambda^{1/2}\beta^{-1/2})^{-3}(\lambda^{1/2-\epsilon_1}\beta^{-1/2})\ll 1.
    \end{align*}
    Hence, by integrating by parts, we have
    \begin{align*}
        \int_\BR \chi_2  \exp\left(- i(\beta N_2\theta_2) \varphi(x,t;s/r,\theta_1,\theta_2) \right)  dx \ll_A (\lambda^{\epsilon_1}\beta)^{-A} \ll_A \lambda^{-A}.
    \end{align*}
    
    When $s/r \geq (2-\sigma)/{2(1-\sigma)}$, we will show the two dimensional integral $\cI(s,r,b,d)$ is $\ll_A s^{-A}$.
    We have
    \begin{align*}
        \cI(s,r,b,d) &=\frac{2}{N_2}\iint_\BR \chi_2  \exp\left( -is \varphi(x,t) \right) dx dt,
    \end{align*}
    with the phase function
    \begin{align*}
         \varphi(x,t;s/r,\theta_1,\theta_2) = -\frac{r}{s}A(b_{\theta_1}n(x_n + 2x/N_2)a(t))+A(b_{\theta_2}n(x_n + 2x/N_2)a(t)).
    \end{align*}
   We define 
    \begin{align*}
        X_2 = X_2(x,\theta_2) := x_n + \frac{2x}{N_2} - \cot(\theta_2/2) = x_n - \cot(\theta_2/2) + O(\lambda^{-1/2}\beta^{1/2}),
    \end{align*}
    and it is well-defined because $\theta_2$ is defined to be away from $0$ in $\BR/2\pi\BZ$.
    Lemma \ref{lemma for gradient of A(bna)} implies that
    \begin{align*}
        \frac{\partial}{\partial t}A(b_{\theta_2} n(x_n+2x/N_2) a(t)) =-\frac{e^{2t} -  X_2^2}{e^{2t}+ X_2^2},
    \end{align*}
    and
    \begin{align*}
        \frac{\partial}{\partial x}A(b_{\theta_2} n(x_n+2x/N_2) a(t))  =-\frac{2X_2}{e^{2t}+ X_2^2}\cdot \frac{2}{N_2}.
    \end{align*}
    Since $\theta_1 \in \bI_0$, by  \eqref{partial derivative wrt x for A(bna)},
    we have
     \begin{align*}
        \frac{\partial}{\partial x}A(b_{\theta_1} n(x_n+2x/N_2) a(t))  \asymp N_2^{-1}\theta_1 \ll \lambda^{-1 + \epsilon_1} \beta.
    \end{align*}
    Let $0<\sigma_1<1$ be a small number to be chosen later. 
    We may divide the set $(\BR / 2\pi\BZ) \backslash (2\bI_0)$ into three parts:
    \begin{align*}
        \bI_{\rom{1}} &:= \{\theta\in (\BR / 2\pi\BZ) \backslash (2\bI_0) \, | \,\sigma_1\leq |x_n - \cot(\theta/2) | \leq \sigma_1^{-1}  \},\\
        \bI_{\rom{2}}&:= \{\theta\in (\BR / 2\pi\BZ) \backslash (2\bI_0) \, | \, |x_n - \cot(\theta/2) | < \sigma_1  \},\\
        \bI_{\rom{3}} &:= \{\theta\in (\BR / 2\pi\BZ) \backslash (2\bI_0) \, | \, |x_n - \cot(\theta/2) | > \sigma_1^{-1}  \}.
    \end{align*}
    
    If $\theta_2 \in \bI_{\rom{1}}$,
    then there exists $\sigma_2 > 0$, depending on $\sigma_1$, so that
    \begin{align*}
       \left| N_2\cdot \frac{\partial}{\partial x}A(b_{\theta_2} n(x_n+2x/N_2) a(t))\right| = \frac{4|X_2|}{e^{2t}+ X_2^2}  \geq \sigma_2
    \end{align*}
    for all $(x,t)\in \supp \chi_2$,
    then
    \begin{align*}
        \frac{\partial}{\partial x}N_2 \varphi(x,t;s/r,\theta_1,\theta_2) \geq \sigma_2 + O(\lambda^{-1 /2+ \epsilon_1} \beta^{1/2}) \gg  \sigma_2 >0.
    \end{align*}
    For any $n\geq 1$, we have
    \begin{align*}
         \frac{\partial^n}{\partial x^n}N_2 \varphi(x,t;s/r,\theta_1,\theta_2) \ll N_2 \cdot N_2^{-n} \ll 1.
    \end{align*}
    Hence, 
    \begin{align*}
        \int \chi_2 \exp(-is\varphi) dx = \int \chi_2 \exp(-i(sN_2^{-1})(N_2\varphi)) dx \ll_A (sN_2^{-1})^{-A}\ll_A s^{-A}.
    \end{align*}
   
    If $\theta_2 \in \bI_{\rom{2}}$, then $|X_2| =  |x_n - \cot(\theta_2/2)| + O(\lambda^{-1/2}\beta^{1/2}) < \sigma_1 + O(\lambda^{-1/2}\beta^{1/2}) \leq 2 \sigma_1$ when $\lambda$ is sufficiently large.
    Hence,
    \begin{align*}
        \frac{\partial}{\partial t}A(b_{\theta_2} n(x_n+2x/N_2) a(t)) =-1 + \frac{2X_2^2}{e^{2t}+X_2^2} = -1 + O(\sigma_1^2).
    \end{align*}
    We should choose $\sigma_1$ to be sufficiently small so that for $\theta_2 \in \bI_{\rom{2}} $
    \begin{align*}
         -\frac{\partial\varphi}{\partial t}(x,t;s/r,\theta_1,\theta_2) = 1 + O(\sigma_1^2)+ \frac{r}{s}(1+O(\theta_1^2))\gg 1.
    \end{align*}
     In this case, it is clear that $\partial^n \varphi / \partial t^n \ll_n 1$.
    Hence, $\int_\BR \chi_2 \exp(-is\varphi) dt\ll_A s^{-A}$.
    
    If $\theta_2 \in \bI_{\rom{3}}$, then $|X_2| =  |x_n - \cot(\theta_2/2)| + O(\lambda^{-1/2}\beta^{1/2}) > \sigma_1^{-1} + O(\lambda^{-1/2}\beta^{1/2}) \geq   \sigma_1^{-1}/2$. Hence
    \begin{align*}
        \frac{\partial}{\partial t}A(b_{\theta_2} n(x_n+2x/N_2) a(t)) =1 - \frac{2e^{2t}}{e^{2t}+X_2^2} = 1 + O(\sigma_1^2).
    \end{align*}
    Again, we can make sure $\sigma_1$ is sufficiently small so that 
     \begin{align*}
         \frac{\partial\varphi}{\partial t}(x,t;s/r,\theta_1,\theta_2) = 1 + O(\sigma_1^2)-\frac{r}{s}(1+O(\theta_1^2))\geq 1-\frac{r}{s} + O(\sigma_1^2)\geq \frac{\sigma}{2-\sigma} + O(\sigma_1^2)\gg 1.
    \end{align*}
    Therefore $\int_\BR \chi_2 \exp(-is\varphi) dt\ll_A s^{-A}$.

    The $L^2$-norm bound for $g_{m,n}$ follows from all the above estimates and the Cauchy-Schwarz inequality.
\end{proof}

\begin{proposition}\label{l2L2 inequality general form}
    Suppose that  $\phi\in\cS(\BH^2)\cap H_\beta$ and let $\phi_{m,n}$ be as in \eqref{eq: defn of geodesic beam}.
    Let $\cS \subset \{ (m,n)\,|\, 0\leq m\leq N_1 -1,0\leq n\leq N_2   \}$ be any subset.
    There is a constant $C>0$ independent of $\phi$ and $\cS$ so that
    \begin{align*}
       \Big| \sum_{\substack{(m_1,n_1),(m_2,n_2)\in\cS\\(m_1,n_1) \neq (m_2,n_2)}} \langle \phi_{m_1,n_1},\phi_{m_2,n_2}\rangle \Big|\leq C\|\phi  \|^2_{L^2(\BH^2)}.
    \end{align*}
\end{proposition}
\begin{proof}
    We may split the sum into two parts
    \begin{align}\label{4-8}
        \sum_m  \sum_{\substack{n_1\neq n_2\\(m,n_1),(m,n_2)\in \cS}} \langle \phi_{m,n_1},\phi_{m,n_2}\rangle + \sum_{m_1\neq m_2}\sum_{\substack{n_1, n_2\\(m_1,n_1),(m_2,n_2)\in \cS}}\left\langle \phi_{m_1,n_1} , \phi_{m_2,n_2} \right\rangle. 
    \end{align}
    Since $\langle \phi_{m,n_1},\phi_{m,n_2}\rangle =0$ unless $\bJ_{n_1}\cap \bJ_{n_2}\neq \emptyset$, the first term in \eqref{4-8} is
    \begin{align*}
        \sum_m  \sum_{\substack{n_1\neq n_2\\(m,n_2)\in \cS}} \langle \phi_{m,n_1},\phi_{m,n_2}\rangle &\ll\sum_m \sum_{n_1 \neq n_2}\left|\langle \phi_{m,n_1},\phi_{m,n_2}\rangle\right| \\
        &\ll \sum_m\left(  \sum_{n_1\neq n_2} \int_{\bJ_{n_1}\cap \bJ_{n_2}} \int_{-1}^1   \left|\chi(x,e^t) \cF^{-1}(\tilde{\phi}\tau_m)\left(b_m\cdot(x,e^t)\right)  \right|^2     e^{-t}dxdt\right)\\
          &\ll \sum_m \| \chi\cF^{-1}(\tilde{\phi}\tau_m) \|_{L^2(\BH^2)}^2 \\
          &\ll \sum_m\|   \cF^{-1}(\tilde{\phi}\tau_m) \|_{L^2(\BH^2)}^2\\
          &= \sum_m\|   \tilde{\phi}\tau_m \|_{L^2(\BR_{>0}\times B)}^2\\
          &= \int_{\BR_{>0}\times B} \sum_{m}\tau_m(b)^2|\tilde{\phi}(s,b)|^2d\mu(s,b)\\
          &\ll \| \tilde{\phi} \|^2_{L^2(\BR_{>0}\times B)}.
    \end{align*}
    Apply the decomposition (\ref{fg decomposition for phi mn}) and Proposition    \ref{Fourier support lemma for a single geodesic beam} to $\phi_{m,n}$, 
    we have $\tilde{\phi_{m,n}} = f_{m,n}+ g_{m,n}$ with the properties 
    $\supp({f_{m,n}})\subset \BR_{>0}\times 4\bI_m$ and $\| g_{m,n} \|_{L^2(\BR_{>0}\times B)}\ll_N \lambda^{-N}\|\phi\|_{L^2(\BH^2)}$. For each fixed $m$, we define
    \begin{align*}
        \phi_{m,\cS} = \sum_{\substack{n\\(m,n)\in\cS}} \phi_{m,n},\quad\quad
        f_{m,\cS} =  \sum_{\substack{n\\(m,n)\in\cS}} f_{m,n},\quad\quad
        g_{m,\cS} =  \sum_{\substack{n\\(m,n)\in\cS}} g_{m,n}.
    \end{align*}
    We claim that
    \begin{align} \label{inequality for phi m delta}
        \sum_m \|\phi_{m,\cS} \|_{L^2(\BH^2)}^2 \ll \| \phi\|_{L^2(\BH^2)}^2.
    \end{align}
    The claim follows from the definitions
    \begin{align*}
    \begin{split}
        \sum_m \|\phi_{m,\cS} \|_{L^2(\BH^2)}^2 &= \sum_m \int_{\BH^2}\left|  \sum_{\substack{n\\(m,n)\in\cS}}\phi_{m,n}(\bz)\right|^2 d\bz\\
        &=\sum_m \int_{\BH^2}\left( \sum_{\substack{n\\(m,n)\in\cS}}\chi(\bz)\eta_n(b_m^{-1}\bz)\right)^2  |\cF^{-1}(\tilde{\phi}\tau_m)(\bz)|^2 d\bz\\
        &\ll \sum_{m}  \int_{\BH^2}\left|\cF^{-1}(\tilde{\phi}\tau_m)(\bz)\right|^2 d\bz\\
        &\ll \sum_m \int_{\BR_{>0} \times B} \left|\tilde{\phi}(s,b)\right|^2 \tau_m(b)^2 d\mu(s,b)\\
        &\ll \|\phi\|_{L^2(\BH^2)}^2.
        \end{split}
    \end{align*}
    It may be seen that $\tilde{\phi_{m,\cS}} = f_{m,\cS}+ g_{m,\cS}$ as functions in $L^2(\BR_{>0}\times B)$, $\supp(f_{m,\cS})\subset \BR_{>0}\times 4\bI_m$ and 
    \begin{align*}
        \| g_{m,\cS} \|_{L^2(\BR_{>0}\times B)} \leq \sum \| g_{m,n} \|_{L^2(\BR_{>0}\times B)} \ll_N \lambda^{-N}\|\phi\|_{L^2(\BH^2)}.
    \end{align*}
    Hence, the second term in (\ref{4-8}) is
    \begin{align*}
         \sum_{m_1\neq m_2}\langle \phi_{m_1,\cS}, \phi_{m_2,\cS}\rangle&= \sum_{m_1\neq m_2}\langle \tilde{\phi_{m_1,\cS}}, \tilde{\phi_{m_2,\cS}}\rangle\\
         &= \sum_{m_1\neq m_2}\left\langle f_{m_1,\cS},f_{m_2,\cS}  \right\rangle +O_N(\lambda^{-N}\|\phi\|_{L^2(\BH^2)}^2),
    \end{align*}
    and $\sum_{m_1\neq m_2}\langle f_{m_1,\cS},f_{m_2,\cS}  \rangle$ is bounded by
    \begin{align*}
        & \sum_{m_1\neq m_2}\left|  \int_{\BR_{>0}\times B}{f_{m_1,\cS}}(s,b)\overline{{f_{m_2,\cS}}(s,b)} d\mu(s,b)\right|\\
        =&  \sum_{m_1\neq m_2}\left|  \int_0^\infty\int_{4\bI_{m_1}\cap 4\bI_{m_2}}{f_{m_1,\cS}}(s,b)\overline{{f_{m_2,\cS}}(s,b)} d\mu(s,b)\right|\\
        \leq &  \sum_{m_1\neq m_2}  \int_{0}^{\infty}\int_{4\bI_{m_1}\cap 4\bI_{m_2}}\left|\tilde{\phi_{m_1,\cS}}(s,b){\tilde{\phi_{m_2,\cS}}(s,b)} \right|d\mu(s,b)\\
        \leq& \frac{1}{2}\sum_{m_1\neq m_2}  \int_{0}^{\infty}\int_{4\bI_{m_1}\cap 4\bI_{m_2}}\left(\left|\tilde{\phi_{m_1,\cS}}(s,b)\right|^2 + \left|{\tilde{\phi_{m_2,\cS}}(s,b)} \right|^2\right)d\mu(s,b)\\
        \ll& \sum_{m_1}  \int_{0}^{\infty}\sum_{m_2\neq m_1}\int_{4\bI_{m_1}\cap 4\bI_{m_2}}\left|\tilde{\phi_{m_1,\cS}}(s,b)\right|^2 d\mu(s,b)\\
        \ll& \sum_{m}  \int_{\BR_{>0}\times B}\left|\tilde{\phi_{m,\cS}}(s,b)\right|^2 d\mu(s,b)\\
        =&\sum_m \|\phi_{m,\cS} \|_{L^2(\BH^2)}^2.
    \end{align*}
    Hence, we obtain the desired bound by applying \eqref{inequality for phi m delta}.
\end{proof}

\begin{proposition}\label{l2L2 inequality}
    Suppose that  $\phi\in\cS(\BH^2)\cap H_\beta$ and let $\phi_{m,n}$ be as in \eqref{eq: defn of geodesic beam} with the geodesic beam decomposition $\chi\phi = \sum_{m,n} \phi_{m,n}$.
    There is a constant $C>0$ independent of $\phi$  so that
    \begin{align*}
         \sum_{m=0}^{N_1-1}\sum_{n=0}^{N_2} \| \phi_{m,n}\|_{L^2(\BH^2)}^2 \leq C   \| \phi\|_{L^2(\BH^2)}^2.
    \end{align*}
\end{proposition}

\begin{proof}
    We have
    \begin{align*}
        \|\chi\phi\|_{L^2(\BH^2)}^2 &= \langle \chi\phi,\chi\phi \rangle \\
        &= \sum_{m_1,n_1}\sum_{m_2,n_2} \langle \phi_{m_1,n_1},\phi_{m_2,n_2}\rangle\\
        &=   \sum_{m,n} \| \phi_{m,n}\|_{L^2(\BH^2)}^2 + \sum_{(m_1,n_1) \neq (m_2,n_2)} \langle \phi_{m_1,n_1},\phi_{m_2,n_2}\rangle.
    \end{align*}
    Hence,
    \begin{align*}
        \sum_{m,n} \| \phi_{m,n}\|_{L^2(\BH^2)}^2 \leq \| \chi\phi\|_{L^2(\BH^2)}^2 + \left|\sum_{(m_1,n_1) \neq (m_2,n_2)} \langle \phi_{m_1,n_1},\phi_{m_2,n_2}\rangle \right|.
    \end{align*}
    The proposition follows by applying Proposition \ref{l2L2 inequality general form}.
\end{proof}

Suppose that  $\phi\in\cS(\BH^2)\cap H_\beta$ and let $\chi\phi = \sum_{m,n}\phi_{m,n}$ be the decomposition given by \eqref{eq: defn of geodesic beam}.
We decompose the sum $\chi\phi = \sum_{m,n}\phi_{m,n}$ further based on $L^2$-norms of $\phi_{m,n}$.
Let $C>0$ be the constant as in Proposition \ref{l2L2 inequality}, so that
\begin{align}\label{l2L2 for our phi}
         \sum_{m=0}^{N_1-1}\sum_{n=0}^{N_2} \| \phi_{m,n}\|_{L^2(\BH^2)}^2 \leq C   \| \phi\|_{L^2(\BH^2)}^2
\end{align}
holds. For any $0<\delta<1$, we define
\begin{align*}
    \cS_{<\delta}(\phi) = \{\phi_{m,n}  : \, \| \phi_{m,n}\|_{L^2(\BH^2)}^2< \delta C \| \phi \|_{L^2(\BH^2)}^2\},
\end{align*}
and
\begin{align*}
    \cS_{\geq\delta}(\phi) = \{\phi_{m,n}  : \, \| \phi_{m,n}\|_{L^2(\BH^2)}^2\geq\delta C \| \phi \|_{L^2(\BH^2)}^2\}.
\end{align*}

\begin{proposition}\label{no many big norm geodesic beams}
    For any $0<\delta<1$, we have $\# \cS_{\geq\delta}(\phi)\leq \delta^{-1}$.
\end{proposition}
\begin{proof}
    By the inequality (\ref{l2L2 for our phi}), we have
    \begin{align*}
        C \| \phi\|_{L^2(\BH^2)}^2 \geq   \sum_{\phi_{m,n}\in \cS_{\geq\delta}(\phi)} \| \phi_{m,n}\|_{L^2(\BH^2)}^2 \geq \#\cS_{\geq\delta}(\phi)  \cdot\delta C\| \phi\|_{L^2(\BH^2)}^2,
    \end{align*}
    so we conclude that $\# \cS_{\geq\delta}(\phi)\leq \delta^{-1}$.
\end{proof}

In the rest of the section, we will prove the following two main estimates for the sums of $\langle\psi,\phi_{m,n} \rangle$. The first estimate is obtained by treating each $\phi_{m,n}$ separately.

\begin{proposition}\label{amplified bound for geodesic beam}
    Suppose that  $\phi\in\cS(\BH^2)\cap H_\beta$ and let $\chi\phi = \sum_{m,n}\phi_{m,n}$ be the decomposition given by \eqref{eq: defn of geodesic beam}.
    If $\lambda^{\epsilon'}\leq \beta\leq \lambda^{1/4}$, we have
    \begin{align}
        \langle\psi,\phi_{m,n} \rangle\ll_{\epsilon',\epsilon_1,\epsilon,A}\lambda^{1/4 - 1/16+ 9\epsilon_1/4 + \epsilon}\beta^{13/16} \|\phi_{m,n} \|_{L^2(\BH^2)} + \lambda^{-A}\|\phi  \|_{L^2(\BH^2)}. \label{eq: amplified geodesic beam, individual est}
    \end{align}
    Moreover, for any $0<\delta<1$, we have
    \begin{align}\label{eq: summing over amplified geodesic beam}
        \sum_{\phi_{m,n} \in \cS_{\geq\delta}(\phi)}\langle\psi,\phi_{m,n} \rangle\ll_{\epsilon',\epsilon_1,\epsilon}\delta^{-1/2}\lambda^{1/4 - 1/16+ 9\epsilon_1/4 + \epsilon}\beta^{13/16} \|\phi\|_{L^2(\BH^2)}.
    \end{align}
\end{proposition}

For the geodesic beams with smaller $L^2$-norms, we take the summation of them to establish the second estimate.

\begin{proposition}\label{amplified bound for small norm}
Suppose that  $\phi\in\cS(\BH^2)\cap H_\beta$ and let $\chi\phi = \sum_{m,n}\phi_{m,n}$ be the decomposition given by \eqref{eq: defn of geodesic beam}.
    If $\lambda^{\epsilon'} \leq \beta\leq \lambda^{1/4}$
    and $\lambda^{-1/2}\beta^{1/2}\leq\delta\leq\beta^{-3/2}$, 
    we have 
    \begin{align*}
        \sum_{\phi_{m,n} \in \cS_{<\delta}(\phi)}\langle\psi,\phi_{m,n} \rangle\ll_{\epsilon',\epsilon_1,\epsilon} \delta^{1/138}\lambda^{1/4+5\epsilon_1+\epsilon}\beta^{1/92} \|\phi\|_{L^2(\BH^2)}.
    \end{align*}
\end{proposition}

\begin{proof}[Proof of Proposition \ref{bound Pi}]
    Given $\phi \in \cS(\BH^2)\cap H_\beta$, by the decomposition from \eqref{eq: defn of geodesic beam}, and applying Propositions \ref{amplified bound for geodesic beam} and \ref{amplified bound for small norm}, we have
    \begin{align*}
        \langle \psi,\chi\phi\rangle &=  \sum_{\phi_{m,n} \in \cS_{\geq\delta}(\phi)}\langle\psi,\phi_{m,n} \rangle + \sum_{\phi_{m,n} \in \cS_{<\delta}(\phi)}\langle\psi,\phi_{m,n} \rangle\\
        &\ll_{\epsilon',\epsilon }\left(\delta^{-1/2}\lambda^{1/4 - 1/16 + \epsilon}\beta^{13/16}  + \delta^{1/138}\lambda^{1/4+\epsilon}\beta^{1/92} \right)\|\phi\|_{L^2(\BH^2)}.
    \end{align*}
    Choosing $\delta = (\lambda^{-69}\beta^{885})^{1/560}>\lambda^{-1/2}\beta^{1/2}$ completes the proof. It may be seen that $\delta\leq\beta^{-3/2}$ if $\beta\leq\lambda^{1/25}$.
\end{proof}

\subsection{Estimates of $\langle\psi,\phi_{m,n}\rangle$}
Suppose that  $\phi\in\cS(\BH^2)\cap H_\beta$ and let $\chi\phi = \sum_{m,n}\phi_{m,n}$ be the decomposition given by \eqref{eq: defn of geodesic beam}.
Given two geodesic beams $\phi_{m_1,n_1}$ and $\phi_{m_2,n_2}$, and $g\in G_0$, we define the pairing
\begin{align}\label{pairing translated by g}
    P(\phi_{m_1,n_1},\phi_{m_2,n_2},g;\lambda) = \iint_{\BH^2} \overline{\phi_{m_1,n_1}(\bz_1)}\phi_{m_2,n_2}(\bz_2) k_\lambda (\bz_1^{-1}g\bz_2) d\bz_1 d\bz_2.
\end{align}
We apply Proposition \ref{nonstationary oscillatory integral estimate} to bound $P(\phi_{m_1,n_1},\phi_{m_2,n_2},g;\lambda)$ in the following case.

\begin{corollary}\label{nonstationary phase bound for P}
Suppose that  $\phi\in\cS(\BH^2)\cap H_\beta$ and let $\chi\phi = \sum_{m,n}\phi_{m,n}$ be the decomposition given by \eqref{eq: defn of geodesic beam}.
    Suppose that $\lambda^{\epsilon'}\leq \beta\leq \lambda^{1/4}$ and let $C_0> 0$ be given. If
    \begin{align}\label{uniform condition}
        d(g,e)\leq C_0\quad \text{ and }\quad d(n(-x_{n_1})b_{m_1}^{-1}gb_{m_2}n(x_{n_2}), MA) \geq \lambda^{-1/2 + 2\epsilon_1}\beta^{1/2},
    \end{align}
     we have
    \begin{align*}
        P(\phi_{m_1,n_1},\phi_{m_2,n_2},g;\lambda) \ll_{\epsilon',\epsilon_1,A} \lambda^{-A}\|\phi\|_{L^2(\BH^2)}^2.
    \end{align*}
\end{corollary}

\begin{proof}
    We unfold the integral $P(\phi_{m_1,n_1},\phi_{m_2,n_2},g;\lambda)$ by applying the definition of $\phi_{m,n}$ and the inverse Harish-Chandra transform for $k_\lambda$ so $P(\phi_{m_1,n_1},\phi_{m_2,n_2},g;\lambda)  $ is equal to
    \begin{align*}
        \int_0^\infty\left(\iint_{\BH^2} \chi(\bz_1) \chi(\bz_2) \eta_{n_1}(b_{m_1}^{-1}\cdot\bz_1)\eta_{n_2}(b_{m_2}^{-1}\cdot\bz_2)\overline{  \cF^{-1}(\tilde{\phi}\tau_{m_1})(\bz_1)  }\cF^{-1}(\tilde{\phi}\tau_{m_2}) (\bz_2)  \varphi_s^{\BH^3} (\bz_1^{-1}g\bz_2) d\bz_1 d\bz_2 \right) h_\lambda(s)d\nu(s).
    \end{align*}
    We only need to show the inner integral above is $\ll_A s^{-A}\|\phi\|_{L^2(\BH^2)}^2$ by assuming $|s-\lambda|\leq\beta$.
    By the change of variables $\bz_1\to b_{m_1}\cdot\bz_1$ and $\bz_2\to b_{m_1}\cdot\bz_2$, the inner integral above is equal to
    \begin{align*}
        \iint_{\BH^2} \chi(b_{m_1}\cdot\bz_1) \chi(b_{m_2}\cdot\bz_2) \eta_{n_1}(\bz_1)\eta_{n_2}(\bz_2)\overline{  \cF^{-1}(\tilde{\phi}\tau_{m_1})(b_{m_1}\cdot\bz_1)  }\cF^{-1}(\tilde{\phi}\tau_{m_2}) (b_{m_2}\cdot\bz_2)  \varphi_s^{\BH^3} (\bz_1^{-1}b_{m_1}^{-1}gb_{m_2}\bz_2) d\bz_1 d\bz_2.
    \end{align*}
    Then we unfold the inverse Helgason transforms inside the integral to obtain
    \begin{align}\label{unfolding pairing}
        \begin{split}
            \iint_{\BR_{>0}\times B} \overline{\tilde{\phi}\tau_{m_1}(s_1,b_1)}\tilde{\phi}\tau_{m_2}(s_2,b_2)\left(\iint_{\BH^2} \chi(b_{m_1}\bz_1) \chi(b_{m_2}\bz_2) \eta_{n_1}(\bz_1)\eta_{n_2}(\bz_2) \varphi_s^{\BH^3} (\bz_1^{-1}b_{m_1}^{-1}gb_{m_2}\bz_2) \right.\\
            \exp\big((1/2-is_1)A(b_1^{-1}b_{m_1}\bz_1) +(1/2+is_2)A(b_2^{-1}b_{m_2}\bz_2)  \big) d\bz_1 d\bz_2  \Big)d\mu(s_1,b_1) d\mu(s_2,b_2).
        \end{split}
    \end{align}
    Since $\supp(\tilde{\phi}\tau_m) \cap (\BR_{>0} \times B)$ is contained in $[\lambda-\beta,\lambda+\beta]\times \bI_m \subset \BR_{>0}\times B$, and the support of $\eta_n$ is contained in $\bJ_n\subset\BR$, after we use the coordinates $(x,t)\in\BR^2\mapsto n(x)a(t)\cdot o\in\BH^2$, it suffices to prove the following integral
    \begin{align}
         \begin{split}
         \iint_{\BR}& \chi_1(x_1,t_1,x_2,t_2)\varphi_s^{\BH^3} (a(-t_1)g_1a(t_2)) \exp\left(-is_1A(b_1^{-1}b_{m_1}n(x_1)a(t_1))+{is_2A(b_2^{-1}b_{m_2}n(x_2)a(t_2))} \right)  dt_1 dt_2 
        \end{split}\label{two-line inner integral}
    \end{align}
    is $\ll_A s^{-A}$, provided $s,s_1,s_2\in [\lambda-\beta,\lambda+\beta]$, $b_1\in \bI_{m_1}$, $x_1\in \bJ_{n_1}$, $b_2\in \bI_{m_2}$, and $x_2\in \bJ_{n_2}$. 
    Here $g_1 = n(-x_1)b_{m_1}^{-1}gb_{m_2}n(x_2) $, and $\chi_1 \in C_c^\infty(\BR^4)$ is the combination of $\chi(b_{m_1}n(x_1)a(t_1)\cdot o)\chi(b_{m_2}n(x_2)a(t_2)\cdot o)$
    with all the amplitude factors in the integral. (Here we omit the variables $b_1,b_2$.)
    In particular, we have $b_1^{-1}b_{m_1}, b_2^{-1}b_{m_2} \in \bI_0$, which is the interval in $B$ centered at the identity with length $\ll  \lambda^{-1/2 + \epsilon_1}\beta^{1/2}$. 
    Applying Corollary \ref{uniformization lemma for B}, there is a constant $\delta>0$,  and a uniformly nonvanishing function $\xi$ on $(-\delta,\delta)\times (-2,2)^2$
    such that
    \begin{align*}
        \frac{\partial}{\partial t} A(b_\theta n(x)a(t)) = 1-\theta^2\xi(\theta,x,t),
    \end{align*}
    for $(\theta,x,t)\in(-\delta,\delta) \times (-2,2)^2 $. If $Z(\theta,x,t)$ is an antiderivative of $\xi$ with respect to $t$ that is smooth as a function of $(\theta,x,t)$, we may integrate this to obtain
    \begin{align*}
        A(b_\theta n(x)a(t)) = t - \theta^2 Z(\theta,x,t) + c(\theta,x)
    \end{align*}
    for some function $c(\theta,x)$. In our case, we have
    \begin{align*}
    A(b_j^{-1}b_{m_j}n(x_j)a(t_j)) = t_j + \rho_j(t_j), \quad j=1,2,
    \end{align*}
    where if we let $\theta_j\in (-\delta,\delta)$ satisfying $b_{\theta_j} = b_j^{-1}b_{m_j}$, then
    \begin{align*}
        \rho_j(t_j) := - \theta_j^2 Z(\theta_j,x_j,t_j) + c(\theta_j,x_j),\quad j=1,2,
    \end{align*}
    satisfy the condition \eqref{condition for rho two dimensional integral} in Proposition \ref{nonstationary oscillatory integral estimate} (by taking $\epsilon_0 = 2\epsilon_1$):
    \begin{align*}
        \left|s_j \frac{\partial\rho_j}{\partial t_j} (t_j)\right|\ll s\theta_j^2 |\xi|\ll s\lambda^{-1+2\epsilon_1}\beta \ll s^{2\epsilon_1}\beta,\\
        \left|s_j \frac{\partial^n\rho_j}{\partial t_j^n} (t_j)\right| \ll s\theta_j^2\left|\frac{\partial^{n-1}\xi}{\partial t_j^{n-1}} \right|  \ll_n s^{2\epsilon_1}\beta.
    \end{align*}
    It follows by Proposition \ref{nonstationary oscillatory integral estimate} that (\ref{two-line inner integral}) is $\ll_{\epsilon_1,A} \lambda^{-A}$ if
    \begin{align}\label{non uniformized condition}
        d(g_1,e)\text{ is bounded, }\;\; \text{ and }\;\; d(g_1,MA)\geq \frac{1}{2}s^{-1/2 + 2\epsilon_1}\beta^{1/2}.
    \end{align}
     Recall that $g_1 = n(-x_1)b_{m_1}^{-1}gb_{m_2}n(x_2)$. 
    Thus, it suffices to show our assumptions (\ref{uniform condition}) can imply (\ref{non uniformized condition}).
    Since
    \begin{align*}
        n(-x_{n_1})b_{m_1}^{-1}gb_{m_2}n(x_{n_2}) = n(-x_{n_1}+ x_1) g_1 n(x_{n_2} - x_2)
    \end{align*}
    and $|x_{n_i}- x_i| \ll \lambda^{-1/2}\beta^{1/2}$, we have
    \begin{align}\label{small error}
        d(n(-x_{n_1})b_{m_1}^{-1}gb_{m_2}n(x_{n_2}),g_1) \ll \lambda^{-1/2}\beta^{1/2}.
    \end{align}
    Then if $\lambda$ is sufficiently large, (\ref{uniform condition}), (\ref{small error}) and the triangle inequality imply (\ref{non uniformized condition}).
\end{proof}

Given $\phi\in\cS(\BH^2)\cap H_\beta$, we define
\begin{align}\label{eq: defn for varphi mn}
    \varphi_{m,n}(\bz) =\eta_n(b_m^{-1}\cdot\bz)\cF^{-1}(\tilde{\phi}\tau_m)(\bz).
\end{align}
Note that $\varphi_{m,n}$ is a smooth function on $\BH^2$ but is not compactly supported. From \eqref{eq: defn of geodesic beam}, we have $\phi_{m,n} = \chi\varphi_{m,n}$, and this relation implies that
\begin{align*}
    I(\lambda,\varphi_{m,n},g) &= \iint_{\BH^2} \overline{\chi\varphi_{m,n}(\bz_1)}\chi\varphi_{m,n}(\bz_2) k_\lambda (\bz_1^{-1}g\bz_2) d\bz_1 d\bz_2\\
    &=\iint_{\BH^2} \overline{\phi_{m,n}(\bz_1)}\phi_{m,n}(\bz_2) k_\lambda (\bz_1^{-1}g\bz_2) d\bz_1 d\bz_2\\
    &=P(\phi_{m,n},\phi_{m,n},g;\lambda).
\end{align*}
Here $I(\lambda,\varphi_{m,n},g)$ is defined as in Proposition \ref{amplification inequality}. Corollary \ref{nonstationary phase bound for P} directly implies the following.
\begin{corollary}\label{small I result}
Suppose that  $\phi\in\cS(\BH^2)\cap H_\beta$ and let $\chi\phi = \sum_{m,n}\phi_{m,n}$ be the decomposition given by \eqref{eq: defn of geodesic beam}.
    Suppose that $\lambda^{\epsilon'}\leq \beta\leq \lambda^{1/4}$ and $d(g,e)\leq 1$. If
    \begin{align*}
       d\left(\left(b_{m}n(x_{n})\right)^{-1}g\left(b_{m}n(x_{n})\right), MA\right) \geq \lambda^{-1/2 + 2\epsilon_1}\beta^{1/2},
    \end{align*}
    then we have 
    \begin{align*}
        I(\lambda,\varphi_{m,n},g)\ll_{\epsilon',\epsilon_1,A} \lambda^{-A}\|\phi\|_{L^2(\BH^2)}^2.
    \end{align*}
\end{corollary}

We also want to bound $I(\lambda,\varphi_{m,n},g)$ without conditions on $g$. Since the support of $\phi_{m,n}$ is at a scale $\asymp \lambda^{-1/2}\beta^{1/2}\times 1$, we have the following control of its $L^1$-norm.

\begin{lemma}\label{L1 norm bound for phi mn}
    We have
    \begin{align*}
        \|\phi_{m,n} \|_{L^1(\BH^2)}\ll \lambda^{-1/4}\beta^{1/4} \| \phi_{m,n}\|_{L^2(\BH^2)}.
    \end{align*}
\end{lemma}
\begin{proof}
    Let $\chi^\prime \in C_c^\infty(\BH^2)$ satisfying
    \begin{itemize}
        \item  $0\leq\chi^\prime\leq 1$,
        \item $\chi^\prime \equiv 1$ on $\{(x,e^t)\in\BH^2: x\in \bJ_n, -C\leq t\leq C\}$,
        \item $\supp(\chi^\prime )\subset \{(x,e^t)\in\BH^2: x\in 2\bJ_n, -2C\leq t\leq 2C\}$.
    \end{itemize}
     Here $C$ is a sufficiently large constant so that $\supp(\chi(b\cdot))$ is contained in   $\{(x,e^t)\in\BH^2: -C\leq x, t\leq C\}$ for any $b\in B$,
     and $2\bJ_n$ is the closed interval in $\BR$ centered at $x_n$ with length $2|\bJ_n|$.
    
    Then we have $\phi_{m,n}(\bz) = \chi^\prime (b_m^{-1}\bz)\phi_{m,n}(\bz)$. By the Cauchy-Schwarz inequality, we obtain
    \begin{align*}
        \int_{\BH^2} \left| \phi_{m,n}(\bz)  \right|d\bz  &= \int_{\BH^2} \left| \chi^\prime(b_m^{-1}\bz)\phi_{m,n}(\bz)  \right|d\bz \\
        &\leq \left(\int_{\BH^2} \left| \chi^\prime(\bz) \right|^2d\bz\right)^{1/2}\| \phi_{m,n}\|_{L^2(\BH^2)}\\
        &\ll \lambda^{-1/4} \beta^{1/4}\| \phi_{m,n}\|_{L^2(\BH^2)}.
    \end{align*}
\end{proof}

\begin{lemma}\label{Bound for oscillatory intgeral of phi mn depending on the k variable}
    Fix a smooth cutoff function $\chi_2\in C_c^\infty(\BH^2)$.
    Suppose $|s-\lambda|\leq \beta$, $\beta\leq 3s^{1/4}$
    and $u\in K_0$ satisfying
    \begin{align}\label{condition for kbm}
        d(u,Mb_m^{-1}) \geq s^{-1/2+2\epsilon_1} \beta^{1/2}.
    \end{align}
    We have
    \begin{align}\label{oscillatory integral for phi mn}
        \int_{\BH^2}\chi_2(\bz) \overline{\phi_{m,n}(\bz)}\exp\left(  is A(u\bz)\right)d\bz \ll s^{-A}\|\phi\|_{L^2(\BH^2)}.
    \end{align}
    The implied constant depends on $\epsilon_1$, and the size of the first $n$ derivatives of $\chi_2$ and $\chi$, where $n$ depends on $\epsilon_1$ and $A$.
\end{lemma}
\begin{proof}
    We apply the definition for $\phi_{m,n}$ to obtain
    \begin{align*}
        &\int_{\BH^2}\chi_2(\bz) \overline{\phi_{m,n}(\bz)}\exp\left(  is A(u\bz)\right)d\bz\\
        =&\int_{\BH^2}\chi_2(\bz) \chi(\bz) \eta_n(b_m^{-1}\bz)\overline{\cF^{-1} (\tilde{\phi}\tau_m)(\bz)}\exp\left(  is A(u\bz)\right)d\bz \\
        =& \int_{\BR_{>0}\times  B}\left( \int_{\BH^2}\chi_2(\bz) \chi(\bz) \eta_n(b_m^{-1}\bz)\exp\left( (-ir+1/2)A(b^{-1}{\bz}) \right)\exp\left(  is A(u\bz)\right)d\bz  \right) \overline{\tilde{\phi}(r,b) }\tau_m(b) d\mu(r,b)\\
        =&\int_{\BR_{>0}\times  B}\int_\BR\left(  \int_\BR  \chi_3(x,t,b)\exp(isA(ub_mn(x)a(t)) - ir A(b^{-1}b_m n(x)a(t)) ) dt\right) \eta_n(x)  \overline{\tilde{\phi}(r,b) }\tau_m(b) \,dx\,d\mu(r,b),
    \end{align*}
    where
    \begin{align*}
        \chi_3(x,t,b) = \chi_2(b_m\bz) \chi(b_m \bz)e^{A(b^{-1}b_m\bz)/2 -t},\quad\text{ for }\bz = (x,e^t),
    \end{align*}
    is a smooth and compactly supported function on $\BR\times\BR\times B$. We may later omit the variables $x$ and $b$ for $\chi_3$. Because of the supports of $\tilde{\phi}$ and $\tau_m$, to prove the bound (\ref{oscillatory integral for phi mn}), it suffices to prove the following bound
    \begin{align}\label{reduced oscillatory integral for phi mn}
        \int_\BR \chi_3(t) \exp(is A(ub_mn(x)a(t)) - ir A(b^{-1}b_mn(x)a(t))) dt \ll s^{-A}
    \end{align}
    for $r\in[\lambda-\beta,\lambda+\beta]$ and $b\in \bI_m$. Hence, $b^{-1}b_m\in \bI_0$. Applying Corollary \ref{uniformization lemma for B}, there is a constant $\delta>0$,  and a uniformly nonvanishing function $\xi$ on $(-\delta,\delta)\times (-2,2)^2$
    such that
    \begin{align*}
        \frac{\partial}{\partial t} A(b_\theta n(x)a(t)) = 1-\theta^2\xi(\theta,x,t),
    \end{align*}
    for $(\theta,x,t)\in(-\delta,\delta) \times (-2,2)^2 $. 
    We denote by $b_\theta = b^{-1} b_m$ where $|\theta| \ll \lambda^{-1/2+\epsilon_1}\beta^{1/2}$.
    If $Z(\theta,x,t)$ is an antiderivative of $\xi$ with respect to $t$ that is smooth as a function of $(\theta,x,t)$, we may integrate this to obtain
    \begin{align*}
        A(b_\theta n(x)a(t)) = t - \theta^2 Z(\theta,x,t) + c(\theta,x)
    \end{align*}
    for some function $c(\theta,x)$. Therefore, we have
    \begin{align*}
    A(b^{-1}b_{m}n(x)a(t)) = t + \rho(t),
    \end{align*}
    where
    \begin{align*}
        \rho(t) = - \theta^2 Z(\theta,x,t) + c(\theta,x)
    \end{align*}
    satisfies the condition (\ref{condition for rho}) in Proposition \ref{first one dimensional intgeral estimate}, i.e.,
    \begin{align*}
        \left|r \frac{\partial\rho}{\partial t} (t)\right|\ll r\theta^2 |\xi|\ll r\lambda^{-1+2\epsilon_1}\beta \ll r^{2\epsilon_1}\beta,\\
        \left|r \frac{\partial^n\rho}{\partial t^n} (t)\right| \ll r\theta^2\left|\frac{\partial^{n-1}\xi}{\partial t^{n-1}} \right|  \ll_n r^{2\epsilon_1}\beta.
    \end{align*}
    We take $\epsilon_0 = 2\epsilon_1$. 
    The condition (\ref{35}) will follow from (\ref{condition for kbm}). That is, there exists a constant $B>0$ so that $d(ub_m, M)\geq Bs^{-1/2+\epsilon_0} \beta^{1/2}$. Then we deduce \eqref{reduced oscillatory integral for phi mn} from Proposition \ref{first one dimensional intgeral estimate}.
\end{proof}

\begin{proposition}\label{uniform bound for phi mn}
Suppose $\lambda^{\epsilon'}\leq \beta\leq \lambda^{1/4}$ and  $\phi\in\cS(\BH^2)\cap H_\beta$. Let $\chi\phi = \sum_{m,n}\phi_{m,n}$ be the decomposition given by \eqref{eq: defn of geodesic beam}. Then we have
    $$I(\lambda,\varphi_{m,n},g)\ll_{\epsilon',\epsilon_1,A}\lambda^{1/2 + 4\epsilon_1} \beta^{3/2}\|\phi_{m,n}\|_{L^2(\BH^2)} ^2 +  \lambda^{-A} \| \phi \|_{L^2(\BH^2)}^2$$    
    for all $g\in G_0$ with $d(g,e)\leq 1$.
\end{proposition}

This bound follows directly from a more general result as follows. 

\begin{proposition}\label{general uniform bound}
Suppose $
\lambda^{\epsilon'}\leq \beta\leq \lambda^{1/4}$ and $\phi\in\cS(\BH^2)\cap H_\beta$. Let $\chi\phi = \sum_{m,n}\phi_{m,n}$ be the decomposition given by \eqref{eq: defn of geodesic beam}. Then we have
    \begin{align*}
        P(\phi_{m_1,n_1},\phi_{m_2,n_2},g;\lambda) \ll_{\epsilon',\epsilon_1,A} \lambda^{1/2+ 4\epsilon_1} \beta^{3/2}\|\phi_{m_1,n_1}\|_{L^2(\BH^2)} \|\phi_{m_2,n_2}\|_{L^2(\BH^2)}+ \lambda^{-A} \| \phi \|_{L^2(\BH^2)}^2
    \end{align*}
    for all $g\in G_0$ with $d(g,e)\leq 1$.
\end{proposition}
\begin{proof}
   We first unfold the integral $P(\phi_{m_1,n_1},\phi_{m_2,n_2},g;\lambda)$ by the inverse Harish-Chandra transform for $k_\lambda$ and  denote by 
   \begin{align*}
       Q(\phi_{m_1,n_1},\phi_{m_2,n_2},g;s) = \iint_{\BH^2} \overline{\phi_{m_1,n_1}(\bz_1)}\phi_{m_2,n_2}(\bz_2) \varphi^{\BH^3}_s (\bz_1^{-1}g\bz_2) d\bz_1 d\bz_2.
   \end{align*}
   Then
   \begin{align*}
       P(\phi_{m_1,n_1},\phi_{m_2,n_2},g;\lambda) =& \int_0^\infty Q(\phi_{m_1,n_1},\phi_{m_2,n_2},g;s)h_\lambda(s) d\nu(s).
    \end{align*}
    By applying Lemma \ref{L1 norm bound for phi mn}, it suffices to prove
   \begin{align}\label{bound for Q phi g s}
       Q(\phi_{m_1,n_1},\phi_{m_2,n_2},g;s)\ll s^{-1+4 \epsilon_1} \beta \|\phi_{m_1,n_1}\|_{L^1(\BH^2)} \|\phi_{m_2,n_2}\|_{L^1(\BH^2)} + O_A(s^{-A} \| \phi \|_{L^2(\BH^2)}^2)
   \end{align}
   for $s\in [\lambda-\beta,\lambda+\beta]$. As before, we write $\bz_j = n(x_j) a(t_j)\cdot o = (x_j,e^{t_j})$.
   If we apply the functional equation $\varphi_s^{\BH^3} = \varphi_{-s}^{\BH^3}$, write $\varphi_{-s}^{\BH^3}$ as an integral over $K_0$ by the integral formula \eqref{eq:HC int for H3}, and apply Lemma \ref{spliting A}, we obtain
       \begin{align*}
           &Q(\phi_{m_1,n_1},\phi_{m_2,n_2},g;s)\\
           =&\iint_{\BH^2} \overline{\phi_{m_1,n_1}(\bz_1)}\phi_{m_2,n_2}(\bz_2) \varphi^{\BH^3}_{-s} (a(-t_1)n(-x_1)gn(x_2)a(t_2)) d\bz_1 d\bz_2\\
           =&\iint_{\BH^2} \int_{M\backslash K_0}\overline{\phi_{m_1,n_1}(\bz_1)}\phi_{m_2,n_2}(\bz_2) \exp\left( (1-is)A(ka(-t_1)n(-x_1)gn(x_2)a(t_2))\right) dk d\bz_1 d\bz_2\\
           =&\iint_{\BH^2} \int_{M\backslash K_0}\overline{\phi_{m_1,n_1}(\bz_1)}\phi_{m_2,n_2}(\bz_2) \exp\left( (1-is)\left(A(ugn(x_2)a(t_2)) -A(un(x_1)a(t_1)) \right)\right) \left|\det \frac{\partial k}{\partial u}\right| du d\bz_1 d\bz_2.
   \end{align*}
    Here we use $\partial k/\partial u$ to denote the Jacobian matrix of the diffeomorphism $u\mapsto k$ from  $M\backslash K_0$ onto itself, which is defined by the relation $k a(-t_1) n(-x_1) \in NA u$, or equivalently $un(x_1)a(t_1)\in NAk$. Let $\chi_1 \in C_c^\infty(\BH^2)$ be a fixed cutoff function so that $\phi_{m,n} \chi_1 = \phi_{m,n}$ for arbitrary $m,n$, and let 
    \begin{align*}
        \chi_2(\bz_1,u) = \chi_1(\bz_1) \exp(-A(un(x_1)a(t_1)))\left|\det \frac{\partial k}{\partial u}\right| ,
    \end{align*}
    which lies in $ C_c^\infty(\BH^2\times M\backslash K_0)$.
    We have
    \begin{align*}
        &Q(\phi_{m_1,n_1},\phi_{m_2,n_2},g;s)\\
        =&\int_{\BH^2}   \int_{M\backslash K_0} \left(  \int_{\BH^2}\chi_2(\bz_1,u) \overline{\phi_{m_1,n_1}(\bz_1)}\exp\left(  is A(un(x_1)a(t_1))\right)d\bz_1\right)\phi_{m_2,n_2}(\bz_2) e^{ (1-is)A(ugn(x_2)a(t_2)) } du  d\bz_2\\
        \ll& \int_{M\backslash K_0} \left|  \int_{\BH^2}\chi_2(\bz_1,u) \overline{\phi_{m_1,n_1}(\bz_1)}\exp\left(  is A(un(x_1)a(t_1))\right)d\bz_1\right| du \cdot \int_{\BH^2}  \left|\phi_{m_2,n_2} (\bz_2)\right|  d\bz_2.
    \end{align*}
    We apply Lemma \ref{Bound for oscillatory intgeral of phi mn depending on the k variable} to get the above integral inside the absolute value is $\ll_A s^{-A}\|\phi\|_{L^2(\BH^2)}$ if $d(u, Mb_{m_1}^{-1}) \geq s^{-1/2+2\epsilon_1} \beta^{1/2}$. 
    This implies that we may restrict the integral over $M\backslash K_0$ to the disc of radius $s^{-1/2+2\epsilon_1} \beta^{1/2}$ centered at $Mb_{m_1}$, and this gives the main term in the estimate \eqref{bound for Q phi g s}.
\end{proof}

To apply the amplification inequality from \eqref{amplification inequality eqn}, we need first to construct an element $\cT_v\in\cH_v$ for a finite place $v$ of $F$ that will form part of the amplifier $\cT$. Recall that we have defined Hecke operators $T_v(1),T_v(2)$ in Section \ref{sec of Hecke algberas}. In this case, we consider those $v\in\sQ$, that is, $v$ is inert.
Note that the Hecke operator $T_v(1)$ (resp. $T_v(2)$) corresponds to the operator summing over the set of nodes at distance 2 (resp. 4) from the given node in the Bruhat-Tits tree for $\mathrm{PGL}(2,E_v)$. Then an elementary computation, or applying the classical relation \cite[\S 1.4, (4.12)]{bump1998automorphic}, gives the following Hecke relation:
\begin{align}\label{Hecke relation inert case}
    T_v(1)*T_v(1) = q_v^2(q_v^2+1) +(q_v^2-1)T_v(1)+ T_v(2).
\end{align}
If we define the real numbers $\tau_v(1)$ and $\tau_v(2)$ by
\begin{align*}
    T_{v}(1)\psi = \tau_v(1)q_v^2 \psi,\quad T_{v}(2)\psi = \tau_v(2)q_v^4\psi,
\end{align*}
then (\ref{Hecke relation inert case}) implies that we cannot have both $|\tau_v(1)|\leq 1/4$ and $|\tau_v(2)|\leq1/4$.
We define
\begin{align}\label{amplifier inert}
    T_v = \begin{cases} {T_{v}(1)}/\tau_{v}(1)q_v^2\quad\text{ if }|\tau_{v}(1)|>1/4,\\
    {T_{v}(2)}/\tau_{v}(2)q_v^4\quad\text{ otherwise. }
    \end{cases}
\end{align}
It follows that $T_v\psi = \psi$ for all $v\in\sQ$. Note that $|K_v(1)/K_v|\asymp q_v^4$ and $|K_v(2)/K_v|\asymp q_v^8$, so $\|T_v(1)\|_{L^2}^2\asymp q_v^4$ and $\|T_v(2)\|_{L^2}^2\asymp q_v^8$. Hence,
\begin{align}\label{eq: L2 bd for Tv, v inert}
    \|T_v\|_{L^2}\ll 1.
\end{align}

\begin{proof}[Proof of Proposition \ref{amplified bound for geodesic beam}]
    Let $N\geq 1$ be a parameter to be chosen later. We define $\sQ_N = \{ v\in\sQ | N/2\leq q_v\leq N\}$, and define
\begin{align*}
    \cT_N = \sum_{v\in\sQ_N}  T_v ,
\end{align*}
where $T_v$'s are defined by (\ref{amplifier inert}). This choice of $\cT_N$ satisfies
\begin{align*}
    |\langle\cT_N\psi,\phi_{m,n} \rangle| = \#\sQ_N \cdot |\langle\psi,\phi_{m,n} \rangle| \gg N^{1-\epsilon} |\langle\psi,\phi_{m,n} \rangle|.
\end{align*}
By Proposition \ref{amplification inequality} applied to $\varphi_{m,n}$, we have
\begin{align}\label{amplified inequality 1}
     |\langle\psi,\phi_{m,n}\rangle|^2 \ll N^{-2+\epsilon} \sum_{\substack{\gamma\in \bG(F)\\ \gamma\in g_0 \cD_0 g_0^{-1}}} \left| (\cT_N*\cT_N^*)(\gamma)  I(\lambda,\varphi_{m,n},g_0^{-1}\gamma g_0) \right|.
\end{align}
Recall that we use $\cD_0\subset G_{0}$ to denote the compact subset consisting of $g\in G_0$ with $d(g,e)\leq 1$.
Corollary \ref{small I result} implies that we only need to consider the terms in (\ref{amplified inequality 1}) with
\begin{align*}
    d\left(\left(g_0b_{m}n(x_{n})\right)^{-1}\gamma \left(g_0b_{m}n(x_{n})\right), MA\right) < \lambda^{-1/2 + 2\epsilon_1}\beta^{1/2}.
\end{align*}
Without loss of generality, we may assume $g_0^\prime : = g_0b_{m}n(x_{n}) \in \Omega_{v_0}$ and $d(g_0^{\prime -1}\gamma g_0^\prime,e)\leq 2$ if $ \gamma\in g_0 \cD_0 g_0^{-1}$.
Let $C>0$ be the constant as in Proposition \ref{Hecke return wrt MA}.
We choose 
\begin{align*}
    X= (C^{-1}\lambda^{-1/2 + 2\epsilon_1}\beta^{1/2})^{-1}  =  C\lambda^{1/2 -2 \epsilon_1}\beta^{-1/2},
\end{align*}
so then take $\delta= CX^{-1} = \lambda^{-1/2 + 2\epsilon_1}\beta^{1/2}$.
We use the uniform bound 
\begin{align*}
    |I(\lambda,\varphi_{m,n},g_0^{-1}\gamma g_0) |\ll \lambda^{1/2 + 4\epsilon_1} \beta^{3/2}\|\phi_{m,n}\|_{L^2(\BH^2)} ^2 +  O_A(\lambda^{-A} \| \phi \|_{L^2(\BH^2)}^2)
\end{align*}
from Proposition \ref{uniform bound for phi mn} to obtain
\begin{align}\label{ine 3}
\begin{split}
   & \sum_{\substack{\gamma\in \bG(F)\\ \gamma\in g_0 \cD_0 g_0^{-1}\\d(g_0^{\prime -1}\gamma g_0^\prime,MA)\leq \delta}}  \left| (\cT_N*\cT_N^*)(\gamma)I(\lambda,\varphi_{m,n},g_0^{-1}\gamma g_0)\right|\\
    \ll& \lambda^{1/2 + 4\epsilon_1} \beta^{3/2}\|\phi_{m,n}\|_{L^2(\BH^2)} ^2\sum_{\substack{\gamma\in \bG(F)\\ \gamma\in g_0 \cD_0 g_0^{-1}\\d(g_0^{\prime-1}\gamma g_0^\prime,MA)\leq \delta}} \left| (\cT_N*\cT_N^*)(\gamma)\right| +  O_A(\lambda^{-A} \| \phi \|_{L^2(\BH^2)}^2).
\end{split}
\end{align}
We next expand $\cT_N*\cT_N^*$ as a sum
\begin{align*}
    \cT_N*\cT_N^* = \sum_{\fn\subset\cO} a_\fn 1_{K(\fn)}
\end{align*}
for some constants $a_\fn$.
If we now suppose $N^4\leq X$, that is
\begin{align*}
    N \leq C^{1/4} \lambda^{1/8 - \epsilon_1/2}\beta^{-1/8},
\end{align*}
then our choice of $X$ and $\delta$ means that we may apply Proposition \ref{Hecke return wrt MA} to show that for any $\fn\neq \cO$ appearing in the expansion of $\cT_N*\cT_N^*$, 
the term $1_{K(\fn)}$ makes a contribution $\ll_\epsilon X^\epsilon \ll_\epsilon \lambda^\epsilon$ to the sum in the right hand side of (\ref{ine 3}).
Moreover, the sum of $a_\fn$'s with $\fn \neq \cO$ is 
\begin{align}\label{boundedness for sum of primes}
    \ll\left(\sum_{v\in\sQ, N/2 \leq q_v \leq N}( {q_v^{-2}}+{q_v^{-4}})\right)^2 + \sum_{v\in\sQ, N/2 \leq q_v \leq N} {q_v^{-2}}\ll 1.
\end{align}
Hence, we only need to consider the term $\fn=\cO$.  By the bound \eqref{eq: L2 bd for Tv, v inert}, we have
\begin{align}\label{eq: bd for aO, v invert}
    a_\cO=[\cT_N*\cT_N^*] (e)=\|\cT_N\|_{L^2}^2= \sum_{v\in\sQ_N}\|T_v \|^2_{L^2}\ll N.
\end{align}
By applying \eqref{eq: bd for aO, v invert} in \eqref{ine 3}, we get
\begin{align*}
      \sum_{\substack{\gamma\in \bG(F)\\ \gamma\in g_0 \cD_0 g_0^{-1}\\d(g_0^{\prime -1}\gamma g_0^\prime,MA)\leq \delta}}  \left| (\cT_N*\cT_N^*)(\gamma)I(\lambda,\varphi_{m,n},g_0^{-1}\gamma g_0)\right| \ll \lambda^{1/2 + 4\epsilon_1} \beta^{3/2}(N + \lambda^\epsilon) \|\phi_{m,n} \|_{L^2(\BH^2)}^2 +  O_A(\lambda^{-A} \| \phi \|_{L^2(\BH^2)}^2).
\end{align*}
We conclude that
\begin{align*}
   |\langle\psi,\phi_{m,n}\rangle|^2 \ll \lambda^{1/2 + 4\epsilon_1} \beta^{3/2} N^{-2+\epsilon} (N + \lambda^\epsilon) \|\phi_{m,n} \|_{L^2(\BH^2)}^2 +  O_A(\lambda^{-A} \| \phi \|_{L^2(\BH^2)}^2).
\end{align*}
Choosing $N =  C^{1/4} \lambda^{1/8 - \epsilon_1/2}\beta^{-1/8}$ gives \eqref{eq: amplified geodesic beam, individual est}.

By the Cauchy–Schwarz inequality, Propositions \ref{l2L2 inequality} and \ref{no many big norm geodesic beams}, we obtain
\begin{align*}
    \sum_{\phi_{m,n} \in \cS_{\geq\delta}(\phi)}\|\phi_{m,n} \|_{L^2(\BH^2)} \leq (\# \cS_{\geq\delta}(\phi))^{1/2} \left(\sum_{\phi_{m,n} \in \cS_{\geq\delta}(\phi)}\|\phi_{m,n} \|_{L^2(\BH^2)}^2\right)^{1/2}\ll\delta^{-1/2}\|\phi\|_{L^2(\BH^2)}
\end{align*}
Then we combine \eqref{eq: amplified geodesic beam, individual est} with the above inequality to get \eqref{eq: summing over amplified geodesic beam}.
\end{proof}

\subsection{Estimates for the sums of $\langle\psi,\phi_{m,n}\rangle$}
We first provide preliminary estimates for proving Proposition \ref{amplified bound for small norm}.
Given $\phi\in\cS(\BH^2)\cap H_\beta$ with the decomposition $\chi\phi = \sum_{m,n}\phi_{m,n}$ given by \eqref{eq: defn of geodesic beam}, for any $0<\delta<1$, we define
\begin{align*}
    \phi_\delta = \sum_{\phi_{m,n}\in\cS_{<\delta}(\phi)} \phi_{m,n}\quad\text{ and }\quad  \varphi_\delta = \sum_{\phi_{m,n}\in\cS_{<\delta}(\phi)}  \varphi_{m,n}.
\end{align*}
Recall that here $\varphi_{m,n}$ is defined by \eqref{eq: defn for varphi mn} and we have
$\phi_{m,n} = \chi\varphi_{m,n}$, and so $\phi_\delta = \chi \varphi_\delta$.


\begin{proposition}\label{L2 norm of phi delta}
    Suppose $\phi\in\cS(\BH^2)\cap H_\beta$.
    For any $0<\delta<1$, we have
    \begin{align*}
        \|\phi_\delta \|_{L^2(\BH^2)}\ll \|\phi\|_{L^2(\BH^2)}.
    \end{align*}
\end{proposition}

\begin{proof}
    We have
    \begin{align}\label{S delta expansion}
         \| \phi_\delta \|_{L^2(\BH^2)}^2 = \sum_{\phi_{m,n}\in \cS_{<\delta}(\phi)} \| \phi_{m,n}\|_{L^2(\BH^2)}^2 + \sum_{\substack{\phi_{m_1,n_1},{\phi_{m_2,n_2}\in \cS_{<\delta}(\phi)}\\(m_1,n_1)\neq(m_2.n_2)}} \langle \phi_{m_1,n_1} ,\phi_{m_2,n_2} \rangle .
    \end{align}
    The first term in the right-hand side of \eqref{S delta expansion} is $\ll \|\phi\|_{L^2(\BH^2)}^2$ by Propositions \ref{l2L2 inequality}.
    The second term is also $\ll \|\phi\|_{L^2(\BH^2)}^2$ by Proposition \ref{l2L2 inequality general form}.
\end{proof}

\begin{proposition}\label{uniform bound for phi delta}
     Suppose $\phi\in\cS(\BH^2)\cap H_\beta$. For any $0<\delta<1$ and $g\in H^\prime$, we have $$I(\lambda,\varphi_{\delta},g)\ll \lambda^{1/2} \|\phi_\delta\|_{L^2(\BH^2)}^2\ll \lambda^{1/2} \|\phi\|_{L^2(\BH^2)}^2.$$
\end{proposition}
\begin{proof}

    Since $g\in H^\prime$, $g$ is an isometry of $\BH^2$. We can define the translation
    \begin{align*}
         \phi_\delta^g (\bz):=\phi_\delta(g\bz),
    \end{align*}
    which is still a function on $\BH^2$.
    We may see that $\|\phi_\delta^g \|_{L^2(\BH^2)} = \|\phi_\delta \|_{L^2(\BH^2)}$.
    Then
    \begin{align*}
        I(\lambda,\varphi_\delta,g)&=\iint_{\BH^2}\overline{\phi_\delta(\bz_1)}\phi_\delta(\bz_2) k_\lambda (\bz_1^{-1}g\bz_2) d\bz_1 d\bz_2\\
        &= \iint_{\BH^2}\overline{\phi_\delta(\bz_1)}\phi_\delta(\bz_2) k_\lambda (\bz_2^{-1}g^{-1}\bz_1) d\bz_1 d\bz_2 \\
        &= \iint_{\BH^2}\overline{\phi_\delta(g\bz_1)}\phi_\delta(\bz_2) k_\lambda (\bz_2^{-1}\bz_1) d\bz_1 d\bz_2\\
        &=\langle \phi_\delta \times k_\lambda,\phi^g_\delta\rangle \\
        &=\langle \tilde{\phi_\delta} \tilde{k_\lambda},\tilde{\phi^g_\delta} \rangle \\
        &\ll \|\tilde{k_\lambda} \|_{L^\infty} \|\phi_\delta^g \|_{L^2(\BH^2)}  \|\phi_\delta \|_{L^2(\BH^2)}.
    \end{align*}
    The proposition then follows from Proposition \ref{supnorm of klambda} and Proposition \ref{L2 norm of phi delta}.
\end{proof}

\begin{proposition}\label{orbital integral away from SL2R}
     Suppose that $\phi\in\cS(\BH^2)\cap H_\beta$, $0<\delta<1$, and $\lambda^{\epsilon'}\leq\beta\leq \lambda^{1/4}$.
     Let $0 < \omega < 1$ be a small parameter satisfying $\omega^{-1} = o(\lambda^{1/2-2\epsilon_1}\beta^{-1/2})$.
     Then there exists an absolute constant $C>0$, such that the following holds.
     If $d(g,e)\leq 1$ and
    $
       d\left(g,H^\prime\right) \geq C \omega,
    $
    then we have 
    $$I(\lambda,\varphi_{\delta},g)\ll_{\epsilon',\epsilon_1,A} \delta\omega^{-2} \lambda^{1/2 + 10\epsilon_1}\beta^{3/2}\|\phi\|_{L^2(\BH^2)}^2 + \lambda^{-A}\|\phi\|_{L^2(\BH^2)}^2.$$
\end{proposition}
\begin{proof}
    We may write 
    \begin{align*}
        I(\lambda,\varphi_{\delta},g) = \sum_{\phi_{m_1,n_1},\phi_{m_2,n_2}\in\cS_{<\delta}(\phi)}P(\phi_{m_1,n_1},\phi_{m_2,n_2},g;\lambda) 
    \end{align*}
    By Corollary \ref{nonstationary phase bound for P}, if 
    \begin{align}\label{5 6 draft}
        d(n(-x_{n_1})b_{m_1}^{-1}gb_{m_2}n(x_{n_2}), MA) \geq \lambda^{-1/2 + 2\epsilon_1}\beta^{1/2},
    \end{align}
    then 
    \begin{align*}
        P(\phi_{m_1,n_1},\phi_{m_2,n_2},g;\lambda) \ll_{\epsilon_1,A} \lambda^{-A}\|\phi\|_{L^2(\BH^2)}^2.
    \end{align*}
    We then need to consider the quadruples $(b_{m_1},x_{n_1},b_{m_2},x_{n_2})$ not satisfying \eqref{5 6 draft}. 
    By Proposition \ref{most of geodesic beams satisfy uniform condition}, there exists a $C>0$, so that as long as $d(g,e)\leq 1$ and
    $
       d\left(g,H^\prime\right) \geq C\omega
    $,
    the number of quadruples not satisfying (\ref{5 6 draft}) is 
    \begin{align*}
        \ll \omega^{-2} (\lambda^{-1/2 + 2\epsilon_1}\beta^{1/2})^4 N_1^2 N_2^2 \ll \omega^{-2} (\lambda^{-1/2 + 2\epsilon_1}\beta^{1/2})^4 (\lambda^{1/2-\epsilon_1}\beta^{-1/2})^2 (\lambda^{1/2}\beta^{-1/2})^2 \ll \omega^{-2} \lambda^{6\epsilon_1}.
    \end{align*}
    
    For each pair $\phi_{m_1,n_1},\phi_{m_2,n_2}\in\cS_{<\delta}$,
    by Proposition \ref{general uniform bound}, we have
    \begin{align*}
        P(\phi_{m_1,n_1},\phi_{m_2,n_2},g;\lambda)\ll \lambda^{1/2+ 4\epsilon_1} \beta^{3/2} \delta\|\phi\|_{L^2(\BH^2)}^2 + O_A(\lambda^{-A}\|\phi\|_{L^2(\BH^2)}^2).
    \end{align*}
    Therefore,
    \begin{align*}
        I(\lambda,\varphi_{\delta},g) \ll& \left(\omega^{-2} \lambda^{6\epsilon_1} \right)\cdot\left( \lambda^{1/2 + 4\epsilon_1}\beta^{3/2}\delta\|\phi\|_{L^2(\BH^2)}^2\right) + O_A(\lambda^{-A}\|\phi\|_{L^2(\BH^2)}^2)\\
        =& \delta\omega^{-2} \lambda^{1/2 + 10\epsilon_1}\beta^{3/2} \|\phi\|_{L^2(\BH^2)}^2 + O_A(\lambda^{-A}\|\phi\|_{L^2(\BH^2)}^2).
    \end{align*}
\end{proof}

Let $v\in\sP$ be a finite place that is split in $E$. 
Similar to \eqref{Hecke relation inert case}, we have the Hecke relation in $\cH_v$:
\begin{align}\label{Hecke relation}
    T_v(1)*T_v(1) = q_v(q_v+1) +(q_v-1)T_v(1)+ T_v(2).
\end{align}
If we define the real numbers $\tau_v(1)$ and $\tau_v(2)$ by
\begin{align*}
    T_{v}(1)\psi = \tau_v(1)q_v \psi,\quad T_{v}(2)\psi = \tau_v(2)q_v^2\psi,
\end{align*}
then (\ref{Hecke relation}) implies that we cannot have both $|\tau_v(1)|\leq 1/4$ and $|\tau_v(2)|\leq1/4$.
We define
\begin{align}\label{amplifier split}
    T_v = \begin{cases} {T_{v}(1)}/\tau_{v}(1)q_v\quad\text{ if }|\tau_{v}(1)|>1/4,\\
    {T_{v}(2)}/\tau_{v}(2)q_v^2\quad\text{ otherwise. }
    \end{cases}
\end{align}
It follows that $T_v\psi = \psi$ for all $v\in\sP$. Note that $|K_v(1,0)/K_v|\asymp q_v^2$ and $|K_v(2,0)/K_v|\asymp q_v^4$, so $\|T_v(1)\|_{L^1}=\|T_v(1)\|_{L^2}^2\asymp q_v^2$ and $\|T_v(2)\|_{L^1}=\|T_v(2)\|_{L^2}^2\asymp q_v^4$. Hence,
\begin{align}\label{eq: L1 bd for Tv, v split}
    \|T_v\|_{L^1}\ll q_v^2,
\end{align}
and
\begin{align}\label{eq: L2 bd for Tv, v split}
    \|T_v\|_{L^2}\ll 1.
\end{align}

\begin{proof}[Proof of Proposition \ref{amplified bound for small norm}]
    Let $N\geq 1$ be a parameter to be chosen later and let $X= N^{4}$. Apply Proposition \ref{Hecke return wrt SL(2,R)} to $g=g_0$ to get a constant $C_1$ (which is denoted by $C$ in Proposition \ref{Hecke return wrt SL(2,R)}) and the set of primes $\sP_0$. We define $\sP_N = \{ v\in\sP_0 \,|\, N/2\leq q_v\leq N\}$, and define
\begin{align*}
    \cT_N = \sum_{v\in\sP_N}  T_v ,
\end{align*}
where $T_v$'s are defined by (\ref{amplifier split}). This choice of $\cT_N$ satisfies
\begin{align*}
    |\langle\cT_N\psi,\phi_\delta \rangle| = \#\sP_N \cdot |\langle\psi, \phi_\delta \rangle| \gg N^{1-\epsilon} |\langle\psi,\phi_\delta \rangle| = N^{1-\epsilon}\left|\sum_{\phi_{m,n} \in \cS_{<\delta}(\phi)}\langle\psi,\phi_{m,n} \rangle\right|.
\end{align*}
By Proposition \ref{amplification inequality}, it follows that
\begin{align}\label{amplified inequality 2}
     \left|\sum_{\phi_{m,n} \in \cS_{<\delta}(\phi)}\langle\psi,\phi_{m,n} \rangle\right|^2 \ll_\epsilon N^{-2+\epsilon} \sum_{\substack{\gamma\in \bG(F)\\ \gamma\in g_0 \cD_0 g_0^{-1}}} \left| (\cT_N*\cT_N^*)(\gamma)I(\lambda,\varphi_\delta,g_0^{-1}\gamma g_0) \right|.
\end{align}
Here $\cD_0\subset G_{0}$ is the compact subset consisting of $g\in G_0$ with $d(g,e)\leq 1$.
Let $C>0$ be the constant given by Proposition \ref{orbital integral away from SL2R}.
We choose $\omega = \frac{1}{2}C^{-1}C_1 X^{-8} = \frac{1}{2}C^{-1}C_1 N^{-32}$, and break the sum in \eqref{amplified inequality 2} into those terms with $d(g_0^{-1}\gamma g_0,H^\prime)< C \omega$ and the complement. We denote the complement by $D$, that is
\begin{align*}
    D =\{\gamma\in G(F)| g_0^{-1}\gamma g_0\in \cD_0\text{ and }d(g_0^{-1}\gamma g_0,H^\prime)\geq C \omega \} .
\end{align*}
To estimate the sum over $D$, 
we first use the bound 
$|I(\lambda,\phi_\delta,g_0^{-1}\gamma g_0) |\ll \delta\omega^{-2} \lambda^{1/2 + 10\epsilon_1}\beta^{3/2}\|\phi\|_{L^2(\BH^2)}^2
\ll N^{64}\delta\lambda^{1/2 + 10\epsilon_1}\beta^{3/2}\|\phi\|_{L^2(\BH^2)}^2$
from Proposition \ref{orbital integral away from SL2R}. 
Here, we need to assume $0 < \omega <1$ satisfies
\begin{align}\label{suitable condition on omega}
    \omega^{-1} =o( \lambda^{1/2-2\epsilon_1}\beta^{-1/2}).
\end{align}
This gives
\begin{align*}
    \sum_{\gamma\in D} \left| (\cT_N*\cT_N^*)(\gamma)I(\lambda,\varphi_\delta,g_0^{-1}\gamma g_0) \right| &\ll N^{64}\delta\lambda^{1/2 + 10\epsilon_1}\beta^{3/2}\|\phi\|_{L^2(\BH^2)}^2\sum_{\gamma\in \bG(F)\cap g_0\cD_0g_0^{-1}}\left| (\cT_N*\cT_N^*)(\gamma) \right|\\
    &\ll N^{64}\delta\lambda^{1/2 + 10\epsilon_1}\beta^{3/2}\|\phi\|_{L^2(\BH^2)}^2\|\cT_N*\cT_N^*\|_{L^1}.
\end{align*}
Note that \eqref{eq: L1 bd for Tv, v split} gives $\|T_v\|_{L^1} \ll N^2$ so $\|\cT_N\|_{L^1} \ll N^3$, and therefore,
\begin{align*}
    \sum_{\gamma\in D} \left| (\cT_N*\cT_N^*)(\gamma)I(\lambda,\phi_\delta,g_0^{-1}\gamma g_0) \right| \ll N^{70}\delta\lambda^{1/2 + 10\epsilon_1}\beta^{3/2}\|\phi\|_{L^2(\BH^2)}^2.
\end{align*}
We next estimate the sum over $d(g_0^{-1}\gamma g_0,H^\prime)<C \omega $.
We first expand $\cT_N*\cT_N^*$ as a sum
\begin{align*}
    \cT_N*\cT_N^* = \sum_{\fn\subset\cO} a_\fn 1_{K(\fn)}
\end{align*}
for some constants $a_\fn$. 
Our choice of $X$, $\omega$ and the constant $C_1$ means that we may apply Proposition \ref{Hecke return wrt SL(2,R)} to show that for any $\fn\neq \cO$ appearing in the expansion of $\cT_N*\cT_N^*$, 
the term $1_{K(\fn)}$ makes no contribution to the sum in the right-hand side of \eqref{amplified inequality 2}.
Hence, we only need to consider the term $\fn=\cO$. By the bound \eqref{eq: L2 bd for Tv, v split}, we have
\begin{align}\label{eq: bd for aO}
    a_\cO=[\cT_N*\cT_N^*] (e)=\|\cT_N\|_{L^2}^2= \sum_{v\in\sP_N}\|T_v \|^2_{L^2}\ll N.
\end{align}
Hence,
\begin{align}\label{ine 3'}
    \sum_{\substack{\gamma\in \bG(F)\\ \gamma\in g_0 \cD_0 g_0^{-1}\\d(g_0^{-1}\gamma g_0,H^\prime)<C \omega}}  \left| (\cT_N*\cT_N^*)(\gamma)I(\lambda,\phi_\delta,g_0^{-1}\gamma g_0) \right| \ll N \sum_{\substack{\gamma\in \bG(F)\\ \gamma\in g_0\cD_0 g_0^{-1}\\d(g_0^{-1}\gamma g_0,H^\prime)<C \omega}} 1_{K_f}(\gamma) \left| I(\lambda,\phi_\delta,g_0^{-1}\gamma g_0) \right|.
\end{align}
Since $\gamma$ varies in a finite set, we may assume that only the identity contributes to the right-hand side of (\ref{ine 3'}).
We apply the bound
$|I(\lambda,\varphi_\delta,e) |\ll \lambda^{1/2} \|\phi\|_{L^2(\BH^2)}^2$ from Proposition \ref{uniform bound for phi delta} to obtain
\begin{align*}
     \sum_{\substack{\gamma\in \bG(F)\\ \gamma\in g_0 \cD_0 g_0^{-1}\\d(g_0^{-1}\gamma g_0,H^\prime)<C\omega}} \left| (\cT_N*\cT_N^*)(\gamma)I(\lambda,\phi_\delta,g_0^{-1}\gamma g_0) \right| \ll N\lambda^{1/2} \|\phi\|_{L^2(\BH^2)}^2.
\end{align*}
Adding our two estimates, we conclude that
\begin{align*}
   \left|\sum_{\phi_{m,n} \in \cS_{<\delta}(\phi)}\langle\psi,\phi_{m,n} \rangle\right|^2 \ll N^{-2+\epsilon}(N^{70}\delta\lambda^{1/2 + 10\epsilon_1}\beta^{3/2}+N\lambda^{1/2})\|\phi\|_{L^2(\BH^2)}^2.
\end{align*}
Choosing $N = \delta^{-1/69} \beta^{-1/46} = (\delta^{1/3}\beta^{1/2})^{-1/23}$ 
gives the desired bound.
Finally, we check that $N\geq1$ and
$\omega =\frac{1}{2}C^{-1} C_1N^{-32} = \frac{1}{2}C^{-1} C_1\delta^{32/69}\beta^{48/69} =\frac{1}{2}C^{-1}  C_1 (\delta^{1/3}\beta^{1/2})^{32/23}$
satisfies the assumption \eqref{suitable condition on omega} provided $\lambda^{\epsilon'}\leq \beta\leq \lambda^{1/4}$ and $\lambda^{-1/2}\beta^{1/2}\leq\delta\leq\beta^{-3/2}$ with $\epsilon_1$ sufficiently small.
\end{proof}

\section{Kakeya-Nikodym norms}\label{section of Kakeya-Nikodym norms}

In this section, we prove Theorem \ref{KN main theorem}.

\subsection{Notation and sketch of proof}
We first recall the Helgason transform on $\BH^3$.
For $f\in C_c^\infty(\BH^3)$,
the Helgason transform of $f$ is defined by
\begin{align*}
    \widehat{f}(s,kM):= \int_{\BH^3} f(x) e^{(1-is)A(k^{-1}x)} dx,
\end{align*}
where $s\in\BC$ and $kM \in K_0/M$.
The boundary $\partial\BH^3$ is naturally identified with $K_0/M$.
For $s\in\BC$ and $x\in\BH^3$, we denote by
$
    \varphi_s(x) = \varphi_s^{\BH^3}(x)
$
the spherical function on $\BH^3$ with spectral parameter $s$.
We shall omit the superscript for the spherical functions in this section.
If $f$ is left $K_0$-invariant, then its Helgason transform agrees with the Harish-Chandra transform on $\BH^3$, that is
\begin{align*}
    \widehat{f}(s,kM) = \widehat{f}(s) = \int_{\BH^3} f(x) \varphi_{-s}(x)dx
\end{align*}
for any $s\in\BC$ and $kM\in K_0/M$. For $f,g\in C_c^\infty(\BH^3)$, their convolution, which we denote by $f\times g$, is defined as the convolution on the group $G_0$, i.e.,
\begin{align*}
    (f\times g)(x) := \int_{G_0} f(y\cdot o) g(y^{-1}\cdot x) dy.
\end{align*}
If moreover $g$ is assumed to be $K_0$-invariant, then
\begin{align*}
    \widehat{f\times g}(s,kM) = \widehat{f}(s,kM) \widehat{g}(s).
\end{align*}
We recall the Plancherel formula on $\BH^3$.
For $f,g\in C_c^\infty(\BH^3)$,
\begin{align*}
    \int_{\BH^3} f(x)\overline{g(x)} dx = \int_0^\infty\int_{K_0/M} \widehat{f}(s,kM)\overline{\widehat{g}(s,kM)} d\nu(s) dk.
\end{align*}
Here we normalize the Haar measure on $K_0/M$ so that the total volume is 1.

Let $l_{[-1/2,1/2]}$ be the oriented vertical geodesic centered at $o$ of unit length, i.e.,
\begin{align*}
    l_{[-1/2,1/2]} = \{a(t)\cdot o: t\in[-1/2,1/2] \},
\end{align*}
and the positive orientation is upward.
Any oriented geodesic of unit length in $\BH^3$ is equal to $g\cdot l_{[-1/2,1/2]}$
for some $g \in G_0$.
We define the $\lambda^{-1/2}$-tube of $l_{[-1/2,1/2]}$ to be
\begin{align*}
    \sT = \sT_{\lambda^{-1/2}}=  \{ (z,e^t) : |z| \leq \lambda^{-1/2}, |t|\leq 1/2 \}.
\end{align*}
The $\lambda^{-1/2}$-tube of $g\cdot l_{[-1/2,1/2]}$ is defined to be $g\cdot \sT$.
Fix $g_0 \in \Omega_{v_0}$.
We shall consider the $L^2$-norm problem for $\psi$ restricted to the tube $g_0 \cdot \sT$. 

Let $\BD = \BD_{\lambda^{-1/2}}$ be the open disc of radius $\lambda^{-1/2}$ centered at $0$ in $\BC$, i.e.,
\begin{align*}
    \BD =\BD_{\lambda^{-1/2}} =  \{z\in\BC:|z|<\lambda^{-1/2} \},
\end{align*}
and denote by $\overline{\BD}$ its closure.
We denote by $\cS(\BC \times \BR)$ the space of Schwartz functions on the Euclidean space $\BC \times \BR \simeq \BR^3$,
and let $\cS(\overline{\BD} \times \BR)$ be the subspace consisting of Schwartz functions supported in $\overline{\BD} \times \BR$. 
Let $\phi(z,t)\in \cS(\BC\times\BR)$. 
We shall consider the Fourier transform in the $t$-variable, i.e.,
\begin{align*}
    (\sF_t \phi)(z,s) := \int_{-\infty}^\infty \phi(z,t) e^{-is t} dt.
\end{align*}
We will denote the Fourier inversion by $\sF_t^{-1}$.
For $\phi \in \cS(\BC\times\BR)$, we have the integral formula for the Fourier inversion
\begin{align*}
    \phi(z,t) = \frac{1}{2\pi} \int_{-\infty}^\infty (\sF_t\phi)(z,s) e^{ist}ds.
\end{align*}
If $b\in C_c^\infty(\overline{\BD} \times (-1,1))$ is a non-negative cutoff function and $\phi\in \cS(\BC\times\BR)$,
we define the inner product on $g_0\cdot \sT$ by
\begin{align*}
    \langle \psi,b\phi\rangle :=  \int_{-\infty}^\infty\int_{\BD} \psi (g_0n(z)a(t)\cdot o) \overline{b(z,t) \phi(z,t)} e^{-2t}dtdz.
\end{align*}
Here the $e^{-2t}$ factor comes from the hyperbolic metric.

The same as the previous sections, we let $\epsilon'>0$ be a small constant satisfying $\epsilon'\leq10^{-6}$.
Let $\beta$ be a parameter satisfying $\lambda^{\epsilon'}\leq\beta\leq\lambda^{1/4}$.
We shall bound the part of $b\psi$ with Fourier transform away from the spectrum $\lambda$ using a purely analytic argument,
and bound the part with Fourier transform near the spectrum via the method of arithmetic amplification.
Theorem \ref{KN main theorem} follows from combining the following propositions with the choice $\beta = \lambda^{1/5}$.

\begin{proposition}\label{KN bound 1-Pi}
    If $\lambda^{\epsilon'}\leq\beta\leq\lambda$ and $\phi \in \cS(\overline{\BD}\times\BR) $  satisfies $\|\phi \|_2 = 1$ and $$\supp(\sF_t\phi) \subset \overline{\BD}\times (\BR\backslash \pm[\lambda-\beta,
    \lambda+\beta]),$$ 
    we have $\langle \psi,b\phi \rangle \ll_{\epsilon',\epsilon} \beta^{-1/4+\epsilon}$.
\end{proposition}

\begin{proposition}\label{KN bound Pi}
    If $\lambda^{\epsilon'}\leq\beta\leq\lambda^{1/4}$ and $\phi \in \cS(\overline{\BD}\times\BR) $  satisfies $\|\phi \|_2 = 1$ and $$\supp(\sF_t\phi) \subset \overline{\BD}\times (\pm[\lambda-\beta,
    \lambda+\beta]),$$ 
    we have $\langle \psi,b\phi \rangle \ll_{\epsilon',\epsilon} \lambda^{-1/16+\epsilon}\beta^{1/16}$.
\end{proposition}

We construct the same test function as in Section \ref{section of amplification}.
We fix a real-valued function $h\in C^\infty(\BR)$ of Paley-Wiener type that is nonnegative and satisfies $h(0)=1$. Define $h_\lambda^0(s) = h(s-\lambda)+h(-s-\lambda)$, and let $k_\lambda^0$ be the $K_0$-bi-invariant function on $\BH^3$ with Harish-Chandra transform $h^0_\lambda$.
The Paley-Wiener theorem implies that $k_\lambda^0$ is of compact support that may be chosen arbitrarily small. Define $k_\lambda = k_\lambda^0*k_\lambda^0$, which has Harish-Chandra transform $h_\lambda = (h_\lambda^0)^2$.

Let $b \in C_c^\infty(\overline{\BD} \times (-1,1))$ be a non-negative cutoff function.
If $g\in G_{0}$ and $\phi\in \cS(\BC\times\BR)$, we define
\begin{align*}
    J(\lambda,\phi,g) = \iint_{\BC\times \BR} \overline{b\phi(z_1,t_1})b\phi(z_2,t_2) k_\lambda (a(-t_1)n(-z_1)gn(z_2)a(t_2)) e^{-2t_1 -2 t_2} dt_1dz_1 dt_2dz_2.
\end{align*}
We prove a similar inequality as in Proposition \ref{amplification inequality} for geodesic tubes.
\begin{proposition}\label{KN amplification inequality for geodesic tube}
    Suppose $\cT\in\cH^S
    $ and $\phi\in\cS(\BC\times\BR)$. We have
    \begin{align}\label{KN amplification inequality eqn}
        \left| \langle \cT\psi,b\phi\rangle\right|^2\ll \sum_{\gamma\in \bG(F)} \left| (\cT*\cT^*)(\gamma)J(\lambda,\phi,g_0^{-1}\gamma g_0) \right|.
    \end{align}
\end{proposition}

\begin{proof}
    We consider the same kernel function as in  Proposition \ref{amplification inequality} with  the spectral expansion
    \begin{align*}
        K(x,y)= \sum_{\gamma\in\bG(F)} k_\infty (\cT*\cT^*)(x^{-1}\gamma y) = \sum_i h_\lambda(\lambda_i)\cT\psi_i(x)\overline{\cT\psi_i(y)}
    \end{align*}
    on $\bG(F)\backslash\bG(\BA)\times\bG(F)\backslash\bG(\BA)$.
    If we integrate it against $\overline{b\phi}\times b\phi$ on $g_0\sT \times g_0\sT$, we obtain
    \begin{align*}
         \sum_i h_\lambda(\lambda_i)\left| \langle \cT\psi,{b}\phi\rangle\right|^2& =\sum_{\gamma\in \bG(F)} (\cT*\cT^*)(\gamma)J(\lambda,\phi,g_0^{-1}\gamma g_0).
    \end{align*}
    Since we have $h_\lambda(\lambda_i)\geq 0 $ for all $i$, dropping all terms but $\psi$ completes the proof.
\end{proof}

\subsection{Bounds away from the spectrum}\label{KN section of bound away from the spectrum}
We start the proof for Proposition \ref{KN bound 1-Pi}.
In Proposition \ref{KN amplification inequality for geodesic tube}, we can choose $\cT \in \cH_f$  to be the characteristic function of a
sufficiently small open subgroup of $\bG(\BA_f)$, then only the identity element $\gamma=e$ will make a nonzero contribution to the sum in (\ref{KN amplification inequality eqn}). This gives
\begin{align}\label{KN amplification inequality trivial case}
    \left| \langle \psi,b\phi\rangle\right|^2\ll\left| J(\lambda,\phi,e) \right|.
\end{align}
We first decompose $k_\lambda$ along the radial direction.
Without loss of generality, we assume the support of $k_\lambda(a(t))$ is contained in $[-1,1]$.
Let $b_0\in C_c^\infty(\BR)$ be a cutoff function that is equal to $1$ on $[-1,1]$ and zero outside $[-2,2]$.
Moreover, we suppose that $b_0$ is even and $0\leq b_0\leq 1$.
Then $k_\lambda(a(t)) = b_0(t) k_\lambda(a(t))$ for any $t\in\BR$.
Let $\epsilon_0 > 0$.
Define $b_1 \in C_c^\infty(\BR)$ to be a function satisfying
\begin{itemize}
    \item $0\leq b_1 \leq 1$ and $b_1$ is even,
    \item $b_1(t) = 1$ for $|t|\leq \beta^{-1/2 + \epsilon_0}$,
    \item $b_1(t) = 0$ for $|t| \geq 2\beta^{-1/2 + \epsilon_0}$.
\end{itemize}
Let $b_2 = b_0 - b_1\in C_c^{\infty}(\BR)$.
Since $b_i$ are constant in some neighborhood of $0$,
we can extend the domains of $b_i$ to $\BH^3$ by the polar coordinates, i.e.,
\begin{align}\label{eq:def of bi}
    b_i(n(z)a(t)):=b_i(t_0),\text{ where }t_0\text{ satisfies the relation }n(z)a(t)\in K_0a(t_0)K_0,
\end{align}
so that $b_i \in C_c^\infty(\BH^3)$ are bi-$K_0$-invariant, and the restrictions $b_i\circ a(t) $ are the original functions, $i=0,1,2$.
Let $k_1, k_2\in C_c^\infty(\BH^3)$ be bi-$K_0$-invariant functions so that
$k_i = b_i k_\lambda$.
Then we decompose the integral $J(\lambda,\phi,e) = J_1(\phi) + J_2(\phi)$
where
\begin{align*}
    J_i(\phi)=\iint_{\BC\times \BR} \overline{b\phi(z_1,t_1})b\phi(z_2,t_2) k_i (n(e^{-t_1}(z_2-z_1))a(t_2-t_1))  e^{-2t_1 -2 t_2} dt_1dz_1 dt_2dz_2,\quad i=1,2.
\end{align*}
By the choice of the width of $\supp(b_1)$,
we will show that
for any $\phi \in \cS(\overline{\BD}\times\BR) $ satisfying $\|\phi \|_2 = 1$ and $\supp(\sF_t\phi) \subset \overline{\BD}\times (\BR\backslash \pm[\lambda-\beta,  \lambda+\beta])$, we have
\begin{align}\label{KN bound of I1}
    J_1(\phi) \ll_{\epsilon'} \beta^{-1/2 + \epsilon_0} ,
\end{align} 
and 
\begin{align}\label{KN bound of I2}
    J_2(\phi)\ll_{\epsilon',\epsilon_0,A}  \lambda^{-A} .
\end{align}
Hence, Proposition \ref{KN bound 1-Pi} will follow from (\ref{KN amplification inequality trivial case}), (\ref{KN bound of I1}) and (\ref{KN bound of I2}).

\subsubsection{Proof of \eqref{KN bound of I1}}
We define $\Phi \in C_c^\infty(\BH^3)$ by $\Phi(n(z)a(t)\cdot o):= b\phi(z,t)$.
It may be seen that $J_1(\phi) = \langle \Phi \times k_1, \Phi \rangle_{\BH^3}$.
Here $\langle \cdot, \cdot  \rangle_{\BH^3}$ is the inner product on $L^2(\BH^3)$ with respect to the hyperbolic metric. 
By the Plancherel formula, 
\begin{align*}
    J_1 (\phi) = \int_0^\infty\int_{K_0/M} \widehat{\Phi}(s,k)\widehat{k_1}(s)\overline{\widehat{\Phi} (s,k)} d\nu(s) dk\ll \| \widehat{k_1} \|_\infty \| \Phi\|_2^2 \ll \| \widehat{k_1} \|_\infty.
\end{align*}
We will prove the following pointwise bound for $\widehat{k_1}$,
and the bound (\ref{KN bound of I1}) will follow from this immediately.
\begin{proposition}
    If $|s|\leq \lambda/2$, then $\widehat{k_1}(s) \ll_{\epsilon',A} \lambda^{-A}$.
    
    If $|s|\geq \lambda/2$, then $\widehat{k_1}(s) \ll \beta^{-1/2 + \epsilon_0}$.
\end{proposition}

We first prove the case $|s| \geq \lambda/2$. We write the integral via the polar coordinates. There exists a constant $c > 0$ so that
\begin{align*}
    \widehat{k_1}(s)  &= \int_{\BH^3}  b_1k_\lambda(x) \varphi_{-s}(x) dx\\
    &= c \int_0^\infty b_1(t) k_\lambda(a(t)) \varphi_{-s}(a(t)) \sinh(t)^2 dt.
\end{align*}
The bounds $k_\lambda(a(t)) \ll \lambda^{2} (1+\lambda|t|)^{-1}$ from \cite[Lemma 2.8]{marshall2016p}, and $\varphi_{-s}(a(t)) \ll (1 + |st|)^{-1}$ from \cite[Theorem 1.3]{marshall2016p} with $|s| \geq \lambda/2$ imply that
\begin{align*}
    \widehat{k_1}(s) \ll \int_0^\infty b_1(t) (\lambda t^{-1})(\lambda t)^{-1} \sinh(t)^2 dt \ll \beta^{-1/2 + \epsilon_0}.
\end{align*}

Now we consider the case $|s| \leq \lambda/2$. We apply the inverse Harish-Chandra transform to $k_\lambda$, and apply the integral expansion of the spherical function. 
We obtain
\begin{align*}
    \widehat{k_1}(s) &= \int_0^\infty \left(\int_{\BH^3}  b_1(x) \varphi_r(x) \varphi_{-s}(x) dx \right)h_\lambda(r)  d\nu(r) \\
    &=\int_0^\infty \left(\int_{\BH^3} \int_{K_0} b_1(x) \varphi_r(x) \exp((1-is)A(kx)) dk dx\right)h_\lambda(r)  d\nu(r)\\
    &=\int_0^\infty \left(\int_{\BH^3}b_1(x) \varphi_r(x) \exp((1-is)A(x))dx \right)h_\lambda(r)  d\nu(r) \\
    &= \int_0^\infty \int_{K_0} \left(\int_{\BH^3} b_1(x) \exp((1+ir)A(kx)) \exp((1-is)A(x)) dx \right)h_\lambda(r) dk d\nu(r).
\end{align*}
It suffices to show for $|r-\lambda|\le\beta/4$,
\begin{align*}
    \int_{\BH^3} b_1(x) \exp((1+ir)A(kx)) \exp((1-is)A(x)) dx\ll_A r^{-A}.
\end{align*}
Combining all amplitude factors and applying the Iwasawa coordinates for $\BH^3$,
it suffices to show
\begin{align}\label{eq bd for J(k)}
    J(k) :=\int_{\BR^3} \chi_1(x,y,t) \exp(ir(A(kn(x+iy)a(t))-(s/r)t)) dx dy dt \ll_A r^{-A}.
\end{align}
Here $\chi_1(x,y,t) \in C_c^\infty(\BR^3)$ is a cutoff function at scale $\beta^{-1/2 + \epsilon_0}$, and the $k$-variable is ignored for notational simplicity.
We denote by $\rho = s/r$ and we can assume it to satisfy $|\rho| \leq 2/3$ by our assumptions on $s$ and $r$. 
The phase function with parameters $k\in K_0$ and $\rho$ is
\begin{align*}
    \phi(x,y,t;k,\rho) = A(kn(x+iy)a(t))-\rho t.
\end{align*}

\begin{lemma}\label{KN rescale nonstationary phase}
    If $\chi \in C_c^\infty(\BR^3)$ is a cutoff function at scale $\beta^{-1/2 + \epsilon_0}$, and $\phi \in C^\infty(\BR^3)$ is real-valued and satisfies
    \begin{align*}
        |\nabla\phi(x)| \geq c > 0, \;\; \phi^{(n)}(x)\ll_n 1
    \end{align*}
    for any $x\in \supp(\chi)$ and $n\geq 1$, then
    \begin{align*}
        \int  \chi(x) e^{ir\phi(x)} dx \ll_A r^{-A}.
    \end{align*}
    Here $c>0$ is a fixed constant.
\end{lemma}
\begin{proof}
    We have
    \begin{align*}
        \int \chi(x) e^{ir\phi(x)} dx = \beta^{-3/2 + 3\epsilon_0} \int \chi(\beta^{-1/2 + \epsilon_0} x) e^{i(r\beta^{-1/2 + \epsilon_0})(\beta^{1/2-\epsilon_0}\phi(\beta^{-1/2 + \epsilon_0} x))} dx.
    \end{align*}
    The cutoff function $\chi(\beta^{-1/2 + \epsilon_0} x)$ is now  at scale $1$.
    The lemma follows by integration by parts.
\end{proof}

Hence, to apply the above lemma to $J(k)$ for showing \eqref{eq bd for J(k)}, it suffices to give a nonzero lower bound for $$\nabla \phi = (\partial\phi/\partial x,\partial\phi/\partial y,\partial\phi/\partial t).$$ We first calculate the gradient of $A(kn(x+iy)a(t))$.
Let $u(x,y,t;k) = \begin{pmatrix}
    \alpha(x,y,t;k) &\beta(x,y,t;k) \\ -\bar{\beta}(x,y,t;k) & \bar{\alpha}(x,y,t;k)  
\end{pmatrix}\in K_0$ such that
\begin{align*}
    kn(x+iy)a(t) \in NAu(x,y,t;k).
\end{align*}
For simplicity, we will write $\alpha = \alpha(x,y,t;k) $ and $\beta = \beta(x,y,t;k) $. 
Then a straight calculation from (\ref{Poincare}) shows that
\begin{align*}
    \frac{\partial}{\partial x} A(kn(x+iy)a(t)) &= e^{-t}(\alpha\bar{\beta}+\bar{\alpha}\beta),\\
    \frac{\partial}{\partial y} A(kn(x+iy)a(t)) &= e^{-t}(\alpha\bar{\beta}-\bar{\alpha}\beta)i,\\
    \frac{\partial}{\partial t} A(kn(x+iy)a(t)) &= |\alpha|^2-|\beta|^2 .
\end{align*}
We have the following lower bound for $|\nabla \phi|$.

\begin{lemma}
    There exists a constant $c>0$.
    For any $(x,y,t)\in \supp(\chi_1)$, we have $|\nabla\phi(x,y,t)| \geq c$.
\end{lemma}
\begin{proof}
    By the compactness, it suffices to show that there do not exist $\alpha,\beta\in\BC$ so that $|\alpha|^2 + |\beta|^2 = 1$ and
    \begin{align}
        \alpha\bar{\beta}+\bar{\alpha}\beta = \alpha\bar{\beta}-\bar{\alpha}\beta = 0,\label{KN partial x y =0}\\
        |\alpha|^2 - |\beta|^2 = \rho\label{KN partial t = 0}.
    \end{align}
    Suppose otherwise and let $(\alpha_0,\beta_0)$ be such pair.
    \eqref{KN partial x y =0} implies that $\alpha_0 = 0$ or $\beta_0 = 0$.
    Then \eqref{KN partial t = 0} shows $\rho =  |\alpha_0|^2 - |\beta_0|^2 = \pm 1$,
    which gives the contradiction.
\end{proof}

\subsubsection{Derivatives of the distance function}\label{subsubsection: Derivatives of the distance function}
The distance function will appear as a part of the phase function in the oscillatory integral we shall study for proving \eqref{KN bound of I2}. We will estimate its derivatives.
We consider the projection $G_0 = K_0 \overline{A^+} K_0 \to \overline{A^+}$ by the Cartan decomposition.
We let
\begin{align*}
    \cA:\BC \times\BR \to [0,\infty)
\end{align*}
be the function defined by
\begin{align*}
    n(z) a(t) \in K_0 a(\cA(z,t)) K_0.
\end{align*}
The function $\cA(z,t)$ is equal to the distance from the origin $o$ to $n(z)a(t)\cdot o$ under the hyperbolic metric, and is given by
    \begin{align}\label{KN definition of mathcal A, the distance function in KAK}
        \cosh(\cA(z,t)) =\frac{1}{2}e^{t} +\frac{1}{2}(|z|^2+1)e^{-t}  =\cosh(t)+\frac{1}{2}|z|^2e^{-t}.
    \end{align}

\begin{lemma}\label{KN lower bound for t from lower bound for A(z,t)}
    Suppose $|z|\ll \lambda^{-1/2}$ and $\cA(z,t)\gg \beta^{-1/2 + \epsilon_0}$.
    Then we have $|t|\gg \beta^{-1/2 + \epsilon_0}$.
\end{lemma}
\begin{proof}
     Since $|z| \ll \lambda^{-1/2}$, by \eqref{KN definition of mathcal A, the distance function in KAK},
     \begin{align*}
         \cosh( \mathcal{A}(z,t)) =\cosh(t)+ O( \lambda^{-1}).
     \end{align*}
     The lemma follows from the above with the Taylor approximation for $\cosh$.
\end{proof}

\begin{proposition}
     Suppose $|z|\ll \lambda^{-1/2}$, $t\ll1$, and $\cA(z,t)\gg \beta^{-1/2 + \epsilon_0}$. Then we have
     \begin{align}
         &\frac{\partial}{\partial t}\cA(z,t)=\sgn t+O(\lambda^{-1}\beta^{1-2\epsilon_0}) , \label{KN approximate for derivative of mathcal A}\\
       & \frac{\partial^n}{\partial t^n} \cA(z,t) \ll_n \lambda^{-1}(\beta^{1/2-\epsilon_0})^{n+1},\qquad\text{ for any integer }n\geq2.\label{KN approximate for higher derivative of mathcal A}
     \end{align}
\end{proposition}
\begin{proof}
    By applying $\partial/\partial t$ to \eqref{KN definition of mathcal A, the distance function in KAK}, we get
    \begin{align*}
        \sinh(\cA(z,t))\frac{\partial}{\partial t}\cA(z,t) = \sinh(t)-\frac{1}{2}|z|^2e^{-t}.
    \end{align*}
    Note that
    \begin{align*}
        \sinh(\cA(z,t))=\sqrt{\left(\cosh(t)+\frac{1}{2}|z|^2e^{-t}\right)^2-1}=\sqrt{\sinh^2(t)+O(\lambda^{-1})}=|\sinh(t)|(1+O(\lambda^{-1}\sinh^{-2}(t))),
    \end{align*}
    so
    \begin{align*}
        \frac{\partial}{\partial t}\cA(z,t)=\sgn\left((\sinh(t)\right)+O(\lambda^{-1}\sinh^{-2}(t))=\sgn t+O(\lambda^{-1}\beta^{1-2\epsilon_0}),
    \end{align*}
    where the last step is due to Lemma \ref{KN lower bound for t from lower bound for A(z,t)}.

    To bound the higher derivatives, we note that
    \begin{align*}
        \frac{\partial}{\partial t}\mathcal{A}(z,t)=\frac{\sinh(t)-\frac{1}{2}|z|^2e^{-t}}{\sqrt{\left(\cosh(t)+\frac{1}{2}|z|^2e^{-t}\right)^2-1}}
    \end{align*}
    and for any $w\in\BC$ with $|w|=O(1)$ the real part inside the square root is
    \begin{multline*}
        \Re\left(\left(\cosh(w)+\frac{1}{2}|z|^2e^{-w}\right)^2-1\right)=\Re\left(\sinh^2(w)+O(\lambda^{-1})\right)=\frac{\cosh(2\Re(w))\cos(2\Im (w))-1}{2}+O(\lambda^{-1})\\
    = \sinh^2(\Re(w)) +O(\Im(w)^2)+O(\lambda^{-1}).
    \end{multline*}
    For $t\ll1$ given in the proposition, by Lemma \ref{KN lower bound for t from lower bound for A(z,t)}, we have $|t|\gg \beta^{-1/2+\epsilon_0}$. Therefore, for $C>0$ and for any $w\in\BC$ with $|w-t|\leq C \beta^{-1/2+\epsilon_0}$, we have
    \[
    \Re\left(\left(\cosh(w)+\frac{1}{2}|z|^2e^{-w}\right)^2-1\right)=\sinh^2(\Re(w))+O((C\beta^{-1/2+\epsilon_0})^2)+O(\lambda^{-1}) \gg\beta^{-1+2\epsilon_0},
    \]
    which holds when $C>0$ is chosen to be sufficiently small. Thus, 
    \[\frac{\partial}{\partial w}\mathcal{A}(z,w)=\frac{\sinh(w)-\frac{1}{2}|z|^2e^{-w}}{\sqrt{\left(\cosh(w)+\frac{1}{2}|z|^2e^{-w}\right)^2-1}}\]
    can be extended to be holomorphic (in $w$) on the ball of radius $C\beta^{-1/2+\epsilon_0}$ centered at $t$. By applying Cauchy's integral formula, with the bound 
    \(\frac{\partial}{\partial w}\cA(z,w)-\sgn t=O(\lambda^{-1}\beta^{1-2\epsilon_0})\) for $|w-t|\leq C \beta^{-1/2+\epsilon_0}$, we obtain, for any $n\geq1$,
    \begin{align*}
        \frac{\partial^{n+1}}{\partial t^{n+1}} \cA(z,t)=\frac{\partial^{n}}{\partial t^{n}} \left(\frac{\partial}{\partial t}\cA(z,t)-\sgn t\right) \ll_n \lambda^{-1}\beta^{1-2\epsilon_0}(C\beta^{-1/2+\epsilon_0})^{-n},
    \end{align*}
    which proves \eqref{KN approximate for higher derivative of mathcal A}.
\end{proof}

\subsubsection{Proof of \eqref{KN bound of I2}}

We first unfold $J_2(\phi)$ by the inverse Harish-Chandra transform for $h_\lambda$ to get
\begin{align*}
     J_2(\phi) =\int_0^\infty J_2(\phi,s) h_\lambda(s)d\nu(s),
\end{align*}
with
\begin{align*}
    J_2(\phi,s)=\iint_{\BC\times\BR} \overline{b\phi(z_1,t_1)}b\phi(z_2,t_2)((b_2\varphi_s)\circ a)(\cA(e^{-t_1}(z_2-z_1),t_2-t_1)e^{-2t_1 -2 t_2} dz_1 dz_2 dt_1 dt_2.
\end{align*}
Since $\beta\geq\lambda^{\epsilon'}$ and $h_\lambda(s)$ rapidly decays away from $\lambda$, it suffices to prove the bound
\begin{align}\label{eq:equa bd for J2}
    J_2(\phi,s)\ll_{A,\epsilon_0} \beta^{-A}
\end{align}
for $|s-\lambda|\le\beta/2$.
We apply asymptotics for the spherical function from \cite[Theorem 1.5]{marshall2016p}.
There are functions $f_{\pm} \in C^\infty ((0,3)\times \BR) $ such that
\begin{align}\label{KN bound for f pm}
    \frac{\partial^n}{\partial t^n} f_\pm(t,s) \ll_n t^{-n}(st)^{-1}
\end{align}
and
\begin{align}\label{KN asymtotic for spherical function}
    \varphi_s (a(t)) = f_+(t,s)e^{ist}  + f_-(t,s)e^{-ist}  + O_A((st)^{-A})
\end{align}
for $t\in(0,3)$. Plug the formula \eqref{KN asymtotic for spherical function} into the integral $J_2(\phi,s)$. Because of the support of $b_2$, the error term in \eqref{KN asymtotic for spherical function} contributes $O_A(s^{-A})$, which may be ignored. We shall only consider the integral involving $f_+$, as the integral involving $f_-$ is similar.
By expanding one $\phi$ via the Fourier inversion and the changes of variables $z_2\mapsto z_2+z_1, t_2\mapsto t_2+t_1$, the integral we are considering is equal to
\begin{align}
    \frac{1}{2\pi}\iint_\BC\iint_\BR&\overline{b\phi(z_1,t_1)}\sF_t\phi(z_2+z_1,s_2)e^{is_2t_1-4t_1}\notag\\
    &\left(\int_0^\infty \chi(t_2)b_2(\cA(e^{-t_1}z_2,t_2)) f_+(\cA(e^{-t_1}z_2,t_2),s)\exp(is\cA(e^{-t_1}z_2,t_2)+is_2t_2)dt_2\right.\label{KN main inner integral away from spec}\\
    +&\left.  \int_{-\infty}^0 \chi(t_2)b_2(\cA(e^{-t_1}z_2,t_2)) f_+(\cA(e^{-t_1}z_2,t_2),s)\exp(is\cA(e^{-t_1}z_2,t_2)+is_2t_2)dt_2 \right)dz_1 dz_2 dt_1 ds_2. \label{KN main inner integral away from spec 2}
\end{align}
Here we denote by $\chi(t_2)=b(z_2+z_1,t_2+t_1)e^{-2 t_2}$ and omit the variables $z_1,z_2,t_1$ as the estimates for the derivatives of $\chi$ with respect to $t_2$ are uniform on $z_1,z_2,t_1$. If we replace $t_2$ with $\beta^{-1/2+\epsilon_0}t_2$, then we write the inner integral \eqref{KN main inner integral away from spec} in the form $\int_0^\infty \alpha(t_2) e^{i|s+s_2|^{1/2+\epsilon_0} f(t_2)} dt_2$ with
\begin{align*}
    &\alpha(t_2)= \beta^{-1/2+\epsilon_0}\chi(\beta^{-1/2+\epsilon_0}t_2)b_2(\cA(e^{-t_1}z_2,\beta^{-1/2+\epsilon_0}t_2)) f_+(\cA(e^{-t_1}z_2,\beta^{-1/2+\epsilon_0}t_2),s),\\
    &f(t_2)=\frac{s}{|s+s_2|^{1/2+\epsilon_0}}\cA(e^{-t_1}z_2,\beta^{-1/2+\epsilon_0}t_2)+\frac{s_2}{|s+s_2|^{1/2+\epsilon_0}}\beta^{-1/2+\epsilon_0}t_2.
\end{align*}
By \eqref{KN approximate for derivative of mathcal A} and the Fourier support assumption on $\phi$: $|s_2\pm \lambda|\geq\beta$, we have the following estimate for the derivative of the phase function $f(t_2)$:
    \begin{align*}
        \left|\frac{\partial f}{\partial t_2}\right| = |s+s_2|^{-1/2-\epsilon_0}\beta^{-1/2+\epsilon_0}\left|s(1 + O(\lambda^{-1}\beta^{1- 2\epsilon_0}))+s_2\right| = \left(\frac{|s+s_2|}{\beta} \right)^{1/2-\epsilon_0}+ O(\beta^{-2\epsilon_0})\gg1.
    \end{align*}
For any $n\geq2$, \eqref{KN approximate for higher derivative of mathcal A} implies that,
    \begin{align*}
        \frac{\partial^nf}{\partial t_2^n}&\ll_n \frac{s}{|s+s_2|^{1/2+\epsilon_0}} \lambda^{-1}(\beta^{1/2-\epsilon_0})^{n+1}(\beta^{-1/2+\epsilon_0})^n\ll \beta^{-2\epsilon_0}\ll1.
    \end{align*}
By \eqref{KN bound for f pm}, if $t\gg \beta^{-1/2+\epsilon_0}$ and $n\geq0$, then
\begin{align*}
    \frac{\partial^n}{\partial t^n} b_2(t)f_+(t,s) \ll_n \lambda^{-1}\beta^{1/2-\epsilon_0}(\beta^{1/2-\epsilon_0})^n=\lambda^{-1} (\beta^{1/2-\epsilon_0})^{n+1} .
\end{align*}
Moreover, by \eqref{KN approximate for derivative of mathcal A} and \eqref{KN approximate for higher derivative of mathcal A}, for any $n\geq1$, it may be seen that
\begin{align*}
    \frac{\partial^n}{\partial t_2^n}\cA(e^{-t_1}z_2,\beta^{-1/2+\epsilon_0}t_2)\ll (1+\lambda^{-1}(\beta^{1/2-\epsilon_0})^{n+1})(\beta^{-1/2+\epsilon_0})^{n}\ll \beta^{-1/2+\epsilon_0}\ll1.
\end{align*}
Therefore, by the chain rule to higher derivatives, for any $n\geq0$ we obtain that the $n$-th derivative of the amplitude function $\alpha(t_2)$ is bounded by
    \begin{align*}
        \frac{\partial^n\alpha}{\partial t_2^n}\ll_n\beta^{-1/2+\epsilon_0}\lambda^{-1}(\beta^{1/2-\epsilon_0})^{n+1}=\lambda^{-1}(\beta^{1/2-\epsilon_0})^{n}.
    \end{align*}
Integrating by parts $n$ times, or applying \cite[(4) p.331]{Stein}, gives that the inner integral \eqref{KN main inner integral away from spec}, which can be written as the form of the oscillatory integral $\int_0^\infty \alpha(t_2) e^{i|s+s_2|^{1/2+\epsilon_0} f(t_2)} dt_2$, is bounded by
\begin{align*}
    |s+s_2|^{-n(1/2+\epsilon_0)}\int_0^\infty\left|\left(\frac{\partial}{\partial t_2}\frac{1}{\partial f/\partial t_2}\right)^n\alpha\right| dt_2\ll_n |s+s_2|^{-n(1/2+\epsilon_0)}\lambda^{-1}(\beta^{1/2-\epsilon_0})^{n}\beta^{1/2-\epsilon_0}\ll |s+s_2|^{-n\epsilon_0}.
\end{align*}
In summary, the inner integral \eqref{KN main inner integral away from spec} is $\ll_{A,\epsilon_0}|s+s_2|^{-A}$. The second inner integral \eqref{KN main inner integral away from spec 2} can be bounded in the same way by replacing $s+s_2$ with $s-s_2$ and is $\ll_{A,\epsilon_0}|s-s_2|^{-A}$. Hence, with these estimates and the fact $|s\pm s_2|\gg\beta$, by applying the Cauchy-Schwarz inequality and the Plancherel theorem, we obtain \eqref{eq:equa bd for J2}.

\subsection{Bounds near the spectrum}\label{KN section of near spectrum}
We want to prove Proposition \ref{KN bound Pi} via the method of arithmetic amplification.
As $\psi$ and $b$ are real-valued,
we can assume $\phi \in \cS(\overline{\BD}\times\BR)$ satisfies  $\supp(\sF_t\phi) \subset\BC\times [\lambda-\beta,\lambda+\beta]$.
We shall need the following bounds for $J(\lambda,\phi,g)$.

\begin{proposition}\label{KN bound for I(lambda,phi,g) near spectrum}
     We suppose that $\lambda^{\epsilon'}\leq\beta\leq\lambda^{1/4}$, $\phi \in \cS(\overline{\BD}\times\BR)$ satisfies $\| \phi \|_2 =1$ and $\supp(\sF_t\phi) \subset\BC\times [\lambda-\beta,\lambda+\beta]$.
     Let $g\in G_0$ satisfying $d(g,e)\leq 1$.
     If $1/8>\epsilon_0 > 0$ and $d(g, MA) \geq \lambda^{-1/2 + \epsilon_0}\beta^{1/2}$, then $$J(\lambda,\phi,g) \ll_{\epsilon',\epsilon_0,A} \lambda^{-A}.$$
\end{proposition}

\begin{proof}
    We can unfold the integral $J(\lambda,\phi,g)$ by the Fourier inversion for $\phi$ and inverse Harish-Chandra transform for $h_\lambda$, that is,
    \begin{align*}
        J(\lambda,\phi,g) =\int_0^\infty \iint_{-\infty}^{\infty} \iint_\BC \overline{\sF_t\phi(z_1,s_1})\sF_t\phi(z_2,s_2) J(s,s_1,s_2,z_1,z_2) h_\lambda(s) dz_1 dz_2 ds_1 ds_2 d\nu(s).
    \end{align*}
    The integral $J(s,s_1,s_2,z_1,z_2)$ is defined by
    \begin{align*}
        J(s,s_1,s_2,z_1,z_2) = \iint_{-\infty}^\infty \chi(z_1,z_2,t_1,t_2)  \exp(-is_1 t_1 + is_2 t_2) \varphi_s (a(-t_1)n(-z_1)gn(z_2)a(t_2))  dt_1 dt_2,
    \end{align*}
    with $s_1,s_2 \in [\lambda-\beta,\lambda+\beta]$,
    and $\chi\in C_c^\infty(\overline{\BD} \times \overline{\BD} \times \BR \times \BR)$ defined by
    \begin{align*}
        \chi(z_1,z_2,t_1,t_2) = (2\pi)^{-2}b(z_1,t_1)b(z_2,t_2)e^{-2t_1 -2 t_2} .
    \end{align*}
    By the rapid decay of $h_\lambda$ away from $\lambda$, we may also assume $s\in [\lambda-\beta,\lambda+\beta]$.
    From the fact that $|z_1|,|z_2| \leq \lambda^{-1/2}$, it may be seen that
    \begin{align*}
        d(n(-z_1)gn(z_2), MA) \gg \lambda^{-1/2 + \epsilon_0}\beta^{1/2}.
    \end{align*}
    Then the proposition follows from Proposition \ref{nonstationary oscillatory integral estimate} applied to $\rho_1 = \rho_2 = 0$.
\end{proof}

\begin{proof}[Proof of Proposition \ref{KN bound Pi}]
    We choose our amplifier $T_v$ for $v\in\sQ$ to be the same as \eqref{amplifier inert}.
    Let $N\geq 1$ be a parameter to be chosen later. We define $\sQ_N = \{ v\in\sQ | N/2\leq q_v\leq N\}$, and define
    $
        \cT_N = \sum_{v\in\sQ_N}  T_v.
    $
    Then we have
    \begin{align*}
        |\langle\cT_N\psi,b\phi\rangle| = \#\sQ_N \cdot |\langle\psi,b\phi \rangle| \gg N^{1-\epsilon} |\langle\psi,b\phi\rangle|.
    \end{align*}
    By Proposition \ref{KN amplification inequality for geodesic tube}, it follows that
    \begin{align}\label{KN amplified inequality 1}
         |\langle\psi,b\phi\rangle|^2 \ll N^{-2+\epsilon} \sum_{\substack{\gamma\in \bG(F)\\ \gamma\in g_0 \cD_0 g_0^{-1}}} \left| (\cT_N*\cT_N^*)(\gamma)  J(\lambda,\phi,g_0^{-1}\gamma g_0) \right|.
    \end{align}
    Recall that we use $\cD_0\subset G_{0}$ to denote the compact subset consisting of $g\in G_0$ with $d(g,e)\leq 1$.
    Let $\epsilon_0 > 0$.
    Proposition \ref{KN bound for I(lambda,phi,g) near spectrum} implies that we only need to consider the terms in (\ref{KN amplified inequality 1}) with
    \begin{align*}
        d\left(g_0^{-1}\gamma g_0, MA\right) < \lambda^{-1/2 + \epsilon_0}\beta^{1/2}.
    \end{align*}
    Let $C>0$ be the constant as in Proposition \ref{Hecke return wrt MA}.
    We choose 
    \begin{align*}
        X= (C^{-1}\lambda^{-1/2 + \epsilon_0}\beta^{1/2})^{-1}  =  C\lambda^{1/2 - \epsilon_0}\beta^{-1/2},
    \end{align*}
    so then take $\delta = CX^{-1} = \lambda^{-1/2 + \epsilon_0}\beta^{1/2}$.
    We use the trivial bound 
    $J(\lambda,\phi,g_0^{-1}\gamma g_0) \ll 1$ to obtain
    \begin{align}\label{KN ine 3}
        \sum_{\substack{\gamma\in \bG(F)\\ \gamma\in g_0 \cD_0 g_0^{-1}\\d(g_0^{ -1}\gamma g_0,MA)< \delta}}  \left| (\cT_N*\cT_N^*)(\gamma)J(\lambda,\phi,g_0^{-1}\gamma g_0)\right| \ll \sum_{\substack{\gamma\in \bG(F)\\ \gamma\in g_0 \cD_0 g_0^{-1}\\d(g_0^{-1}\gamma g_0,MA)< \delta}} \left| (\cT_N*\cT_N^*)(\gamma)\right| .
    \end{align}
    We next expand $\cT_N*\cT_N^*$ as a sum
    \begin{align*}
        \cT_N*\cT_N^* = \sum_{\fn\subset\cO} a_\fn 1_{K(\fn)}
    \end{align*}
    for some constants $a_\fn$.
    If we now suppose $N^4\leq X$, that is
    \begin{align*}
        N \leq C^{1/4} \lambda^{1/8 - \epsilon_0/4}\beta^{-1/8},
    \end{align*}
    then our choice of $X$ and $\delta$ means that we may apply Proposition \ref{Hecke return wrt MA}  to show that for any $\fn\neq \cO$ appearing in the expansion of $\cT_N*\cT_N^*$, the term $1_{K(\fn)}$ makes a contribution $\ll_\epsilon X^\epsilon \ll \lambda^\epsilon$ to the sum in the right hand side of (\ref{KN ine 3}), 
    and the sum of $a_\fn$'s with $\fn \neq \cO$ is $\ll 1$ by \eqref{boundedness for sum of primes}.
    Therefore, we only need to consider the term $\fn=\cO$. Since $a_\cO \ll N$ by \eqref{eq: bd for aO}, we have
    \begin{align*}
          \sum_{\substack{\gamma\in \bG(F)\\ \gamma\in g_0 \cD_0 g_0^{-1}\\d(g_0^{ -1}\gamma g_0,MA)< \delta}}  \left| (\cT_N*\cT_N^*)(\gamma)J(\lambda,\phi,g_0^{-1}\gamma g_0)\right| \ll N + \lambda^\epsilon.
    \end{align*}
    We obtain
    \begin{align*}
       |\langle\psi,b\phi\rangle|^2 \ll N^{-2+\epsilon} (N + \lambda^\epsilon).
    \end{align*}
    Choosing $N =C^{1/4} \lambda^{1/8 - \epsilon_0/4}\beta^{-1/8}$ gives  the result.
    \end{proof}


\bibliographystyle{plain} 
\bibliography{refs.bib} 

\end{document}